\documentclass[reqno,12pt]{amsart} 
\usepackage{amssymb} 
 
\theoremstyle{plain} 
\newtheorem{theorem}{\indent\sc Theorem}[section] 
\newtheorem{lemma}[theorem]{\indent\sc Lemma}
\newtheorem{corollary}[theorem]{\indent\sc Corollary}
\newtheorem{proposition}[theorem]{\indent\sc Proposition}
\newtheorem{claim}[theorem]{\indent\sc Claim}
\theoremstyle{definition} 
\newtheorem{definition}[theorem]{\indent\sc Definition}
\newtheorem{remark}[theorem]{\indent\sc Remark}
\newtheorem{example}[theorem]{\indent\sc Example}

\begin{document}

\title[Ricci curvature]{Spectral convergence under bounded Ricci curvature} 

\author[Shouhei Honda]{Shouhei Honda} 



\address{ 
Mathematical Institute \endgraf
Tohoku University \endgraf 
Sendai 980-8578 \endgraf  
Japan
}
\email{shonda@m.tohoku-u.ac.jp}

\maketitle
\begin{abstract}
For a noncollapsed Gromov-Hausdorff convergent sequence of Riemannian manifolds with a uniform bound of Ricci curvature,
we establish two spectral convergence.
One of them is on the Hodge Laplacian acting on differential one-forms.
The other is on the connection Laplacian acting on tensor fields of every type, which include all differential forms.
These are sharp generalizations of Cheeger-Colding's spectral convergence of the Laplacian acting on functions to the cases of tensor fields and differential forms.
These spectral convergence have two direct corollaries.
One of them is to give new bounds on such eigenvalues, in terms of bounds on volume, diameter and the Ricci curvature.
The other is that we show the upper semicontinuity of the first Betti numbers with respect to the Gromov-Hausdorff topology, and give the equivalence between the continuity of them and the existence of a uniform spectral gap.
On the other hand we also define measurable curvature tensors of the noncollapsed Gromov-Hausdorff limit space of a sequence of  Riemannian manifolds with a uniform bound of Ricci curvature, which include Riemannian curvature tensor, the Ricci curvature, and the scalar curvature.
As fundamental properties of our Ricci curvature, we show that the Ricci curvature coincides with the difference between the Hodge Laplacian and the connection Laplacian, and is compatible with Gigli's one and Lott's Ricci measure. 
Moreover we prove a lower bound of the Ricci curvature is compatible with a reduced Riemannian curvature dimension condition.
We also give a positive answer to Lott's question on the behavior of the scalar curvature with respect to the Gromov-Hausdorff topology by using our scalar curvature.
\end{abstract}
\tableofcontents
\section{Introduction}
In this paper we study the spectral behavior of the Hodge Laplacian and the connection Laplacian with respect to the Gromov-Hausdorff topology.
The spectral behavior of the Laplacian $\Delta$ acting on functions was first studied in \cite{fu} by Fukaya.
He gave the definition of measured Gromov-Hausdorff convergence and proved the continuity of $k$-th eigenvalues $\lambda_k$ of $\Delta$:
\begin{align}\label{nn9}
\lambda_k(X_i) \to \lambda_k(X)
\end{align}
as $i \to \infty$ for a measured Gromov-Hausdorff convergent sequence 
\[\left(X_i, \frac{H^n}{H^n(X_i)}\right) \stackrel{GH}{\to} (X, \upsilon)\]
under a uniform  bound of the sectional curvature:
\begin{align}\label{sectionalc}
|\mathrm{Sec}_{X_i}|\le C <\infty,
\end{align}
where $X_i$ is an $n$-dimensional closed Riemannian manifold, $H^n$ is the $n$-dimensional Hausdorff measure, and $(X, \upsilon)$ is a compact metric measure space. 
He also conjectured that the spectral convergence (\ref{nn9}) holds even if we replace the assumption (\ref{sectionalc}) by a uniform lower bound of the Ricci curvature:
\[\mathrm{Ric}_{X_i} \ge C'>-\infty.\]
This conjecture was proved in \cite{ch-co3} by Cheeger-Colding and is extended to the case of $RCD$-spaces in \cite{gms} by Gigli-Mondino-Savar\'e. 
Moreover similar spectral convergence of other differential operators acting on functions, which include the Schr\"odinger operator, the $p$-Laplacian, and the $\overline{\partial}$-Laplacian, are also known.
See \cite{FHS, hoch, hoell}.

It is worth pointing out that the spectral convergence (\ref{nn9}) holds even if the sequence is collapsed (i.e. $\liminf_{i \to \infty}\mathrm{dim}\,X_i>\mathrm{dim}\,X$ holds).
However for a collapsed sequence the spectral convergence of the Hodge Laplacian does not hold in general.
We give an example.

Let us consider the following setting which is a collapsing of normalized flat tori:
\[\left( \mathbf{S}^1(1) \times \mathbf{S}^1(r), \frac{H^2}{4\pi^2r}\right) \stackrel{GH}{\to} \left(\mathbf{S}^1(1), \frac{H^1}{2\pi}\right) \]
as $r \to 0$, where $\mathbf{S}^1(r)$ denotes the circle of the length $2\pi r$.
Since Hodge theory states that the first Betti number $b_1$ coincides with the dimension of the space of harmonic $1$-forms, we see that $b_1(\mathbf{S}^1(1) \times \mathbf{S}^1(r))=2$ yields $\lambda_2^{H, 1}(\mathbf{S}^1(1) \times \mathbf{S}^1(r))=0$ and that $b_1(\mathbf{S}^1(1))=1$ yields $\lambda_2^{H, 1}(\mathbf{S}^1(1))>0$,
where $\lambda_k^{H, 1}$ denote the $k$-th eigenvalue, counted with multiplicities,  of the Hodge Laplacian $\Delta_{H, 1}$ acting on differential one-forms.
In particular 
\[\lim_{r \to 0}\lambda_2^{H, 1}\left(\mathbf{S}^1(1) \times \mathbf{S}^1(r)\right)<\lambda_2^{H, 1}\left(\mathbf{S}^1(1)\right).\]
Note that since the spaces above are flat, Bochner's formula yields that the Hodge Laplacian coincides with the connection Laplacian:
\[\Delta_{C, 1}:=\nabla^*\nabla,\]
where $\nabla$ is the covariant derivative and $\nabla^*$ is the adjoint operator.
This example tells us that the noncollapsing assumption is essential in order to establish the spectral convergence of the Hodge Laplacian. 
See for instance \cite{lott, lott2} by Lott for studies of the spectral convergence of differential forms in this direction.

Our setting in this paper is the following.

Let $n \in \mathbf{N}$, $d, v \in (0, \infty)$, $K_1, K_2 \in \mathbf{R}$,
let $\mathcal{M}:=\mathcal{M}(n, d, K_1, K_2, v)$ be the set of (isometry classes of) $n$-dimensional closed Riemannian manifolds $M$ with $\mathrm{diam}\,M \le d, H^n(M) \ge v$ and
\begin{align}
K_1 \le \mathrm{Ric}_M \le K_2,
\end{align}
where $\mathrm{diam}\,M$ denotes the diameter of $M$, and let $\overline{\mathcal{M}}:=\overline{\mathcal{M}(n, d, K_1, K_2, v)}$ be the set of Gromov-Hausdorff limits of sequences in  $\mathcal{M}$. Gromov's compactness theorem in \cite{gr} states that $\overline{\mathcal{M}}$ is compact with respect to the Gromov-Hausdorff topology.
Moreover Cheeger-Colding's volume convergence theorem in \cite{ch-co1} (see also \cite{Colding}) with Bishop's inequality yields the following.
\begin{itemize} 
\item For every $X \in \overline{\mathcal{M}}$, 
\[v \le H^n(X) \le C(n, K_1, d)<\infty,\]
where $C(n, K_1, d)$ is a positive constant depending only on $n, K_1, d$.
\item If $X_i$ Gromov-Hausdorff converges to $X$ (we denote it by $X_i \stackrel{GH}{\to} X$ for short) in $\overline{\mathcal{M}}$, then $(X_i, H^n) \stackrel{GH}{\to} (X, H^n)$ in the measured Gromov-Hausdorff sense.
Thus we always consider $H^n$ as a standard measure on $X$. 
\end{itemize}

A main result of this paper is the following.
\begin{theorem}[Spectral convergence of $\Delta_{H, 1}$ and $\Delta_{C, 1}$]\label{spectr}
Let $X \in \overline{\mathcal{M}}$. We have the following.
\begin{enumerate}
\item{(Ricci curvature)}  There exists a canonical symmetric $L^{\infty}$- tensor 
\[\mathrm{Ric}_X \in L^{\infty}(T^*X \otimes T^*X),\]
called Ricci curvature,  with
\[K_1 \le \mathrm{Ric}_X \le K_2\]
on $X$ a.e. sense (we give the precise meaning of `canonical' in subsection 1.2). 
\item{(Weitzenb\"ock formula)}
Let us denote by $\mathcal{D}^2(\Delta_{H, 1}, X)$ and $\mathcal{D}^2(\Delta_{C, 1}, X)$ the domains of the Hodge Laplacian $\Delta_{H, 1}$ and the connection Laplacian $\Delta_{C, 1}$ acting on differential one-forms on $X$ as an $RCD(K_1, \infty)$-space defined in \cite{gigli}  by Gigli, respectively.
Then we have $\mathcal{D}^2(\Delta_{H, 1}, X)=\mathcal{D}^2(\Delta_{C, 1}, X)$.
Moreover
\begin{align}\label{chara}
\Delta_{H, 1}\omega =\Delta_{C, 1}\omega +\mathrm{Ric}_X(\omega^*, \cdot)
\end{align}
for every $\omega \in \mathcal{D}^2(\Delta_{H, 1}, X)=\mathcal{D}^2(\Delta_{C, 1}, X)$,
where the $L^2$-differential one-form $\mathrm{Ric}_X(\omega^*, \cdot ) \in L^2(T^*X)$ is defined by satisfying
\[\int_X\langle \mathrm{Ric}_X(\omega^*, \cdot ), \eta \rangle dH^n=\int_X\langle \mathrm{Ric}_X, \omega \otimes \eta \rangle dH^n\]
for every $L^2$-differential one-form $\eta \in L^2(T^*X)$.
\item{(Finite dimensionalities of eigenspaces)} For every $\alpha \ge 0$, the sets
\[E^{H, 1}_{\alpha}:=\left\{\omega \in \mathcal{D}^2(\Delta_{H, 1}, X); \Delta_{H, 1}\omega =\alpha \omega \right\},\]
\[E^{C, 1}_{\alpha}:=\left\{\omega \in \mathcal{D}^2(\Delta_{C, 1}, X); \Delta_{C, 1}\omega =\alpha \omega \right\} \]
are finite dimensional.
\item{(Discreteness and unboundedness of spectrums)} Let us consider the spectrums of $\Delta_{H, 1}$ and $\Delta_{C, 1}$ as follows:
\[S^{H, 1}:=\left\{ \lambda \in \mathbf{R}; E^{H, 1}_{\lambda} \neq \{0\} \right\},\]
\[S^{C, 1}:=\left\{ \lambda \in \mathbf{R}; E^{C, 1}_{\lambda} \neq \{0\}  \right\},\]
respectively.
Then these are nonnegetive, discrete, and unbounded, where we say that a subset $A$ of $\mathbf{R}$ is discrete if there is no $a \in \mathbf{R}$ such that 
\[\left(B_r(a) \setminus \{a\}\right) \cap A \neq \emptyset\]
holds for every $r>0$.
\item{(Spectral convergence)} Let us denote the spectrums of $\Delta_{H, 1}$ and $\Delta_{C, 1}$ by
\[0\le \lambda_1^{H, 1} \le \lambda_2^{H, 1} \le \cdots \to \infty\]
and
\[0\le \lambda_1^{C, 1}\le \lambda_2^{C, 1} \le \cdots \to \infty\]
counted with multiplicities, respectively.
If $X_i \stackrel{GH}{\to} X$ in $\overline{\mathcal{M}}$, then 
\[\lambda_k^{H, 1}(X_i) \to \lambda_k^{H, 1}(X)\]
and
\[\lambda_k^{C, 1}(X_i) \to \lambda_k^{C, 1}(X)\]
hold for every $k \ge 1$.
\end{enumerate}
\end{theorem}
Note that the equality (\ref{chara}) is well-known on a  smooth closed Riemannian manifold as Bochner's formula.

An example by Page in \cite{pa} and Kobayashi-Todorov in \cite{kt} gives a noncollapsed sequence of Calabi-Yau manifolds of dimension $4$, whose second Betti numbers are greater than that of the limit space (see Example \ref{pakt}).
Hein-Naber also give in their forthcoming paper \cite{hn} an example of a noncollapsed sequence of Calabi-Yau manifolds of dimension $6$ with unbounded Betti numbers.
These examples tell us that in general for every $k \ge 2$ we can not expect the spectral convergence of the Hodge Laplacian $\Delta_{H, k}$ acting on differential $k$--forms even if we consider Ricci flat manifolds because
the spectral convergence implies the finiteness of the $k$-th Betti number of the limit space and the upper semicontinuity of them (see Theorem \ref{bettibetti}). 
Thus the spectral convergence of the Hodge Laplacian in Theorem \ref{spectr} is sharp in this sense. 

On the other hand for every $X \in \overline{\mathcal{M}}$ we can define the connection Laplacian $\Delta_{C, k}$ acting on differential $k$-forms, and the connection Laplacian $\Delta_{C, (r, s)}$ acting on tensor fields of type $(r, s)$.
For the connection Laplacian, the spectral convergence always hold, i.e. we have the following. 
\begin{theorem}[Spectral convergence of the connection Laplacian]\label{wll}
The spectral convergence of $\Delta_{C, k}$ and $\Delta_{C, (r, s)}$ hold, i.e. similar statements as in Theorem \ref{spectr} hold for $\Delta_{C, k}$ and $\Delta_{C, (r, s)}$.
\end{theorem} 
We now introduce three applications of Theorems \ref{spectr} and \ref{wll}.
One of them is to give uniform bounds on eigenvalues of these Laplacian, which are new even if X is smooth.
\begin{theorem}[Uniform bounds on eigenvalues]\label{estim}
We have
\[0 \le C_1(n, d, v, K_1, K_2, l)\le \lambda_l^{H, 1}(X) \le C_2(n, d, v, K_1, K_2, l)<\infty,\]
\[0 \le C_1(n, d, v, K_1, K_2, l)\le \lambda_l^{C, 1}(X) \le C_2(n, d, v, K_1, K_2, l)<\infty,\]
\[0\le C_1(n, d, v, K_1, K_2, k, l)\le \lambda_l^{C, k}(X) \le C_2(n, d, v, K_1, K_2, k, l)<\infty \]
and
\[0\le C_1(n, d, v, K_1, K_2, r, s, l) \le \lambda_l^{C, (r, s)}(X) \le C_2(n, d, v, K_1, K_2, r, s, l)<\infty\]
for any $X \in \overline{\mathcal{M}}$ and $l \ge 1$, where $\lambda_l^{C, k}$ and $\lambda_l^{C, (r, s)}$ denote the $l$-th eigenvalues of $\Delta_{C, k}$ and $\Delta_{C, (r, s)}$, respectively, and each $C_1$ as above goes to $\infty$ as $l \to \infty$.
\end{theorem}
For the upper bound of $\lambda^{H, 1}_l$, the dependence of an upper bound of the Ricci curvature is not essential. See (\ref{unibo}).

The second application is to give $L^{\infty}$-bounds for eigen-one-forms:
\begin{theorem}[$L^{\infty}$-estimates for eigen-one-forms]\label{gradd}
Let $\alpha \le \beta$ and let $\omega \in \mathcal{D}^2(\Delta_{H, 1}, X)=\mathcal{D}^2(\Delta_{C, 1}, X)$ with 
\[\frac{1}{H^n(X)}\int_X|\omega|^2dH^n=1.\]
Assume that $\omega \in E^{H, 1}_{\alpha}$ or $\omega \in E^{C, 1}_{\alpha}$.
Then 
\[||\omega||_{L^{\infty}} \le C(n, K_1, d, \beta).\]
\end{theorem}
It is easy to check that if $f$ is an $\alpha$-eigenfunction of the Laplacian $\Delta$ on $X$, then $df \in E^{H, 1}_{\alpha}$.
In particular Theorem \ref{gradd} implies a quantitative gradient estimate of an $\alpha$-eigenfunction $f$ of $\Delta$:
\[||df||_{L^{\infty}}\le C(n, K_1, d, \beta)\]
which was known in \cite{ch-co3} as a Cheng-Yau type gradient estimate for an eigenfunctions given in \cite{ch-yau}.
Thus Theorem \ref{gradd} can be regarded as a generalization of Cheng-Yau's gradient estimate for eigenfunctions to the case of eigen-one-forms on a nonsmooth setting. 
Note that similar $L^{\infty}$-bounds for eigen-$k$-forms and  eigen-tensor fields of $\Delta_{C, k}$ and $\Delta_{C, (r, s)}$ hold. 
See Theorem \ref{contrough}. 

In order to introduce the third application, 
let us denote the dimension of the space of harmonic one-forms on $X$ by $h_1(X)$, i.e.
\[h_1(X):=\mathrm{dim}\,E_0^{H, 1}.\]
Recall that Gigli defined  in \cite{gigli} the de Rham cohomology groups on an $RCD(K, \infty)$-space and proved the Hodge theorem.
Note that $(X, H^n)$ is an $RCD(K_1, \infty)$-space for every $X \in \overline{\mathcal{M}}$ (see subsection 1.2).
In particular $h_1(X)$ is equal to the first Betti number of $X$ in the sense of Gigli, which is defined by the dimension of the first de Rham cohomology group by him.

The final application is the upper semicontinuity of the first Betti numbers:
\begin{theorem}[Upper semicontinuity of first Betti numbers]\label{bettibetti}
If $X_i \stackrel{GH}{\to} X$ in $\overline{\mathcal{M}}$,
then
\begin{align}\label{uppersemibetti}
\limsup_{i \to \infty}h_1(X_i) \le h_1(X).
\end{align}
Moreover
\begin{align}\label{contibetti}
\lim_{i \to \infty}h_1(X_i) = h_1(X)
\end{align}
holds if and only if a uniform spectral gap for $\Delta_{H, 1}$ exists, i.e.  
\begin{align}\label{spectralgap}
\inf_i \mu_{H, 1}(X_i)>0,
\end{align}
where $\mu_{H, 1}$ is the first positive eigenvalue of $\Delta_{H, 1}$. 
\end{theorem}
Note that by the min-max principle, the condition (\ref{spectralgap}) is equivalent to that the uniform Poincar\'e inequality for one-forms holds, i.e. there exists $C>0$ such that
\[\inf_{\omega \in E^{H, 1}_0}\int_{X_i}\left| \eta-\omega \right|^2dH^n \le C\int_{X_i}\left( |d\eta|^2+|\delta \eta|^2\right)dH^n
\]
for any $i$ and $\eta  \in H^{1, 2}_H(T^*X_i)$, where $ H^{1, 2}_H(T^*X_i)$ is a Sobolev space of one-forms defined in \cite{gigli} (see subsection 3.2.1).

From the next subsection we devote to the study of \textit{curvature} of $X$.
Cheeger-Colding proved in \cite{ch-co1} that the regular set $\mathcal{R}$ of every $X \in \overline{\mathcal{M}}$, which is the set of $x \in X$ such that every tangent cone at $x$ is isometric to $\mathbf{R}^n$,  is smooth with the $C^{1, \alpha}$-Riemannian metric $g_X$ for every $\alpha \in (0, 1)$.
Moreover it is known in \cite{peters} by Peters that this regularity is sharp, i.e. we can not expect the $C^{1, 1}$-regulality of $g_X$ even if we assume a uniform bound of the sectional curvature  and a uniform positive lower bound of the injectivity radii.
In particular in general we can not define the Riemannian curvature tensor, the Ricci curvature and the scalar curvature at every regular point in the ordinary way of Riemannian geometry.

In order to overcome this difficulty, we use a nonsmooth approach in \cite{gigli} by Gigli with Cheeger-Naber's work of  \cite{chna}.
Then we can define the Riemannian curvature tensor $R_X$ of every $X \in \overline{\mathcal{M}}$ as an $L^q$-tensor of type $(0, 4)$ for every $q<2$.
\subsection{Riemannian curvature}
Recall that for any $X \in \mathcal{M}$ and $f_i \in C^{\infty}(X)$, by the definition of the Riemannian curvature tensor $R_X \in C^{\infty}(T^0_4X)$ of type $(0, 4)$ (we denote by $C^{\infty}(T^r_sX)$ the space of smooth tensor fields of type $(r, s)$ on $X$), we have
\begin{align*}
&f_0R_X(\nabla f_1, \nabla f_2, \nabla f_3, \nabla f_4)\\
&=f_0\left\langle \nabla_{\nabla f_1}\nabla_{\nabla f_2}\nabla f_3-\nabla_{\nabla f_2}\nabla_{\nabla f_1}\nabla f_3-\nabla_{[\nabla f_1, \nabla f_2]}\nabla f_3, \nabla f_4\right\rangle \\
&=\left\langle f_0\nabla f_1, \nabla \left\langle \nabla_{\nabla f_2}\nabla f_3, \nabla f_4\right\rangle \right\rangle-f_0\left\langle \nabla_{\nabla f_2}\nabla f_3, \nabla_{\nabla f_1}\nabla f_4 \right \rangle \\
&\,\,\,\,\,-\left\langle f_0\nabla f_2, \nabla \left\langle \nabla_{\nabla f_1}\nabla f_3, \nabla f_4\right\rangle \right\rangle+f_0\left\langle \nabla_{\nabla f_1}\nabla f_3, \nabla_{\nabla f_2}\nabla f_4 \right \rangle -\left\langle f_0\nabla_{[\nabla f_1, \nabla f_2]}\nabla f_3, \nabla f_4\right\rangle,
\end{align*}
where $\langle \cdot, \cdot \rangle$ denote the canonical metrics defined by the Riemannian metric $g_X$.
Integrating this equality on $X$ with the divergence formula yields
\begin{align}\label{riem}
&\int_Xf_0\langle R_X, df_1 \otimes df_2, \otimes df_3 \otimes df_4\rangle dH^n \nonumber \\
&=\int_X\left(\left( -\langle \nabla f_0, \nabla f_1\rangle +f_0\Delta f_1\right)\mathrm{Hess}_{f_3}\left(\nabla f_2, \nabla f_4\right)-f_0\left\langle \mathrm{Hess}_{f_3}(\nabla f_2, \cdot), \mathrm{Hess}_{f_4}(\nabla f_1, \cdot )\right\rangle \right) dH^n\nonumber \\
&\,\,\,\,\,-\int_X\left(\left( -\langle \nabla f_0, \nabla f_2\rangle +f_0\Delta f_2\right)\mathrm{Hess}_{f_3}\left(\nabla f_1, \nabla f_4\right)-f_0\left\langle \mathrm{Hess}_{f_3}(\nabla f_1, \cdot), \mathrm{Hess}_{f_4}(\nabla f_2, \cdot )\right\rangle \right)dH^n\nonumber \\
&\,\,\,\,\,-\int_Xf_0\mathrm{Hess}_{f_3}\left( [\nabla f_1, \nabla f_2], \nabla f_4\right)dH^n.
\end{align}
Note that this formula (\ref{riem}) gives a characterization of $R_X$, i.e.
if a tensor $T \in C^{\infty}(T^0_4X)$ satisfies (\ref{riem}) instead of $R_X$ for any $f_i \in C^{\infty}(X)$, then $T=R_X$.
This is a direct consequence of that the space
\[\left\{ \sum_{i=1}^Nf_{0, i}df_{1, i}\otimes df_{2, i} \otimes df_{3, i} \otimes df_{4, i}; N \in \mathbf{N}, f_{j, i} \in C^{\infty}(X)\right\}\]
is dense in the space of $L^p$-tensor fields of type $(0, 4)$ on $X$, denoted by $L^p(T^0_4X)$, for every $p \in (1, \infty)$.

We generalize this observation to our nonsmooth setting based on the regularity theory of the heat flow  given in \cite{ags} by Ambrosio-Gigli-Savar\'e.
A key idea is to use \textit{test functions} instead of smooth functions. 
Note that Gigli established in \cite{gigli} the second order differential calculus on $RCD$-spaces (which contain nonsmooth spaces) in this direction and that the author gave in \cite{ho} the second order differential structure on Ricci limit spaces in the other direction.
Moreover it was shown in \cite{hoell} that these second order differential calculus are compatible to each other on Ricci limit spaces.

We first recall the definition of the Sobolev spaces for functions.

Let $X \in \overline{\mathcal{M}}$.
For every $p \in (1, \infty)$ let $H^{1, p}(X)$ be the completion of the space of Lipschitz functions on $X$, denoted by
$\mathrm{LIP}(X)$, 
with respect to the norm
\begin{align}\label{wbt}
||f||_{H^{1, p}}:=\left( ||f||_{L^p}^p+||\nabla f||_{L^p}^p\right)^{1/p},
\end{align}
where $\nabla f$ is the gradient vector field of $f$ which is well-defined on $\mathcal{R}$ a.e. sense by Rademacher's theorem. See for instance \cite{ch1} for a more general setting.

Second, we recall the definition of the Dirichlet Laplacian $\Delta$ acting on functions on $X$.

Let $\mathcal{D}^2(\Delta, X)$ denote the set of $f \in H^{1, 2}(X)$ such that there exists a (unique) function $F \in L^2(X)$ such that
\[\int_{X}g_X(\nabla f, \nabla h)dH^n=\int_XFhdH^n\] 
for every $h \in H^{1, 2}(X)$.
Then we denote $F$ by $\Delta f$.

We introduce several regularities of $f \in \mathcal{D}^2(\Delta, X)$ given in \cite{hoell}. 
\begin{enumerate}
\item[(1.1)] We have $|\nabla f|^2 \in H^{1, p_n}(X)$, where $p_n:=2n/(2n-1)$.
\item[(1.2)] $f$ is weakly twice differentiable (or weak $C^{1, 1}$) on $X$ in the sense of \cite{ho}, i.e. there exists a countable family of Borel subsets $A_i$ of $\mathcal{R}$ with
\[H^n\left(\mathcal{R} \setminus \bigcup_iA_i\right)=0\]
such that  for every smooth coordinate patch $\phi :U \to \mathbf{R}^n$ of $\mathcal{R}$ we see that $f \circ \phi^{-1}$ is differentiable on $\phi(A_i \cap U)$ and that the Jacobi matrix
\[J\left( f \circ \phi^{-1}\right)\]
is Lipschitz on $\phi(A_i \cap U)$.
Note that this is weaker than that $f$ is $C^2$ on $\mathcal{R}$ with compact support. 
\item[(1.3)] In the ordinary way of Riemannian geometry with (1.2),  we can define the Hessian of $f$, denoted by $\mathrm{Hess}_f$, as a Borel tensor field of type $(0, 2)$ on $X$.
Moreover we see that the Hessian is $L^2$, i.e. $\mathrm{Hess}_f \in L^2(T^0_2X)$, that $\mathrm{Hess}_f$ coincides with Gigli's one defined in \cite{gigli} and that the trace is equal to $-\Delta f$.
In particular
\begin{align}\label{pqty}
|\mathrm{Hess}_f|^2 \ge \frac{(\Delta f)^2}{n}.
\end{align}
\end{enumerate}
See \cite[Theorems $1.9$ and $4.11$]{hoell} for the detail.

As we mentioned in (1.2), by the $C^{1, \alpha}$-regularity of the Riemannian metric $g_X$ on $\mathcal{R}$, we see that the space of $C^2$-functions on $\mathcal{R}$ with compact support, denoted by  $C^2_c(\mathcal{R})$, is a subspace of $\mathcal{D}^2(\Delta, X)$ and that $\mathrm{Hess}_f$ and $\Delta f$ of $f \in C^2_c(\mathcal{R})$ in the sense above coincide with the ordinary them in Riemannian geometry, respectively.

Define \textit{the space of test functions} by
\begin{align}\label{testdefini}
\mathrm{Test}F(X):=\{f \in \mathcal{D}^2(\Delta, X) \cap \mathrm{LIP}(X); \Delta f\in H^{1, 2}(X)\}.
\end{align}
We are now in a position to give the definition of the Riemannian curvature tensor $R_X$.
\begin{theorem}[Riemannian curvature]\label{curvaturetensor}
There exists a unique tensor $R_X \in \bigcap_{q<2}L^q(T^0_4X)$ such that
(\ref{riem}) holds for any $f_i \in \mathrm{Test}F(X)$.
We call $R_X$ \textit{the Riemannian curvature tensor of $X$}.
\end{theorem}
We give a remark on Theorem \ref{curvaturetensor}.

The $L^q$-bound of $R_X$ closely depends on Cheeger-Naber's estimate given in \cite{chna}:
\[\int_M|R_M|^qdH^n \le C(n, d, K_1, K_2, v)\]
for every $M \in \mathcal{M}$.
Note that by the proof of Theorem \ref{curvaturetensor} if $n=4$ or Conjecture $9.1$ in \cite{chna} holds, then $R_X \in L^2(T^0_4X)$.
See also Conjecture 6.3 in \cite{na14}.

We will establish several properties on our Riemannian curvature tensor, which include the first Bianchi identity a.e. sense.
We end this subsection by introducing one of them, which is a continuity of Riemannian curvature tensor with respect to the Gromov-Hausdorff topology:
\begin{theorem}[$L^p$-weak continuity of Riemannian curvature]\label{riem2}
Let $X_i \stackrel{GH}{\to} X$ in $\overline{\mathcal{M}}$.
Then $R_{X_i}$ $L^q$-converges weakly to $R_{X}$ on $X$ for every $q \in (1, 2)$.
\end{theorem}
See Section 2 or \cite{holp} for the definition of $L^p$-convergence with respect to the Gromov-Hausdorff topology. 
\subsection{Ricci curvature}
Let $X \in \overline{\mathcal{M}}$.
By using the Riemannian curvature tensor $R_X$, define \textit{the Ricci curvature $\mathrm{Ric}_X$ of} $X \in \overline{\mathcal{M}}$ by a contraction of $R_X$ in the ordinary way of Riemannian geometry:
\begin{align}\label{definitionricci}
\mathrm{Ric}_X(u, v):= \sum R_X(e_i, u, v, e_i),
\end{align}
where $u, v \in T_xX$ and $\{e_i\}_i$ is an orthogonal basis of $T_xX$.
Note that this is well-defined on $X$ a.e. sense and that the $L^q$-bound of $R_X$ as in Theorem \ref{curvaturetensor} yields
\[\mathrm{Ric}_X \in \bigcap_{q<2}L^q(T^0_2X).\]
However, recall that we will prove $\mathrm{Ric}_X \in L^{\infty}(T^0_2X)$ in (i) of Theorem \ref{spectr}. 

Let us introduce several results on $\mathrm{Ric}_X$.
The first one is Bochner's formula, which is well-known on a smooth closed Riemannian manifold.
\begin{theorem}[Bochner's formula for differential one-forms]\label{pointwiseboch2}
For every $\omega \in \mathcal{D}^2(\Delta_{H, 1}, X)$,
\begin{align}\label{89m}
-\frac{1}{2}\Delta |\omega|^2=|\nabla \omega|^2-\langle \Delta_{H, 1}\omega, \omega \rangle + \mathrm{Ric}_X(\omega^*, \omega^*),
\end{align}
where $\omega^* \in L^2(TX)$ is the dual vector field of $\omega$ by the Riemannian metric $g_X$. 
\end{theorem}
Note that the left hand side of (\ref{89m}) is taken by a generalized Laplacian of $|\omega|^2$ and is in $L^1(X)$.
See Theorem \ref{pointwiseboch} for the precise statement.

Applying Theorem \ref{pointwiseboch2} when $\omega$ is an exact one-form, i.e. $\omega :=df$ for some $f \in \mathcal{D}^2(\Delta, X)$, yields Bochner's formula for functions:
\begin{theorem}[Bochner's formula for functions]\label{boch}
For every $f \in \mathcal{D}^2(\Delta, X)$,
\[-\frac{1}{2}\Delta |\nabla f|^2 =|\mathrm{Hess}_f|^2-\left\langle \nabla \Delta f, \nabla f\right\rangle + \mathrm{Ric}_X(\nabla f, \nabla f).\]
In particular for every $h \in\mathcal{D}^2(\Delta, X)$ with $\Delta h \in L^{\infty}(X)$,
\begin{align}\label{ss}
-\frac{1}{2}\int_{X}\Delta h |\nabla h|^2dH^n=\int_{X}\left(h |\mathrm{Hess}_{f}|^2-h\left\langle \nabla \Delta f, \nabla f\right\rangle + h\mathrm{Ric}_X(\nabla f, \nabla f)\right)dH^n.
\end{align} 
\end{theorem}

Next we give an application of Theorem \ref{boch} to the study of $(X, H^n)$ as an \textit{$RCD$-space}.

The structure theory of nonsmooth spaces with Ricci curvature bounded below is rapidly studied, recently.
See for instance \cite{ags, lv, sturm1, sturm2} for pionear works by Ambrosio-Gigli-Savar\'e, Lott-Villani, and Sturm.
Roughly speaking, a metric measure space $(Y, \upsilon)$ is said to be an $CD(K, N)$-space if the Ricci curvature is bounded below by $K$ and the dimension is bounded above by $N$, where $K \in \mathbf{R}$ and $N \in (0, \infty]$.
Moreover if the Sobolev space $H^{1, 2}(Y)$ is a Hilbert space, then it is said to be an $RCD(K, N)$-space.
On the other hand Bacher-Sturm introduced in \cite{bs} a 
reduced version of a $CD$-condition, which is called $CD^*(K, N)$-condition.
They proved that 
every $CD(K, N)$-space is an $CD^*(K, N)$-space, but the converse is not true in general. 
However every $RCD^*(K, N)$-space ( i.e.  $CD^*(K, N)$-space satisfying that the Sobolev space $H^{1, 2}$ is a Hibert space) is an  $RCD(K, \infty)$-space. 
A Ricci limit space gives 
a typical example of $RCD$-spaces.
In particular for every $X \in \overline{\mathcal{M}}$, $(X, H^n)$ is an $RCD(K_1, n)$-space.

Combining (\ref{pqty}) and (\ref{ss}) with the equivalence between Bochner's inequality and a curvature dimension condition \cite{ams, eks, gigli} by Ambrosio-Mondino-Savar\'e, Erber-Kuwada-Sturm, and Gigli yields the following new equivalence on our setting.
\begin{theorem}[Equivalence between lower bounds on Ricci curvature and $RCD$-conditions]\label{equ}
For every $K \in \mathbf{R}$, the following three conditions are equivalent:
\begin{enumerate}
\item $\mathrm{Ric}_X \ge K$ on $X$ a.e. sense, i.e. for a.e. $x \in X$, we have
\[\mathrm{Ric}_X(u, u) \ge K|u|^2\]
for every $u \in T_xX$.
\item $(X, H^n)$ is an $RCD^*(K, n)$-space.
\item $(X, H^n)$ is an $RCD(K, \infty)$-space.
\end{enumerate}
\end{theorem}
Note that in Theorem \ref{equ} if $X$ is smooth, i.e. $X \in \mathcal{M}$, by the smoothness of the Ricci curvature, it is checked directly that (i) of Theorem \ref{equ} holds if and only if $\mathrm{Ric}_X \ge K$ holds on $X$ in the ordinary sense.
It also follows from the compatibility between a $CD$-condition and a lower bound of Ricci curvature on the smooth setting in \cite{lv, sturm1, sturm2}.

Finally we give a behavior of our Ricci curvature with respect to the Gromov-Hausdorff topology:
\begin{theorem}[$L^p$-weak continuity of Ricci curvature]\label{ellp}
If $X_i \stackrel{GH}{\to} X$ in $\overline{\mathcal{M}}$, then $\mathrm{Ric}_{X_i}$ $L^p$-converges weakly to $\mathrm{Ric}_{X}$ on $X$ for every $p \in (1, \infty)$.
\end{theorem}
\subsection{Scalar curvature}
Let us define the scalar curvature $s_X$ of $X \in \overline{\mathcal{M}}$ by taking the trace of $\mathrm{Ric}_X$:
\[s_X(x):=\sum_{i=1}^n\mathrm{Ric}_X(e_i, e_i),\]
where $\{e_i\}_i$ is an orthogonal basis of $T_xX$.
Note that this is well-defined on $X$ a.e. sense and that 
by (i) of Theorem \ref{spectr} we have $s_X \in L^{\infty}(X)$ with
\[K_1n \le s_{X}(x) \le K_2n\]
a.e. $x \in X$.

We will also  give several results on the scalar curvature which include the Gauss-Bonnet formula if $n=2$, and the Chern-Gauss-Bonnet formula if $n=4$ on a suitable setting (Proposition \ref{gauss}).

We here introduce one of them on a behavior of our scalar curvature which is similar to Theorem \ref{ellp}:
\begin{theorem}[$L^p$-weak continuity of scalar curvature]\label{contsca}
If $X_i \stackrel{GH}{\to} X$ in $\overline{\mathcal{M}}$, then $s_{X_i}$ $L^p$-converges weakly to $s_X$ on $X$ for every $p \in (1, \infty)$, i.e. $\sup_i\|s_{X_i}\|_{L^p}<\infty$ is satisfied for any $p \in (1, \infty)$ and 
\[\lim_{i \to \infty}\int_{B_r(x_i)}s_{X_i}dH^n=\int_{B_r(x)}s_XdH^n\]
holds for any $r>0$ and $x_i \stackrel{GH}{\to} x$. 
In particular the total scalar curvature of $X_i$ converges to that of $X$:
\[\lim_{i \to \infty} \int_{X_i}s_{X_i}dH^n=\int_Xs_XdH^n.\]
\end{theorem}
Note that this $L^p$-weak continuity is sharp in some sense. See Example \ref{notsca}.

From now on we apply Theorem \ref{contsca} in order to answer a question by Lott in \cite{lott}.
For that we recall the following Lott's result.
\begin{theorem}[Stability of lower bounds on scalar curvature under bounded sectional curvature \cite{lott}]\label{ypk}
Let $X_i$ be a sequence of $n$-dimensional closed Riemannian manifolds with 
\begin{align}\label{secbou}
\sup_i||\mathrm{Sec}_{X_i}||_{L^{\infty}}<\infty,
\end{align}
let $X$ be the noncollapsed Gromov-Hausdorff limit, and let $C \in \mathbf{R}$ with
\[s_{X_i}\ge C\]
on $X_i$ for every $i$,
where $\mathrm{Sec}_{X_i}$ is the sectional curvature of $X_i$.
Then we have
\[s_{X} \ge C\]
on $X$.
\end{theorem}
Note that it was known that the scalar curvature is well-defined on the setting as in Theorem \ref{ypk}, 
that the upper bound version of Theorem \ref{ypk} also holds,  and that Lott discussed a similar stability in a collapsing case \cite{lott}. 

He also gave the following question:
\begin{enumerate}
\item[\textbf{(Q0)}] Does the same conclusion as in Theorem \ref{ypk} hold even if we replace the assumption (\ref{secbou}) by 
a uniform bound of Ricci curvature;
\[\sup_i||\mathrm{Ric}_{X_i}||_{L^{\infty}}<\infty?\] 
\end{enumerate}

Applying Theorem \ref{contsca} yields the following. 
\begin{corollary}[Stability of bounds of scalar curvature with bounded Ricci curvature]\label{nn6}
Lott's question \textbf{(Q0)} has a positive answer.
\end{corollary}
We discuss a refinement of Corollary \ref{nn6} in the next subsection.
\subsection{Stability}
In this section we discuss stabilities of curvature bounds with respect to the Gromov-Hausdorff topology.

Theorem \ref{ypk} and Corollary \ref{nn6} give stabilities on scalar curvature.
We recall other well-known two stabilities:
\begin{enumerate}
\item{(Stability of sectional curvature bounds \cite{bgp})} The Gromov-Hausdorff limit space of a sequence of Alexandrov (or CAT, respectively) space of curvature  bounded below (or above, respectively) by a real number $K$ is also an Alexandrov (or CAT, respectively) space of curvature bounded below (or above, respectively) by $K$.
\item{(Stability of lower Ricci curvature bounds \cite{ags, bs, eks, lv, sturm1, sturm2})} The measured Gromov-Hausdorff limit space of a sequence of $RCD(K, N)$-spaces (or $RCD^*(K, N)$, or $CD(K, N)$, or $CD^*(K, N)$-spaces, respectively) is also an $RCD(K, N)$-space (or $RCD^*(K, N)$, or $CD(K, N)$, or $CD^*(K, N)$-space, respectively).
\end{enumerate}
From now on we introduce new stabilities of curvature bounds on our setting.
In order to give the precise statement let us define the sectional curvature 
\[\mathrm{Sec}_X:\mathrm{Gr}_2(TX) \to \mathbf{R}\]
of $X \in \overline{\mathcal{M}}$ by the ordinary way of Riemannian geometry:
\[\mathrm{Sec}_X(\sigma):=R_X(u, v, v, u),\]
where $\mathrm{Gr}_2(V)$ denotes the set of $2$-dimensional subspaces of a vector space $V$ and $u, v \in T_xX$ is an orthogonal basis of $\sigma$.
Note that for a.e. $x \in X$ and every $\sigma \in \mathrm{Gr}_2(T_xX)$, $\mathrm{Sec}_X(\sigma)$ is well-defined.
 
For a Borel subset $A$ of $X$, we say that 
\[\mathrm{Sec}_X\ge K\]
on $A$ a.e. sense if for a.e. $x \in A$, we have
\[\mathrm{Sec}_X(\sigma)\ge K\]
for every $\sigma \in \mathrm{Gr}_2(T_xX)$.
As with the case of Ricci curvature it is easy to check that if $X$ is smooth and $A$ is an open subset of $X$, then $\mathrm{Sec}_X \ge K$ on $A$ is satisfied a.e. sense  if and only if  $\mathrm{Sec}_X \ge K$ on $A$ is satisfied in the  ordinary sense.

Similarly we define bounds of sectional curvature, of Ricci curvature, and of scalar curvature a.e. sense.

We are now in a position to give new stabilities.
\begin{theorem}[Stabilities of curvature bounds]\label{stabilityricci}
Let $X_i \stackrel{GH}{\to} X$ in $\overline{\mathcal{M}}$, let $r>0$, let $K \in \mathbf{R}$, and let $x_i \stackrel{GH}{\to} x$, where $x_i \in X_i$ and $x \in X$.
Then we have the following:
\begin{enumerate}
\item If $\mathrm{Sec}_{X_i} \ge K$ (or $\mathrm{Sec}_{X_i} \le K$, respectively) on $B_r(x_i)$ a.e. sense for every $i$, then $\mathrm{Sec}_X \ge K$ (or $\mathrm{Sec}_{X} \le K$, respectively) on $B_r(x)$ a.e. sense.
\item  If $\mathrm{Ric}_{X_i} \ge K$ (or $\mathrm{Ric}_{X_i} \le K$, respectively) on $B_r(x_i)$ a.e. sense for every $i$, then $\mathrm{Ric}_X \ge K$ (or $\mathrm{Ric}_{X} \le K$, respectively) on $B_r(x)$ a.e. sense.
\item  If $s_{X_i} \ge K$ (or $s_{X_i} \le K$, respectively) on $B_r(x_i)$ a.e. sense for every $i$, then $s_X \ge K$ (or $s_{X} \le K$, respectively) on $B_r(x)$ a.e. sense.
\end{enumerate}
\end{theorem}
\subsection{Compatibility}
In this section we discuss compatibilities of our curvature on several settings.

Let $X \in \overline{\mathcal{M}}$.
We first introduce the compatibility with the smooth setting:
\begin{proposition}[Compatibility with the smooth setting]\label{curvaturecompatibility}
If an open subset $U$ of $\mathcal{R}$ satisfies that $(U, g_X|_U)$ is a $C^{\infty}$-Riemannian manifold, then the Riemannian curvature tensor $R_X$ as in Theorem \ref{curvaturetensor} coincides with the ordinary one on $U$ a.e. sense.
In particular similar compatibilities hold for the Ricci curvature and the scalar curvature.
\end{proposition}

Next we discuss compatibilities with nonsmooth settings.
For that we first recall \textit{Gigli's Ricci curvature} on our setting.

He defined the Ricci curvature $\mathbf{Ric}_X$ on an $RCD(K, \infty)$-space as a measure valued bilinear continuous map:
\[\mathbf{Ric}_X:H^{1, 2}_H(TX) \times H^{1, 2}_H(TX) \to \mathrm{Meas}(X),\]
based on Bochner's formula, where $H^{1, 2}_H(TX)$ is a Sobolev space of vector fields on $X$ defined in \cite{gigli} and $\mathrm{Meas}(X)$ is the Banach space of finite signed Radon measures on $X$ equipped with the total cariation norm.
See \cite{gigli} or subsections 3.2.1 and 3.4.3 for the precise definitions.

On the other hand, by the $L^{\infty}$-bound on Ricci curvature $\mathrm{Ric}_X$ we can define a bilinear continuous map
\[\mathrm{Ric}_X:H^{1, 2}_H(TX) \times H^{1, 2}_H(TX) \to L^1(X)\]
by 
\[(V, W) \mapsto \mathrm{Ric}_X(V, W).\]

Then we have the following compatibility.
\begin{theorem}[Compatibility with Gigli's Ricci curvature]\label{gigliricci}
We see that Gigli's Ricci curvature $\mathbf{Ric}_X$ coincides with our Ricci curvature $\mathrm{Ric}_X$ via  the canonical inclusion 
\[\imath: L^1(X) \hookrightarrow \mathrm{Meas}(X)\]
defined by
\[\imath (f)(A):=\int_AfdH^n,\]
i.e.
\[\mathbf{Ric}_X(V, W)(A)=\int_A\mathrm{Ric}_X(V, W)dH^n\]
holds for any Borel subset $A$ of $X$ and $V, W \in H^{1, 2}_H(TX)$.
\end{theorem}
By Theorem \ref{gigliricci} we can give positive answers to several questions in \cite{gigli} by Gigli on our setting.
For example we see that $\mathbf{Ric}_X$ can be extended to a  bilinear continuous map
\[L^2(X) \times L^2(X) \to \mathbf{R},\]
canonically.
This gives a positive answer to a his question stated in 156 page of \cite{gigli}.
See Remark \ref{gigliquestion} for the detail.

We end this subsection by introducing the other compatibility with \textit{Lott's Ricci measure}.
In \cite{lott2} he defined the Ricci measure on a $C^{1, 1}$-manifold with a measurable Riemannian metric and a tame connection.
Note that by the $C^{1, \alpha}$-regularity of the Riemannian metric $g_X$ on $\mathcal{R}$, the triple $(\mathcal{R}, g_X|_{\mathcal{R}}, \nabla)$ gives such an example, where $\nabla$ is the Levi-Civita connection.

Then we have the following.
\begin{theorem}[Compatibility with Lott's Ricci measure]\label{comlott}
Our Ricci curvature $\mathrm{Ric}_X$ coincides with Lott's Ricci measure on $\mathcal{R}$.
\end{theorem}
\subsection{The structure of the paper}
This paper is organized as follows.

In Section 2, we give a key result on $L^2$-strong convergence of Hessians (Theorem \ref{L2Hess}) in order to prove the existence of our Riemannian curvature tensor, which loosely states that if $\Delta f_i$ $L^2$-converges strongly to $\Delta f$ with respect to the Gromov-Hausdorff convergence $X_i \stackrel{GH}{\to} X$ in $\overline{\mathcal{M}}$, then $\mathrm{Hess}_{f_i}$ $L^2$-converges strongly to $\mathrm{Hess}_f$.
It is worth pointing out that a key estimate to establish this $L^2$-strong convergence is Cheeger-Naber's $L^{p}$-estimate of the Hessian:
\begin{align}\label{Lpboundhess}
\int_M|\mathrm{Hess}_f|^{p}dH^n \le C(n, d, K_1, K_2, v, L, p)
\end{align}
for any $M \in \mathcal{M}$, $p \in [1, 4)$ and $f \in C^{\infty}(M)$ with $|f|+|\Delta f| \le L$ (see \cite[Theorem $7.20$]{chna}).

In Section 3, first, by using the $L^2$-strong convergence of Hessians we show the existence of our curvature and 
give a new smooth approximation theorem (Theorems \ref{14}).
Second, we see that the approximation theorem yields 
several results on our curvature we stated in subsections 1.1-1.5.
Moreover we discuss the $L^p$-strong convergence of our curvature in subsection 3.3.

Section 4 is devoted to their applications which include spectral convergence.
Note that a weak version of the sepctral convergence of $\Delta_{H, 1}$ was shown in \cite[Theorem 1.12]{hoell} under a uniform lower bound of Ricci curvature only.  
An essential difference from the previous result is to prove that the $L^2$-strong limit of differential eigen-one-forms is in $H^{1, 2}_H(T^*X)$ (note that the previous result, \cite[Theorem 1.12]{hoell}, states that it is in $W^{1, 2}_H(T^*X)$).
It is worth pointing out that we introduce four key results in order to prove it. 
The first one is Li-Tam's mean value inequality in \cite{LT}, the second one is the $C^{1}$-regularity of the Riemannian metric on the regular set in \cite{ch-co1} by Cheeger-Colding, the third one is that the Hausdorff dimension of the singular set is at most $n-4$ in \cite{chna} by Cheeger-Naber, and the fourth one is the Rellich compactness theorem for functions and tensor fields with respect to the Gromov-Hausdorff topology in \cite{holp, hoell} by the author.

As other applications, we study noncollapsed K\"ahler Ricci limit spaces with bounded Ricci curvature.
In particular we consider a sequence of compact K\"ahler manifolds $(X_i, g_{X_i}, J_i)$ with a uniform bound of Ricci curvature and a constant multiple of the K\"ahler form is in the first Chern class, and the noncollapsed Gromov-Hausdorff limit $(X, g_X, J)$.
Then we prove the equivalence between $L^p$-strong convergence of Ricci curvature and that of scalar curvature.
Moreover we define the Ricci potential $F_X$ of the limit $(X, g_X, J)$,  prove that Ricci potentials $F_{X_i}$ converges uniformly to $F_X$ on $X$, and show the equivalence between $L^p$-strong convergence of Ricci curvature and $L^2$-strong convergence of Hessians of Ricci potentials (Theorem \ref{equivalence}).
We also establish the Weitzenb\"ock formula on a Fano Ricci limit space with bounded Ricci curvature (Theorem \ref{23}).

In the final section, Section 5, we give several open problems.

\textbf{Acknowledgments.}
The author is grateful to Aaron Naber for his kind communications on examples given in \cite{hn}.
The author would like to express his appreciation to Hausdorff Research Institute for Mathematics (HIM) for warm hospitality.
He got a key idea of the paper 
during the stay of 
the Junior Hausdorff Trimester Program
on ``Optimal Transportation” in HIM.
He is also grateful to the referee for the very thorough reading, and valuable suggestions which are
greatly improved the overall mathematical quality of the paper.
This work was supported  by Grant-in-Aid for Young Scientists (B) $24740046$,  $16K17585$ and Grant-in-Aid for challenging Exploratory Research $26610016$.

\section{$L^2$-strong convergence of Hessians}
We first recall the definition of Gromov-Hausdorff convergence.
For a sequence of compact metric spaces $Y_i$, we say that \textit{$Y_i$ Gromov-Hausdorff converges to a compact metric space $Y$} if there exist a sequence of positive numbers $\epsilon_i$ with $\epsilon_i \to 0$, and a sequence of maps $\phi_i: Y_i \to Y$ such that the following two conditions hold:
\begin{enumerate}
\item{($\phi_i$ is an almost isometric embedding.)}
We have
\[\left|d_Y\left( \phi_i(x), \phi_i(y)\right) - d_{Y_i}\left(x, y\right)\right|<\epsilon_i\]
for any $i$ and $x, y \in Y_i$.
\item{($\phi_i$ is almost surjective.)}
We have
\[Y=B_{\epsilon_i}\left(\phi_i(Y_i)\right),\]
where $B_r(A)$ denotes the open $r$-neighbourhood of $A$.
\end{enumerate} 
Then we denote the convergence by
\[Y_i \stackrel{GH}{\to} Y.\]
For a sequence $y_i \in Y_i$, we say that \textit{$y_i$ Gromov-Hausdorff converges to a point $y \in Y$ (with respect to $Y_i \stackrel{GH}{\to} Y$)} if $\phi_i(y_i) \to y$ in $Y$. Then we denote it by 
\[y_i \stackrel{GH}{\to} y.\]
Moreover for a sequence of Radon measures $\upsilon_i$ on $Y_i$, we say that {$\upsilon_i$ measured Gromov-Hausdorff converges to a Radon measure $\upsilon$ on $Y$ (with respect to $Y_i \stackrel{GH}{\to} Y$)} if 
\[\lim_{i \to \infty}\upsilon_i(B_r(y_i))=\upsilon (B_r(y))\]
for any $r>0$ and $y_i \stackrel{GH}{\to} y$.
Then we denote it by 
\[(Y_i, \upsilon_i) \stackrel{GH}{\to} (Y, \upsilon).\]

Our setting in this section is the following:
\begin{itemize}
\item Let $\{X_i\}_i$ be a sequence in $\mathcal{M}$ and let $X \in \overline{\mathcal{M}}$ be the Gromov-Hausdorff limit.
\end{itemize}
Recall that as we mentioned in Section 1, $(X_i, H^n) \stackrel{GH}{\to} (X, H^n)$ holds.

One of key notions in this paper is \textit{$L^p$-convergence with respect to the Gromov-Hausdorff topology} which were established in \cite{KS, holp} in the cases of functions and tensor fields, respectively.

We first recall it in the case of functions.
Let $p \in (1, \infty)$, let $R>0$ and let $x_i \stackrel{GH}{\to} x$, where $x_i \in X_i$ and $x \in X$.
\begin{definition}[$L^p$-convergence of functions]
Let $f_i$ be a sequence in $L^p(B_R(x_i))$.
\begin{enumerate}
\item[(i)] We say that \textit{$f_i$ $L^p$-converges weakly to an $L^p$-function $f \in L^p(B_R(x))$ on $B_R(x)$} if 
\[\sup_i||f_i||_{L^p}<\infty\]
and
\[\lim_{i \to \infty}\int_{B_r(y_i)}f_i\,dH^n=\int_{B_r(y)}f\,dH^n\]
hold for any $y_i \stackrel{GH}{\to} y$ and $r>0$ with $\overline{B}_r(y) \subset B_R(x)$, where $y_i \in B_R(x_i)$ and $y \in B_R(x)$.
\item[(ii)] We say that \textit{$f_i$ $L^p$-converges strongly to $f \in L^p(B_R(x))$ on $B_R(x)$} if it is an $L^p$-weak convergent sequence on $B_R(x)$ with
\[\limsup_{i \to \infty}\int_{B_R(x_i)}|f_i|^p\,dH^n \le \int_{B_R(x)}|f|^p\,dH^n.\]
\end{enumerate}
\end{definition}

Next we recall the case of tensor fields.
Let 
\[T^r_sX:=\bigotimes_{i=1}^rTX \otimes \bigotimes_{i=1}^s T^*X\]
and let us denote the set of $L^p$-tensor fields of type $(r, s)$ over a Borel subset $A$ of $X$ by $L^p(T^r_sA)$. 
\begin{definition}[$L^p$-convergence of tensor fields]\label{Lpreal}
Let $r, s \in \mathbf{Z}_{\ge 0}$ and let $T_i$ be a sequence in $L^p(T^r_sB_R(x_i))$.
\begin{enumerate}
\item[(i)] We say that \textit{$T_i$ $L^p$-converges weakly to an $L^p$-tensor field $T \in L^p(T^r_sB_R(x))$ on $B_R(x)$} if 
\[\sup_i||T_i||_{L^p}<\infty\]
and 
\begin{align*}
&\lim_{i \to \infty}\int_{B_r(y_i)}\left\langle T_i, \nabla r_{z_i^1} \otimes \cdots \otimes \nabla r_{z_i^r}\otimes dr_{z_i^{r+1}} \otimes \cdots \otimes dr_{z_i^{r+s}}\right\rangle \,dH^n\\
&=\int_{B_r(y)}\left\langle T, \nabla r_{z^1} \otimes \cdots \otimes \nabla r_{z^r}\otimes dr_{z^{r+1}} \otimes \cdots \otimes dr_{z^{r+s}} \right\rangle\,dH^n
\end{align*}
hold for any $y_i, z_i^j \stackrel{GH}{\to} y, z^j$, respectively, and $r>0$ with $\overline{B}_r(y) \subset B_R(x)$,  where $y_i, z_i^j \in B_R(x_i)$,  $y, z^j \in B_R(x)$ and $r_z$ denotes the distance function from $z$.
\item[(ii)] We say that \textit{$T_i$ $L^p$-converges strongly 
to $T \in L^p(T^r_sB_R(x))$ on $B_R(x)$} if it is an $L^p$-weak convergent sequence on $B_R(x)$ with 
\[\limsup_{i \to \infty}\int_{B_R(x_i)}|T_i|^p\,dH^n\le \int_{B_R(x)}|T|^p\,dH^n.\]
\end{enumerate}
\end{definition}

We need the following previous results:
\begin{enumerate}
\item[(2.1)] {(Lipschitz representation)} For any $p \in (1, \infty)$ and  $f \in H^{1, p}(X)$, $\nabla f \in L^{\infty}(TX)$ holds if and only if $f \in \mathrm{LIP}(X)$ holds. Moreover the Lipschitz constant of $f$ is equal to $||\nabla f||_{L^{\infty}}$.
See for instance \cite[Proposition $2.8$]{hoch}.  
\item[(2.2)]{(Lipschitz regularity of solutions of Poisson's equations)} For any $q \in (n, \infty]$ and $f \in \mathcal{D}^2(\Delta, X)$ with $||\Delta f||_{L^{q}} \le L$ we have 
\[||\nabla f||_{L^{\infty}}\le C(n, v, K_1, d, L, q).\] 
See \cite[Theorem $1.2$]{jiang11} (see also \cite[Theorem 3.1]{jiang14}).
\item[(2.3)]{(Rellich compactness theorem)} For every sequence $f_i \in H^{1, p}(B_R(x_i))$ with
\[\sup_i||f_i||_{H^{1, p}}<\infty,\]
there exist a subsequence $i(j)$ and $f \in H^{1, p}(B_R(x))$ such that $f_{i(j)}$ $L^p$-converges strongly to $f$ on $B_R(x)$ and that $\nabla f_{i(j)}$ $L^p$-converges weakly to $\nabla f$ on $B_R(x)$, where $H^{1, p}(U)$ is the Sobolev space on an open subset $U$ of $X$ defined by the completion of the space of locally Lipschitz functions on $U$, denoted by $\mathrm{LIP}_{\mathrm{loc}}(U)$, with respect to the norm
\[||f||_{H^{1, p}}:=\left( ||f||_{L^p}^p+||\nabla f||_{L^p}^p\right)^{1/p}.\]
See \cite[Theorem 4.9]{holp}. 
\item[(2.4)]{(Closedness of the Dirichlet Laplacian)} Let $f_i$ be a sequence in $\mathcal{D}^2(\Delta, X_i)$ with 
\[\sup_i \left( ||f_i||_{L^2}+||\Delta f_i||_{L^2}\right)<\infty,\]
and let $f$ be the $L^2$-weak limit of $f_i$ on $X$. Then we see that $f \in \mathcal{D}^2(\Delta, X)$, that $f_i, \nabla f_i$  $L^2$-converge strongly to $f, \nabla f$ on $X$, respectively and  that $\Delta f_i, \mathrm{Hess}_{f_i}$ $L^2$-converge weakly to $\Delta f, \mathrm{Hess}_f$ on $X$, respectively. See \cite[Theorem $1.3$]{holp}.
\item[(2.5)]{(Equivalence of $L^q$-strong convergence)} For  any sequence $T_i \in L^p(T^r_sB_R(x_i))$ with 
\[\sup_i||T_i||_{L^p}<\infty,\]
and $T \in L^p(T^r_sB_R(x))$, we see that $T_i$ $L^q$-converges strongly to $T$ on $B_R(x)$ for some $q \in (1, p)$ if and only if  $T_i$ $L^q$-converges strongly to $T$ on $B_R(x)$ for every $q \in (1, p)$.
See \cite[Proposition $3.3$]{hoell}.
\item[(2.6)]{(Existence of approximate sequence)} For every $f \in \mathrm{LIP}(X)$, there exists a sequence $f_i \in \mathrm{LIP}(X_i)$ with 
\[\sup_i||\nabla f_i||_{L^{\infty}}<\infty\]
such that $f_i, \nabla f_i$ $L^2$-converge strongly to $f, \nabla f$ on $X$, respectively.
See \cite[Theorem $4.2$]{holip}.
\item[(2.7)]{(Sobolev inequality)} For every $f \in H^{1, 2}(X)$, we have
\[\left(\frac{1}{H^n(X)}\int_X\left| f- \frac{1}{H^n(X)}\int_XfdH^n\right|^{2n/(n-2)}dH^n\right)^{(n-2)/(2n)}\le \frac{C(n, K_1, d)}{H^n(X)}\int_X|\nabla f|^2dH^n.\] 
See \cite[The\'eor\`eme 1.1]{mas} and \cite[Theorem 3.2]{hoell}.
\item[(2.8)]{(Continuity of solutions of Poisson's equations with respect to Gromov-Hausdorff topology)} Let $g_i \in L^2(X_i)$ be a sequence with
\[\int_{X_i}g_idH^n=0,\]
and let $g \in L^2(X)$ be the $L^2$-weak limit on $X$.
Then 
there exists a unique $f \in \mathcal{D}^2(\Delta, X)$ with 
\[\int_XfdH^n=0\]
such that $f=\Delta^{-1}g$, i.e. $\Delta f=g$ holds, and that $\Delta^{-1}g_i, \nabla \Delta^{-1}g_i$ $L^2$-converge strongly to $f, \nabla f$ on $X$, respectively.
See \cite[Theorem 1.1]{hoell}.
\item[(2.9)]{(Compactness of $L^p$-weak convergence)}
For every sequence $T_i \in L^p(T^r_sB_R(x_i))$ with
\[\sup_i||T_i||_{L^p}<\infty,\]
there exist a subsequence $i(j)$ and $T \in L^p(T^r_sB_R(x))$ such that $T_{i(j)}$ $L^p$-converges weakly to $T$ on $B_R(x)$.
See \cite[Proposition 3.50]{holp}.
\item[(2.10)]{(Lower semicontinuity of norms with respect to $L^p$-weak convergence)} 
Let $T_i$ be a sequence in $L^p(T^r_sB_R(x_i))$ and let $T \in L^p(T^r_sB_R(x))$ be the $L^p$-weak limit on $B_R(x)$.
Then
\[\liminf_{i \to \infty}||T_i||_{L^p} \ge ||T||_{L^p}.\]
See \cite[Proposition 3.64]{holp}.
\end{enumerate}
We first prove the following.
\begin{theorem}\label{8}
let $f_i$ be a sequence in $C^{\infty}(X_i)$ with
\[\sup_{i<\infty} \left(||\Delta f_i||_{L^{\infty}} + ||\Delta f_i||_{H^{1, 2}}\right)<\infty,\]
and let $f \in L^2(X)$ be the $L^2$-strong limit on $X$.
Then we see that $f \in \mathcal{D}^2(\Delta, X) \cap \mathrm{LIP}(X)$ and that $\mathrm{Hess}_{f_i}$ $L^{2q}$-converges strongly to $\mathrm{Hess}_{f}$ on $X$ for every $q<2$.
\end{theorem}
\begin{proof}
From (2.1) - (2.4) and (2.10), we see that 
 $f \in \mathcal{D}^2(\Delta, X) \cap \mathrm{LIP}(X)$, 
that $\Delta f \in H^{1, 2}(X) \cap L^{\infty}(X)$, that $f_i, \nabla f_i, \Delta f_i$ $L^2$-converge strongly to $f, \nabla f, \Delta f$ on $X$, respectively and that
$\nabla \Delta f_i, \mathrm{Hess}_{f_i}$ $L^2$-converge weakly to $\nabla \Delta f, \mathrm{Hess}_f$ on $X$, respectively.

In particular  it suffices to check that $\mathrm{Hess}_{f_i}$ $L^{2q}$-converges strongly to $\mathrm{Hess}_f$ on $X$ for every $q<2$.
Let 
\begin{align}
L:=\sup_{i<\infty} \left(||\Delta f_i||_{L^{\infty}} + ||\Delta f_i||_{H^{1, 2}}\right)<\infty.
\end{align}
\begin{lemma}\label{3}
We have the following.
\begin{enumerate}
\item $|\nabla f|^2 \in H^{1, 2q}(X)$ for every $q<2$.
\item $\nabla |\nabla f_i|^2$ $L^{2q}$-converges strongly to $\nabla |\nabla f|^2$ on $X$ for every $q<2$.
\end{enumerate}
\end{lemma}
The proof is as follows.

Recall that Bochner's formula states
\[
-\frac{1}{2}\Delta |\nabla f_i|^2 = |\mathrm{Hess}_{f_i}|^2+ \langle \nabla \Delta f_i, \nabla f_i \rangle + \mathrm{Ric}_{X_i}(\nabla f_i, \nabla f_i).\]
The Lipschitz regularity of solutions of Poisson's equation (2.2) gives
\begin{align*}
\int_{X_i}|\Delta |\nabla f_i|^2|^qdH^n \le 10\int_{X_i}|\mathrm{Hess}_{f_i}|^{2q}dH^n+C(n, d, L, K_1, K_2)
\end{align*}
for any $i<\infty$ and $q \in [1, 2)$.
By the $L^{2q}$-estimate of a Hessian (\ref{Lpboundhess}), we have
\begin{align}\label{chnabound}
\int_{X_i}|\mathrm{Hess}_{f_i}|^{2q}dH^n \le C(n, d, v, K_1, K_2, L, q).
\end{align}
In particular we have 
\begin{align}\label{1}
\int_{X_i}|\Delta |\nabla f_i|^2|^{q}dH^n + \int_{X_i}|\nabla |\nabla f_i|^2|^{2q}dH^n \le C(n, d, v, K_1, K_2, L, q),
\end{align}
where we used the inequalty
\[\left|\nabla |\nabla f_i|^2\right|\le 2|\nabla f_i||\mathrm{Hess}_{f_i}|\le C(n, K_1, d, v, L)|\mathrm{Hess}_{f_i}|.\]
Let $g_i:=|\nabla f_i|^2$ and let $g:=|\nabla f|^2$.
Note that the equivalence of $L^q$-strong convergence (2.5) and the $L^2$-strong convergence of $\nabla f_i$ to $\nabla f$ with $\sup_i\|\nabla f_i\|_{L^{\infty}(X_i)}<\infty$ yield the $L^p$-strong convergence of $g_i$ to $g$ for every $p \in (1, \infty)$.
In particular since (\ref{1}) yields $\sup_i\| g_i\|_{H^{1, 2q}}<\infty$, applying the Rellich compactness theorem (2.3) for $g_i$ proves (i).

By the existence of an approximate sequence (2.6), there exists a sequence $\hat{g}_i \in H^{1, 2}(X_i)$ such that $\hat{g}_i, \nabla \hat{g}_i$ $L^2$-converge strongly to $g, \nabla g$ on $X$, respectively.
It follows from a direct calculation that
\begin{align*}
\int_{X_i}|\nabla \hat{g}_i|^2dH^n-2\int_{X_i}(\Delta g_i)\hat{g}_idH^n &= \int_{X_i}|\nabla g_i|^2dH^n-2\int_{X_i}(\Delta g_i)g_idH^n \nonumber \\
& + \int_{X_i}|\nabla (g_i-\hat{g}_i)|^2dH^n.
\end{align*}
In particular
\begin{align}\label{2}
\int_{X_i}|\nabla \hat{g}_i|^2dH^n-2\int_{X_i}(\Delta g_i)\hat{g}_idH^n \ge \int_{X_i}|\nabla g_i|^2dH^n-2\int_{X_i}(\Delta g_i)g_idH^n.
\end{align}

On the other hand the equivalence of $L^q$-strong convergence (2.5) and the Sobolev inequality (2.7) give
\begin{align*}
&\left|\int_{X_i}(\Delta g_i)\hat{g}_idH^n- \int_{X_i}(\Delta g_i)g_idH^n\right|\\
&\le \left(\int_{X_i}|(\Delta g_i)|^{q}dH^n\right)^{1/q}\left( \int_{X_i}|\hat{g}_i-g_i|^{p}dH^n\right)^{1/p} \to 0
\end{align*}
as $i \to \infty$ for every $q<2$ such that the conjugate exponent $p$ of $q$ satisfies $p<2n/(n-2)$.
Hence letting $i \to \infty$ in (\ref{2}) gives
\begin{align*}
\limsup_{i \to \infty} \int_{X_i}|\nabla g_i|^2dH^n \le \int_{X}|\nabla g|^2dH^n.
\end{align*}
This completes the proof of Lemma \ref{3}.

The next claim is a direct consequence of Lemma \ref{3} and the expression:
\begin{align}
\mathrm{Hess}_f(\nabla g, \nabla h)=\frac{1}{2}\left(\langle \nabla g, \nabla \langle \nabla f, \nabla h \rangle \rangle + 
\langle \nabla h, \nabla \langle \nabla f, \nabla g \rangle \rangle - \langle \nabla f, \nabla \langle \nabla g, \nabla h \rangle \rangle \right).
\end{align}
\begin{claim}\label{10}
Let $g_i, h_i$ be sequences in $C^{\infty}(X_i)$ with
\begin{align}\label{0927}
\sup_{i<\infty}\left( ||\Delta g_i||_{L^{\infty}} + ||\Delta g_i||_{H^{1, 2}} + ||\Delta h_i||_{L^{\infty}} + ||\Delta h_i||_{H^{1, 2}} \right)<\infty,
\end{align}
and let $g, h \in \mathrm{LIP}(X)$ be the $L^2$-strong limits of them on $X$, respectively.
Then $\mathrm{Hess}_{f_i}(\nabla g_i, \nabla h_i)$ $L^2$-converges strongly to  $\mathrm{Hess}_{f}(\nabla g, \nabla h)$ on $X$.
\end{claim}
The following is a direct consequence of the regularity theory of the heat flow from \cite{ags} (thus, it holds on an $RCD$-space).
\begin{claim}\label{4}
Let $L_0>0$ and let $f \in \mathcal{D}^2(\Delta, X) \cap \mathrm{LIP}(X)$ with
\[||\Delta^{\upsilon}f||_{L^{\infty}} \le L_0.\]
Then for every $0<s<1$ we have the following.
\begin{enumerate}
\item $H_sf \in \mathrm{LIP}(X) \cap \mathcal{D}^2(\Delta, X)$ with $||H_sf-f||_{H^{1, 2}}\le \Psi (s; n, L_0)$.
\item $||\Delta(H_sf)||_{L^{\infty}}+||\nabla (H_sf)||_{L^{\infty}} \le C(n, K_1, L_0)$.
\item $\Delta(H_sf) \in \mathrm{LIP}(X)$ with $||\nabla \Delta(H_sf)||_{L^{\infty}} \le C(n, K_1, s, L_0)$.
\end{enumerate}
where $H_s$ is the heat flow and $\Psi (\epsilon_1, \ldots, \epsilon_k; c_1, \ldots, c_l)$ denotes some positively valued function on $\mathbf{R}^k_{\>0} \times \mathbf{R}^l$ satisfying
\[\lim_{\epsilon_1, \ldots, \epsilon_k \to 0}\Psi (\epsilon_1, \ldots, \epsilon_k; c_1, \ldots, c_l)=0\]
for fixed real numbers $c_1, \ldots, c_l$.
In particular for every $p \in (1, \infty)$ we have 
\begin{align*}
||H_sf-f||_{H^{1, p}}\le \Psi (s; n, K_1, L_0, p).
\end{align*}
\end{claim}
We give a proof of (i) only for reader's convenience.
We have
\begin{align*}
\int_X|H_sf-f|^2dH^n&=\int_X\left| \int_0^s\Delta (H_tf)dt\right|^2dH^n\\
&\le \int_Xs\int_0^s(\Delta (H_tf))^2dtdH^n \\
&= s\int_0^s\int_X(\Delta (H_tf))^2dH^ndt \\
&= s\int_0^s\int_X(H_t(\Delta f))^2dH^ndt \\
&\le s\int_0^s \int_X(\Delta f)^2dH^ndt \\
&=s^2||\Delta f||_{L^2}^2.
\end{align*}

On the other hand
\begin{align*}
\int_X|\nabla (H_sf-f)|^2dH^n &=\int_X(H_sf-f)\Delta (H_sf-f)dH^n \\
&\le ||H_sf-f||_{L^2} \times ||\Delta (H_sf-f)||_{L^2} \\
&\le ||H_sf-f||_{L^2} \times 2||\Delta f||_{L^2}.
\end{align*}
From \cite[Theorems $6.1$ and $6.2$]{ags} with these estimates, this completes the proof of (i).

\begin{claim}\label{5}
In Claim \ref{10}, the same conclusion holds even if we relapce the assumption (\ref{0927}) by 
\begin{align}\label{9999990}
L_1:=\sup_{i<\infty}\left( ||\Delta g_i||_{L^{\infty}} + ||\Delta h_i||_{L^{\infty}}\right) < \infty.
\end{align}
\end{claim}
The proof is as follows.

By Claims \ref{10} and \ref{4} we see that $\mathrm{Hess}_{f_i}(\nabla (H_sg_i), \nabla (H_sh_i))$ $L^2$-converges strongly to $\mathrm{Hess}_{f}(\nabla (H_sg), \nabla (H_sh))$ on $X$ for every $s \in (0, 1)$.
On the other hand Claim \ref{4} with (\ref{chnabound}) gives
\begin{align*}
&\int_{X_i}|\mathrm{Hess}_{f_i}(\nabla g_i, \nabla h_i)-\mathrm{Hess}_{f_i}(\nabla (H_sg_i), \nabla (H_sh_i))|^2dH^n\\
&\le \int_{X_i}\left|\mathrm{Hess}_{f_i}\right|^2\left|\nabla g_i \otimes \nabla h_i- \nabla (H_sg_i) \otimes \nabla (H_sh_i)\right|^2dH^n\\
&\le \left( \int_{X_i}\left|\mathrm{Hess}_{f_i}\right|^{2q}dH^n\right)^{1/q}\left( \int_{X_i}\left|\nabla g_i \otimes \nabla h_i- \nabla (H_sg_i) \otimes \nabla (H_sh_i)\right|^{2p}dH^n\right)^{1/p} \\
&\le \Psi(s ; n, d, v, K_1, K_2, L, L_1, q)
\end{align*}
for any $i \le \infty$ and $q \in (1, 2)$, where $p$ is the conjugate exponent of $q$. 
By letting $i \to \infty$ and $s \to 0$,
this completes the proof of Claim \ref{5}.

\begin{claim}\label{6}
Let $g_i$ be a sequence as in Claim \ref{5} (i.e. (\ref{9999990}) is satisfied),  let $R>0$, let $x \in X$, let $x_i \in X_i$ with $x_i \stackrel{GH}{\to} x$, let $k_i$ be a sequence in $C^{\infty}(B_R(x_i))$ with
\begin{align*}
L_2:=\sup_{i<\infty}\left(  ||k_i||_{L^{\infty}(B_R(x_i))} + ||\nabla k_i||_{L^{\infty}(B_R(x_i))} + ||\Delta k_i||_{L^{\infty}(B_R(x_i))} \right)<\infty,
\end{align*}
and let $k \in H^{1, 2}(B_R(x))$ be the $L^2$-strong limit on $B_R(x)$.
Then $\mathrm{Hess}_{f_i}(\nabla k_i, \nabla g_i)$ $L^2$-converges strongly to $\mathrm{Hess}_{f}(\nabla k, \nabla g)$ on $B_R(x)$.
\end{claim}
The proof is as follows.

Fix $\epsilon>0$.
The existence of good cut-off functions constructed in \cite[Theorem $6.33$]{ch-co}  by Cheeger-Colding states that for every $i<\infty$ there exists $\phi_i \in C^{\infty}(X_i)$ such that the following four conditions hold:
\begin{enumerate}
\item $0 \le \phi_i \le 1$.
\item $\phi_i|_{B_{R-2\epsilon}(x_i)} \equiv 1$.
\item $\mathrm{supp}\,\phi_i \subset B_{R-\epsilon}(x_i)$.
\item $|\nabla \phi_i|+|\Delta \phi_i| \le C(n, \epsilon, K_1, R)$.
\end{enumerate}
By the closedness of the Dirichlet Laplacian (2.4), without loss of generality we can assume that the $L^2$-strong limit $\phi \in \mathrm{LIP}(X) \cap \mathcal{D}^2(\Delta, X)$ of $\phi_i$ on $X$ exists.

Then applying
Claim \ref{5} for $h_i:=\phi_ik_i$ gives that $\mathrm{Hess}_{f_i}(\nabla (\phi_ik_i), \nabla g_i)$ $L^2$-converges strongly to $\mathrm{Hess}_{f}(\nabla (\phi k), \nabla g)$ on $X$ (c.f. \cite[Theorem 4.5]{hoell}).
On the other hand applying Claim \ref{5} for $h_i:=\phi_i$ yields that $k_i\mathrm{Hess}_{f_i}(\nabla \phi_i, \nabla g_i)$ $L^2$-converges strongly to $k\mathrm{Hess}_{f}(\nabla \phi, \nabla g)$ on $B_R(x)$.
Since 
\begin{align*}
\mathrm{Hess}_{f_i}(\nabla (\phi_ik_i), \nabla g_i)=\phi_i\mathrm{Hess}_{f_i}(\nabla k_i, \nabla g_i)+k_i\mathrm{Hess}_{f_i}(\nabla \phi_i, \nabla g_i),
\end{align*}
we see that $\phi_i\mathrm{Hess}_{f_i}(\nabla k_i, \nabla g_i)$ $L^2$-converges strongly to $\phi \mathrm{Hess}_{f}(\nabla k, \nabla g)$ on $B_R(x)$.

On the other hand for every $q<2$ the H$\ddot{\text{o}}$lder inequality yields
\begin{align*}
&\int_{B_R(x_i)}|\mathrm{Hess}_{f_i}(\nabla k_i, \nabla g_i)-\phi_i\mathrm{Hess}_{f_i}(\nabla k_i, \nabla g_i)|^2dH^n\\
&\le \int_{B_R(x_i) \setminus B_{R-2\epsilon}(x_i)}|\mathrm{Hess}_{f_i}|^2(1-\phi_i)^2|\nabla k_i|^2|\nabla g_i|^2dH^n\\
&\le (L_1L_2)^2\int_{B_R(x_i) \setminus B_{R-2\epsilon}(x_i)}|\mathrm{Hess}_{f_i}|^2dH^n\\
&\le (L_1L_2)^2(H^n(B_R(x_i) \setminus B_{R-2\epsilon}(x_i)))^{1/p}\left( \int_{X_i}|\mathrm{Hess}_{f_i}|^{2q}dH^n\right)^{1/q}\\
&\le \Psi(\epsilon; n, d, v, R, L, L_1, L_2, q),
\end{align*}
where $p$ is the conjugate exponent of $q$.
This completes the proof of Claim \ref{6}.

\begin{claim}\label{7}
Let $k_i$ be a sequence as in Claim \ref{6}, let $e_i$ be a sequence in $C^{\infty}(B_R(x_i))$ with
\begin{align}
\sup_{i<\infty}\left( ||e_i||_{L^{\infty}(B_R(x_i))} + ||\nabla e_i||_{L^{\infty}(B_R(x_i))} + ||\Delta e_i||_{L^{\infty}(B_R(x_i))}\right) < \infty,
\end{align}
and let $e \in H^{1, 2}(B_R(x))$ be the $L^2$-strong limit on $B_R(x)$.
Then $\mathrm{Hess}_{f_i}(\nabla k_i, \nabla e_i)$ $L^2$-converges strongly to $\mathrm{Hess}_{f}(\nabla k, \nabla e)$ on $B_R(x)$.
\end{claim}
The proof is as follows (and is essentially same to that of Claim \ref{6}).

We will use the same notation as in the proof of Claim \ref{6}.
Claim \ref{6} yields that $\mathrm{Hess}_{f_i}(\nabla k_i, \nabla (\phi_ie_i))$ $L^2$-converges strongly to $\mathrm{Hess}_{f}(\nabla k, \nabla (\phi e))$ on $B_R(x)$ and that
$e_i\mathrm{Hess}_{f_i}(\nabla k_i, \nabla \phi_i)$ $L^2$-converges strongly to $e \mathrm{Hess}_{f}(\nabla k, \nabla \phi)$ on $B_R(x)$.
Since
\begin{align*}
\mathrm{Hess}_{f_i}(\nabla k_i, \nabla (\phi_ie_i))= e_i\mathrm{Hess}_{f_i}(\nabla k_i, \nabla \phi_i)+\phi_i \mathrm{Hess}_{f_i}(\nabla k_i, \nabla e_i),
\end{align*} 
we see that $\phi_i \mathrm{Hess}_{f_i}(\nabla k_i, \nabla e_i)$ $L^2$-converges strongly to $\phi \mathrm{Hess}_{f}(\nabla k, \nabla e)$ on $B_R(x)$.
Then by an argument similar to the proof of Claim \ref{6}, we have Claim \ref{7}.

We need a previous result given in \cite{hoell}.
\begin{claim}{\cite[Definition $3.6$ and Proposition $3.8$]{hoell}}\label{mmmmm}
Let $p>1$, let $F_i$ be a sequence in $L^p(X_i)$ with
\[\sup_i||F_i||_{L^p}<\infty,\]
and let $F \in L^p(X)$.
Assume that $\{F_i\}_i$ is $L^1_{\mathrm{loc}}$-weakly upper semicontinuous at a.e. $x \in X$, i.e. there exists a Borel subset $A$ of $X$ such that the following two conditions hold.
\begin{enumerate}
\item $H^n(X \setminus A)=0$.
\item For any $\epsilon>0$, $z \in A$ and subsequence $i(j)$, there exist a subsequence $j(k)$ of $i(j)$, $r_0>0$ and a convergent sequence $z_{j(k)} \stackrel{GH}{\to} z$ such that
\[\limsup_{k \to \infty}\left(\frac{1}{H^n(B_r(z_{j(k)}))}\int_{B_r(z_{j(k)})}F_{j(k)}dH^n\right)-\frac{1}{H^n(B_r(z))}\int_{B_r(z)}FdH^n \le \epsilon\]
for every $r<r_0$.
\end{enumerate}
Then
\[\limsup_{i \to \infty}\int_{X_i}f_idH^n\le \int_XfdH^n.\]
\end{claim}

We are now in a position to finish the proof of Theorem \ref{8}.

By an argument similar to the proof of \cite[Theorem $6.7$]{hoell} (more precisely, see page $64-65$ in \cite{hoell}) with Claim \ref{7} we see that 
$\{|\mathrm{Hess}_{f_i}|^2\}_i$ is $L^1_{\mathrm{loc}}$-weakly upper semicontinuous at a.e. $x \in X$.
Thus Claim \ref{mmmmm} with (\ref{chnabound}) yields
\begin{align}\label{hesskey}
\limsup_{i \to \infty}\int_{X_i}|\mathrm{Hess}_{f_i}|^2dH^n\le \int_{X}|\mathrm{Hess}_{f}|^2dH^n.
\end{align}
Thus $\mathrm{Hess}_{f_i}$ $L^2$-converges strongly to $\mathrm{Hess}_{f}$ on $X$.
By the equivalence of $L^q$-strong convergence (2.5), this completes the proof of Theorem \ref{8}.
\end{proof}
Theorem \ref{8} also holds even if each $X_i$ is in $\overline{\mathcal{M}}$:
\begin{corollary}\label{ppm}
Let $Y_i \stackrel{GH}{\to} Y$ in $\overline{\mathcal{M}}$, let $f_i$ be a sequence in $\mathcal{D}^2(\Delta, Y_i)$ with  $\Delta f_i \in L^{\infty}(Y_i) \cap H^{1, 2}(Y_i)$ such that
\[\sup_i\left( ||\Delta f_i||_{L^{\infty}}+||\Delta f_i||_{H^{1,2}}\right)<\infty,\]
and let $f \in L^2(Y)$ be the $L^2$-strong limit on $Y$.
Then we see that $f \in \mathcal{D}^2(\Delta, Y) \cap \mathrm{LIP}(Y)$ and that $\mathrm{Hess}_{f_i}$ $L^{2q}$-converges strongly to $\mathrm{Hess}_f$ on $Y$ for every $q \in (1, 2)$.
\end{corollary}
\begin{proof}
By the Lipschitz regularity of solutions of Poisson's equations (2.2), the closedness of the Dirichlet Laplacian (2.4) and the lower semicontinuity of norms with respect to $L^p$-weak convergence (2.10), we see that $f \in \mathcal{D}^2(\Delta, Y) \cap \mathrm{LIP}(Y)$, that $f_i, \Delta f_i$ $L^2$-converge strongly to $f, \Delta f$ on $Y$, respectively and that $\mathrm{Hess}_{f_i}$ $L^2$-converges weakly to $\mathrm{Hess}_f$ on $Y$.
Therefore it suffices to check that
\begin{align}\label{8833}
\lim_{i \to \infty}\int_{Y_i}|\mathrm{Hess}_{f_i}|^{2q}dH^n= \int_Y|\mathrm{Hess}_f|^{2q}dH^n
\end{align}
for every $q \in (1, 2)$.
For every $i$ let $Y_{i, j}$ be a sequence in $\mathcal{M}$ such that $Y_{i, j} \stackrel{GH}{\to} Y_i$. 
From the existence of an approximation sequence (2.6), there exist sequences $g_{i, j} \in H^{1, 2}(Y_j)$ with 
\[\int_{Y_j}g_{i, j}dH^n=0\]
such that $g_{i, j}, \nabla g_{i, j}$ $L^2$-converge strongly to $\Delta f_i, \nabla \Delta f_i$ on $Y$, respectively.
Without loss of generality we can assume that 
\[\sup_{j, j}||g_{i, j}||_{L^{\infty}}<\infty\]
because the sequence $\hat{g}_{i, j} \in H^{1, 2}(Y_j) \cap L^{\infty}(Y_j)$ defined by
\begin{align*}
\hat{g}_{i, j}(x):=
\begin{cases} ||\Delta f_i||_{L^{\infty}} \,\,\,\,\,\mathrm{if}\,g_{i, j}(x) \ge ||\Delta f_i||_{L^{\infty}}, \\
g_{i, j}(x) \,\,\,\,\,\mathrm{if}\, |g_{i, j}(x)| <||\Delta f_i||_{L^{\infty}}, \\
-||\Delta f_i||_{L^{\infty}} \,\,\,\,\,\mathrm{if}\, g_{i, j}(x) \le -||\Delta f_i||_{L^{\infty}},
\end{cases}
\end{align*}
satisfies that $\hat{g}_{i, j}, \nabla \hat{g}_{i, j}$ $L^2$-converge strongly to $\Delta f_i, \nabla \Delta f_i$ on $Y$, respectively.

Moreover by the smoothing via the heat flow (c.f. \cite{ags, grigo}), we can assume that  $g_{i, j} \in C^{\infty}(Y_{i, j})$.
Let $f_{i, j}:=\Delta^{-1}g_{i, j} \in C^{\infty}(Y_{i, j})$, where the smoothness of $f_{i, j}$ follows from the elliptic regurality theorem.
Note that by the continuity of solutions of Poisson's equations (2.8), $f_{i, j}, \nabla f_{i, j}$ $L^2$-converge strongly to $f_i, \nabla f_i$ on $Y_i$, respectively. 
Then Theorem \ref{8} yields that
$\mathrm{Hess}_{f_{i, j}}$ $L^{2q}$-converges strongly to $\mathrm{Hess}_{f_i}$ on $Y_i$
for every $q \in (1, 2)$.
Thus
\begin{align*}
\lim_{j \to \infty}\int_{Y_{i, j}}|\mathrm{Hess}_{f_{i, j}}|^{2q}dH^n=\int_{Y_i}|\mathrm{Hess}_{f_i}|^{2q}dH^n.
\end{align*}
Therefore there exists a subsequence $j(i)$ with 
\[\sup_i\left( ||\Delta f_{i, j(i)}||_{L^{\infty}} +||\Delta f_{i, j(i)}||_{H^{1, 2}} \right)<\infty\]
such that $f_{i, j(i)}, \Delta f_{i, j(i)}$ $L^2$ converge strongly to $f, \Delta f$ on $Y$, respectively and that
\[\lim_{i \to \infty}\left| \int_{Y_{i, j(i)}}|\mathrm{Hess}_{f_{i, j(i)}}|^{2q}dH^n-\int_{Y_i}|\mathrm{Hess}_{f_i}|^{2q}dH^n\right|=0.\]
Thus since Theorem \ref{8} yields
\begin{align*}
\lim_{i \to \infty}\int_{Y_{i, j(i)}}|\mathrm{Hess}_{f_{i, j(i)}}|^{2q}dH^n=\int_{Y}|\mathrm{Hess}_{f}|^{2q}dH^n,
\end{align*}
we have
\[\lim_{i \to \infty}\int_{Y_i}|\mathrm{Hess}_{f_i}|^{2q}dH^n=\int_{Y}|\mathrm{Hess}_{f}|^{2q}dH^n.\]
This completes the proof.
\end{proof}
The following is the main result in this section.
\begin{theorem}[$L^2$-strong convergence of Hessians]\label{L2Hess}
Let $Y_i \stackrel{GH}{\to} Y$ in $\overline{\mathcal{M}}$, let $f_i$ be a sequence in $\mathcal{D}^2(\Delta, Y_i)$, and let $f \in \mathcal{D}^2(\Delta, Y)$ be the $L^2$-strong limit of them on $Y$.
Assume that $\Delta f$ is the $L^2$-strong limit of $\Delta f_i$ on $Y$.
Then we see that $\mathrm{Hess}_{f_i}$ $L^2$-converges strongly to $\mathrm{Hess}_{f}$ on $Y$.
\end{theorem}
\begin{proof}
By an argument similar to the proof of Corollary \ref{ppm}, without loss of generality we can assume that $Y_i \in \mathcal{M}$ and that $f_i \in C^{\infty}(Y_i)$.

Let  $g_{i}$ be a sequence in $\mathrm{LIP}(Y)$ with
\[\int_{Y}g_{i}dH^n=0\]
such that $g_{i} \to \Delta f$ in $L^2(Y)$

By the existence of an approximate sequence (2.6) and an argument similar to the proof of Corollary \ref{ppm}, there exists a sequence of $g_{i, j} \in C^{\infty}(Y_j)$ with
\[\sup_j||\nabla g_{i, j}||_{L^{\infty}}<\infty\]
such that $g_{i, j}, \nabla g_{i, j}$ $L^2$-converge strongly to $g_{i}, \nabla g_{i}$ on $Y$, respectively  and that 
\[\int_{Y_j}g_{i, j}dH^n=0.\]

Let $f_{i, j}:=(\Delta)^{-1}g_{i, j} \in C^{\infty}(Y_j)$.
Then Theorem \ref{8} with the continuity of solutions of Poisson's equations (2.8) yields that $\mathrm{Hess}_{f_{i, j}}$ $L^2$-converges strongly to $\mathrm{Hess}_{\Delta^{-1}g_{i}}$ on $Y$.
Integrating the Bochner inequality: 
\[-\frac{1}{2}\Delta |\nabla (f_{i, j}-f_j)|^2 \ge |\mathrm{Hess}_{f_{i, j}-f_j}|^2 -\langle \nabla \Delta (f_{i, j}-f_j), \nabla (f_{i, j}-f_j)\rangle +K_1|\nabla (f_{i, j}-f_j)|^2\]
gives
\[\int_{Y_j}|\mathrm{Hess}_{f_{i, j}-f_j}|^2dH^n \le \int_{Y_j}\left(\Delta (f_{i, j}-f_j))^2 - K_1|\nabla (f_{i, j}-f_j)|^2\right)dH^n.\]
Thus  letting $j \to \infty$ and $i \to \infty$ yields
\begin{align}\label{bg}
&\lim_{i \to \infty}\left(\limsup_{j \to \infty}\int_{Y_j}|\mathrm{Hess}_{f_{i, j}-f_i}|^2dH^n\right) \nonumber \\
&\le \lim_{i \to \infty}\left( \int_{Y}\left(\Delta (\Delta^{-1}g_{i}-f))^2 - K_1|\nabla (\Delta^{-1}g_{i}-f)|^2\right)dH^n\right)=0. 
\end{align}
On the other hand, Gigli's Bochner inequality on an $RCD$-space \cite[Corollary $3.3.9$]{gigli} yields
\[\int_{Y}|\mathrm{Hess}_{\Delta^{-1}g_{i}-f}|^2dH^n \le \int_{Y}\left(\Delta (\Delta^{-1}g_{i}-f))^2 + K_1|\nabla (\Delta^{-1}g_{i}-f)|^2\right)dH^n.\]
In particular
\begin{align}\label{bg1}
\lim_{i \to \infty}\int_Y|\mathrm{Hess}_{\Delta^{-1}g_i}|^2dH^n = \int_Y|\mathrm{Hess}_f|^2dH^n.
\end{align}
(\ref{bg}) and (\ref{bg1}) yield
\[\limsup_{i \to \infty}\int_{Y_i}|\mathrm{Hess}_{f_i}|^2dH^n\le \int_Y|\mathrm{Hess}_f|^2dH^n.\]
This completes the proof. 
\end{proof}
\begin{corollary}\label{nn}
Let $Y_i \stackrel{GH}{\to} Y$ in $\overline{\mathcal{M}}$, let $f_i, g_i$ be sequences in $\mathcal{D}^2(\Delta, Y_i)$ with
\[\sup_i\left( ||\nabla f_i||_{L^{\infty}} + ||\nabla g_i||_{L^{\infty}}\right)<\infty,\]
and let $f, g \in \mathcal{D}^2(\Delta, Y)$ be the $L^2$-strong limits of them on $Y$, respectively.
Assume that $\Delta f_i, \Delta g_i$ $L^2$-converge strongly to $\Delta f, \Delta g$ on $Y$, respectively.
Then $[\nabla f_i, \nabla g_i]$ $L^2$-converges strongly to $[\nabla f, \nabla g]$ on $Y$. 
\end{corollary}
\begin{proof}
Since
\[[\nabla f_i, \nabla g_i]^*=\mathrm{Hess}_{g_i}(\nabla f_i, \cdot) - \mathrm{Hess}_{f_i}(\nabla g_i, \cdot),\]
Corollary \ref{nn} is a direct consequence of Theorem \ref{L2Hess} and \cite[Proposition $3.70$]{holp}.
\end{proof}
\section{Curvature of limit spaces}
We first recall two approximation theorems by new test functions given in \cite{hoell}.

Let $X \in \overline{\mathcal{M}}$ and let 
\begin{align}\label{3wwc}
\widetilde{\mathrm{Test}}F(X):=\{f \in \mathcal{D}^2(\Delta, X); \Delta f \in \mathrm{LIP}(X)\}.
\end{align}
Note that the Lipschitz regularity of solutions of Poisson's equations (2.2) yields
\[\widetilde{\mathrm{Test}}F(X) \subset \mathrm{Test}F(X)\]
(recall (\ref{testdefini}) for the definition of $\mathrm{Test}F(X)$).
Then the following were proven in \cite{hoell}.
\begin{enumerate}
\item[(3.1)]{(Existence of $H^{2, 2}$-approximation)} For every $f \in \mathrm{LIP}(X) \cap \mathcal{D}^2(\Delta, X)$, there exists a sequence $f_i \in \widetilde{\mathrm{Test}}F(X)$ with 
\[\sup_i||\nabla f_i||_{L^{\infty}}<\infty\]
such that $f_i \to f$ in $H^{1, 2}(X)$ and that $\Delta f_i \to \Delta f$ in $L^2(X)$.
In particular
$\mathrm{Hess}_{f_i} \to \mathrm{Hess}_f$ in $L^2(T^0_2X)$.
Moreover if $f \in \mathrm{Test}F(X)$, then we can assume that $\Delta f_i \to \Delta f$ in $H^{1, 2}(X)$.
See \cite[Proposition $7.5$]{hoell}.
\item[(3.2)]{(Existence of $H^{1, p}$-approximation)} For every $f \in \mathrm{LIP}(X)$, there exists a sequence $f_i \in \widetilde{\mathrm{Test}}F(X)$ with
\[\sup_i||\nabla f_i||_{L^{\infty}}<\infty\]
such that $f_i \to f$ in $H^{1, p}(X)$ for every $p \in (1, \infty)$.
\end{enumerate}
Note that (3.2) essentially follows from (3.1).

For reader's convenience, we give the sketches of these proofs.

We first recall a key point to prove (3.1), which is considering the following approximation:
\[H_{\delta}(\widetilde{H}_{\epsilon}f),\]
where $\widetilde{H}_{\epsilon}$ is a mollified heat flow defined by
\begin{align}\label{mollified}
\widetilde{H}_{\epsilon}f:=\frac{1}{\epsilon}\int_0^{\infty}H_sf\phi (s \epsilon^{-1})ds,
\end{align}
and $\phi$ is a nonnegatively valued smooth function on $(0, 1)$ with
\[\int_0^1\phi (s)ds=1.\]
Then by the regularity theory of the heat flow in \cite[Theorems $6.1$ and $6.2$]{ags},  letting $\delta \to 0$ and $\epsilon \to 0$ for $H_{\delta}(\widetilde{H}_{\epsilon}f)$ gives (3.1). 

In order to get (3.2), we recall that 
\[\sup_{\epsilon<1}||\nabla \widetilde{H}_{\epsilon}f||_{L^{\infty}}<\infty \]
and that $\widetilde{H}_{\epsilon}f \to f$ in $H^{1, 2}(X)$ (see \cite[(3.2.3)]{gigli}).
Since $\widetilde{H}_{\epsilon}f \in \mathrm{LIP}(X) \cap \mathcal{D}^2(\Delta, X)$, applying (3.1) for  $\widetilde{H}_{\epsilon}f$ with the equivalence of $L^p$-strong convergence (2.5) yields (3.2).
\subsection{Riemannian curvature tensor}
We prove Theorem \ref{curvaturetensor}.

\textit{Proof of Theorem \ref{curvaturetensor}.}

We first prove the existence of $R_X$.
Let $X_i \in \mathcal{M}$ with $X_i \stackrel{GH}{\to} X$.
The $L^{q}$-bounds on $R_{X_i}$ in \cite[Theorem 1.8]{chna} by Cheeger-Naber yields  
\[\sup_{i}||R_{X_i}||_{L^q}<\infty\]
for every $q \in (1, 2)$.
Thus by the compactness of $L^q$-weak convergence (2.9), without loss of generality we can assume that there exists $T \in \bigcap_{q<2}L^q(T^4_0X)$ such that $R_{X_i}$ $L^q$-converges weakly to $T$ on $X$ for every $q \in (1, 2)$.
\begin{claim}\label{20}
$T$ satisfies (\ref{riem}) instead of $R_X$ if $f_i \in \widetilde{\mathrm{Test}}F(X)$ holds for every $i$.
\end{claim}
The proof is as follows.
Let $g_i:=\Delta f_i$.
By an argument similar to the proof of Corollary \ref{ppm}, there exist sequences $g_{i, j} \in C^{\infty}(X_j)$ with
\[\sup_{j}||\nabla g_{i, j}||_{L^{\infty}}<\infty\]
such that $g_{i, j}, \nabla g_{i, j}$ $L^2$-converge strongly to $g_i, \nabla g_i$ on $X$, respectively and that
\[\int_{X_j}g_{i, j}dH^n=0.\]
Let $f_{i, j}=\Delta^{-1}g_{i, j} \in C^{\infty}(X_j)$.
From the Lipschitz regurality and the continuity of solutions of Poisson's equations (2.2) and (2.8), we see that 
\[\sup_{j}||\nabla f_{i, j}||_{L^{\infty}}<\infty\]
 and that  $f_{i, j}, \nabla f_{i, j}$ $L^2$-converge strongly $f_i, \nabla f_i$ on $X$, repsectively.
Note that (\ref{riem}) gives
\begin{align}\label{aav}
&\int_{X_j}f_{0, j}\langle R_{X_j}, df_{1, j} \otimes df_{2, j}, \otimes df_{3, j} \otimes df_{4, j}\rangle dH^n \nonumber \\
&=\int_{X_j}\left(\left( -\langle \nabla f_{0, j}, \nabla f_{1, j}\rangle + f_{0, j}\Delta f_{1, j}\right)\mathrm{Hess}_{f_{3, j}}\left(\nabla f_{2, j}, \nabla f_{4, j}\right)\right)dH^n \nonumber \\
&\,\,\,\,\,-\int_{X_j}f_{0, j}\left\langle \mathrm{Hess}_{f_{3, j}}(\nabla f_{2, j}, \cdot), \mathrm{Hess}_{f_{4, j}}(\nabla f_{1, j}, \cdot )\right\rangle  dH^n\nonumber \\
&\,\,\,\,\,-\int_{X_j}\left(\left( -\langle \nabla f_{0, j}, \nabla f_{2, j}\rangle +f_{0, j}\Delta f_{2, j}\right)\mathrm{Hess}_{f_{3, j}}\left(\nabla f_{1, j}, \nabla f_{4, j}\right)\right)dH^n \nonumber \\
&\,\,\,\,\,+\int_{X_j}f_{0, j}\left\langle \mathrm{Hess}_{f_{3, j}}(\nabla f_{1, j}, \cdot), \mathrm{Hess}_{f_{4, j}}(\nabla f_{2, j}, \cdot )\right\rangle dH^n\nonumber \\
&\,\,\,\,\,-\int_{X_j}f_{0, j}\mathrm{Hess}_{f_{3, j}}\left( [\nabla f_{1, j}, \nabla f_{2, j}], \nabla f_{4, j}\right)dH^n \nonumber.
\end{align}
Thus letting $j \to \infty$ with Theorem \ref{L2Hess} and Corollary \ref{nn} yields Claim \ref{20}.

Claim \ref{20} with the existence of an $H^{2, 2}$-approximate sequence (3.1) and Corollary \ref{nn} yields that $T$ satisfies (\ref{riem})  instead of $R_X$ for any $f_i \in \mathrm{Test}F(X)$.
Thus we have the existence of $R_X$.

Next we prove the uniqueness of $R_X$.
For that it suffices to check that for any $q \in (1, \infty)$ and $R \in L^q(T^r_sX)$, if 
\[\int_X\langle R, T\rangle dH^n=0\]
for every $T \in \mathrm{Test}T^r_sX$, then $R=0$, where
\begin{align}\label{testtensor}
\mathrm{Test}T^r_sX:=\left\{\sum_{i=1}^Nf_{i, 0}\nabla f_{i, 1}\otimes \cdots \otimes \nabla f_{i, r}  \otimes df_{i, r+1} \otimes \cdots \otimes df_{i, r+s}; N \in \mathbf{N}, f_{i, j} \in \mathrm{Test}F(X)\right\}.
\end{align}

However this is a direct consequence of the following:
\begin{claim}\label{21}
$\mathrm{Test}T^r_sX$ is dense in $L^q(T^r_sX)$ for every $q \in (1, \infty)$.
\end{claim}
Claim \ref{21} follows directly from the regularity theory of the heat flow in \cite{ags}.
For reader's convenience, we give a proof in the case when $r=1, s=0$. 

Since it is easy to check that the space
\[\mathrm{LIP}(TX):=\left\{ \sum_{i=1}^Nf_i\nabla g_i; N \in \mathbf{N}, f_i, g_i \in \mathrm{LIP}(X)\right\}\]
is dense in $L^q(TX)$, it suffices to check that $\mathrm{Test}T^0_1X$ is dense in $\mathrm{LIP}(TX)$ with respect to the $L^q$-norm.
However this is a direct consequence of applying the existence of an $H^{1, p}$-approximate sequence (3.2) to $f_i, g_i$.
Thus we have Claim \ref{21}.

This completes the proof of Theorem \ref{curvaturetensor}.
$\,\,\,\,\,\,\,\Box$

\textit{Proof of Theorem \ref{riem2}.}

This is a direct consequence of the proof of Theorem \ref{curvaturetensor}.
$\,\,\,\,\,\,\,\Box$

We end this subsection by giving several results on $R_X$ which include the symmetries and the first Bianchi identity.

For any $V, W, Z \in L^{\infty}(TX)$, we define $R_X(V, W)Z \in \bigcap_{q<2}L^q(TX)$ by satisfying
\[\langle R_X(V, W)Z, Y\rangle =R_X(V, W, Z, Y)\]
on $X$ a.e. sense for every $Y \in L^{\infty}(TX)$. 
\begin{proposition}\label{bianchi}
For any $V, W, Z, Y \in L^{\infty}(TX)$, we have the following:
\begin{enumerate}
\item $R_X(V, W)Z=-R_X(W, V)Z$ on $X$  a.e. sense.
\item $R_X(V, W)Z+R_X(W, Z)V+R_X(Z, V)W=0$ on $X$  a.e. sense.
\item $\langle R_X(V, W)Z, Y\rangle+\langle R_X(V, W)Y, Z\rangle =0$ on $X$  a.e. sense.
\item $\langle R_X(V, W)Z, Y\rangle=\langle R_X(Z, Y)V, W\rangle$ on $X$  a.e. sense.
\end{enumerate}
\end{proposition}
\begin{proof}
We give a proof of (i) only because the proofs in other cases are similar.
By the Lebesgue differentiation theorem, it suffices to check that 
\begin{align}\label{hho}
\int_{B_r(x)}\langle R_X(V, W)Z, Y\rangle dH^n=-\int_{B_r(x)}\langle R_X(W, V)Z, Y\rangle dH^n
\end{align}
for any $x \in X$ and $r>0$.

Let $X_i$ be a sequence in $\mathcal{M}$ such that $X_i \stackrel{GH}{\to} X$.
From the existence of $L^p$-strong approximate sequence in \cite[Proposition $3.56$]{holp}, there exist sequences $V_i, W_i, Z_i, Y_i \in L^{\infty}(TX_i)$ with
\[\sup_i (||V_i||_{L^{\infty}}+||W_i||_{L^{\infty}}+||Z_i||_{L^{\infty}}+||Y_i||_{L^{\infty}})<\infty\]
such that $V_i, W_i, Z_i, Y_i$ $L^p$-converge strongly to $V, W, Z, Y$ on $X$ for every $p \in (1, \infty)$, respectively.

Let $r>0$ and let $x_i \stackrel{GH}{\to} x$.
Then since (i) holds on a smooth Riemannian manifold, we have
\[\int_{B_r(x_i)}\langle R_{X_i}(V_i, W_i)Z_i, Y_i\rangle dH^n=-\int_{B_r(x_i)}\langle R_{X_i}(W_i, V_i)Z_i, Y_i\rangle dH^n\]
for every $i$.
Thus
letting $i \to \infty$ with Theorem \ref{riem2} yields (\ref{hho}).
\end{proof}

\subsection{Ricci tensor}
Recall that our Ricci curvature $\mathrm{Ric}_X$ of $X \in \overline{\mathcal{M}}$ is defined by (\ref{definitionricci}).

\textit{Proofs of Theorem \ref{ellp} and of $L^{\infty}$-bound on Ricci curvature.}

Recall that every contraction of every $L^p$-weak convergent sequence of tensor fields is also an $L^p$-weak convergent sequence on a noncollapsed setting. See \cite[Proposition $3.72$]{holp}.
This result with Theorem \ref{riem2} and the lower semicontinuity of norms with respect to $L^p$-weak convergence (2.10) (c.f. \cite[Proposition 3.64]{holp}) gives Theorem \ref{ellp} and that $\mathrm{Ric}_X \in L^{\infty}(T^0_2X)$.
$\,\,\,\,\,\,\,\Box$

\subsubsection{Approximation theorems}
Main purpose of this subsection is to establish approximation theorems for Gigli's test differential forms:
\[\mathrm{TestForm}_k(X):=\left\{\sum_{i=1}^Nf_{0, i}df_{1, i} \wedge \cdots \wedge df_{k, i}; N \in \mathbf{N}, f_{j, i} \in \mathrm{Test}F(X)\right\}\]
and test tensor fields $\mathrm{Test}T^r_sX$ 
with respect to the Gromov-Hausdorff topology.
In order to give the precise statements, we give a quick introduction of Gigli's Sobolev spaces $H^{1, 2}_H(\bigwedge^kT^*X), H^{1, 2}_C(\bigwedge^kT^*X)$ for differential forms defined in \cite{gigli}, which will play fundamental roles in the study of the spectral convergence in Section 4.

For every $\omega \in \mathrm{TestForm}_k(X)$, we define $d \omega \in \mathrm{TestForm}_{k+1}(X), \delta \omega \in L^2(\bigwedge^{k-1}T^*X)$ and $\nabla \omega \in L^2(T^{0}_{k+1}X)$ as follows:
\begin{itemize}
\item If 
\[\omega=\sum_{i=1}^Nf_{0, i}df_{1, i} \wedge \cdots \wedge df_{k, i}\]
for some $f_{j, i} \in \mathrm{Test}F(X)$,
then
\[d\omega=\sum_{i=1}^Ndf_{0, i} \wedge df_{1, i} \wedge \cdots \wedge df_{k, i}.\]
\item $\delta \omega$ can be characterized by satisfying
\[\int_X\langle \delta \omega, \eta \rangle dH^n=\int_X\langle \omega, d\eta \rangle dH^n\]
for every $\eta \in \mathrm{TestForm}_{k-1}(X)$.
\item If $k=1$ and 
\[\omega=\sum_{i=1}^Nf_idg_i\]
for some $f_i, g_i \in \mathrm{Test}F(X)$,
then
\[\nabla \omega=\sum_{i=1}^N\left(dg_i\otimes df_i + f_i\mathrm{Hess}_{g_i}\right).\]
Then $\nabla$ is well-defined for test differential $k$-forms by satisfying Leibniz's rule.
\end{itemize}
Then the Sobolev spaces $H^{1, 2}_H(\bigwedge^kT^*X)$ and $H^{1, 2}_C(\bigwedge^kT^*X)$ are defined by the completions of $\mathrm{TestForm}_k(X)$ with respect to the norms
\[||\omega||_{H^{1, 2}_H}:=\left( ||\omega ||_{L^2}^2+||d\omega ||_{L^2}^2 + ||\delta \omega ||_{L^2}^2\right)^{1/2}\]
and
\[||\omega||_{H^{1, 2}_C}:=\left( ||\omega ||_{L^2}^2+||\nabla \omega ||_{L^2}^2\right)^{1/2},\]
respectively.

Similarly we can define the Sobolev space $H^{1, 2}_C(T^r_sX)$ for tensor fields of type $(r, s)$ by the completion of $\mathrm{Test}T^r_sX$ with respect to the norm
\[||T||_{H^{1, 2}_C}:=\left( ||T||_{L^2}^2+||\nabla T||_{L^2}^2\right)^{1/2}.\]
Note that these definitions are equivalent versions of the original them by Gigli.
See Section 3 in \cite{gigli} or subsection 2.5.4 in \cite{hoell} for the details.
\begin{remark}
As we mentioned in Section 1, we can find another approach in order to define the covariant derivative, the exterior derivative and so on via a  rectifiable coordinate system in \cite{ho}. 
A main result of \cite{hoell} states that these are compatible with Gigli's them as above.
\end{remark}

Next we recall the Hodge Laplacian, denoted by $\Delta_{H, k}$, acting on differential $k$-forms on $X$ defined in \cite{gigli}  by Gigli . 

Let $\mathcal{D}^2(\Delta_{H, k}, X)$ be the set of $\omega \in H^{1, 2}_H(\bigwedge^kT^*X)$ such that there exists 
$\eta \in L^2(\bigwedge^kT^*X)$ such that
\[\int_X\langle \eta, \xi \rangle dH^n=\int_X\left(\langle d\omega, d\eta \rangle +\langle \delta \omega, \delta \eta \rangle \right)dH^n\]
for every $\xi \in \mathrm{TestForm}_k(X)$.
Since $\eta$ is unique if it exists, we denote it by $\Delta_{H, k}\omega$.

Gigli proved in \cite{gigli} that
\[\mathrm{TestForm}_k(X) \subset \mathcal{D}^2(\Delta_{H, k}, X).\]
Moreover if $k=1$ and 
\[\omega=fdg \in \mathrm{TestForm}_k(X)\]
for some $f, g \in \mathrm{Test}F(X)$, then
\begin{align}\label{hodgeformula}
\Delta_{H, 1}\omega = fd\Delta g+\Delta fdg-2\mathrm{Hess}_g(\nabla f, \cdot )
\end{align}
and
\begin{align}\label{deltaformula1}
\delta \omega= -\langle df, dg \rangle +f\Delta g.
\end{align}
See \cite[Propositions $3.5.12$ and $3.6.1$]{gigli} or \cite[Remark $4.3$]{holp}.

Let us denote by $C^{\infty}(T^r_sM)$ and $C^{\infty}(\bigwedge^kT^*M)$ the spaces of smooth tensor fields of type $(r, s)$ and of smooth differential $k$-forms on a smooth closed Riemannian manifold $M$, respectively.
The following is the main result in this subsection.
\begin{theorem}[Smooth approximation]\label{14}
Let $X_i$ be a sequence in $\mathcal{M}$, let $X \in \overline{\mathcal{M}}$ be the Gromov-Hausdorff limit.
Then we have the following.
\begin{enumerate}
\item Let $\omega \in \mathrm{TestForm}_k(X)$.
Then there exists a sequence $\omega_{i} \in C^{\infty}(\bigwedge^kT^*X_{i})$ with
\[\sup_i(||\omega_{i}||_{L^{\infty}}+||d\omega_{i}||_{L^{\infty}})<\infty\]
such that
$\omega_{i}$, $d\omega_{i}$, $\nabla d\omega_{i}$, $\delta \omega_{i}$, $\nabla \omega_{i}$ $L^2$-converge strongly to $\omega$, $d\omega$, $\nabla d\omega$, $\delta \omega$, $\nabla \omega$ on $X$, respectively.
\item Let $\omega \in \mathrm{TestForm}_1(X)$.
Then
there exists a sequence $\omega_{i} \in C^{\infty}(T^*X_{i})$ with
\[\sup_i\left( ||\omega_{i}||_{L^{\infty}}+||d\omega_{i}||_{L^{\infty}} \right)<\infty\]
such that 
$\omega_{i}$, $d\omega_{i}$, $\nabla d\omega_{i}$, $\delta \omega_{i}$, $\nabla \delta \omega_{i}$, $\nabla \omega_{i}$, $\Delta_{H, 1} \omega_{i}$ $L^2$-converge strongly to $\omega$, $d\omega$, $\nabla d\omega$, $\delta \omega$, $\nabla \delta \omega$, $\nabla \omega$, $\Delta_{H, 1} \omega$ on $X$, respectively.
\item Let $T \in \mathrm{Test}T^r_sX$.
Then there exists a sequence $T_{i} \in C^{\infty}(T^r_sX_{i})$ with 
\[\sup_i| |T_{i}||_{L^{\infty}}<\infty\]
such that $T_{i}, \nabla T_{i}$ $L^2$-converge strongly to $T, \nabla T$ on $X$, respectively.
\end{enumerate}
\end{theorem}
\begin{proof}
We give a proof of (i) only because the proofs in other cases are similar.

Without loss of generality we can assume that
\[\omega=f_0 df_1 \wedge \cdots \wedge df_k\]
for some $f_i \in \mathrm{Test}F(X)$.

We first prove the following:
\begin{claim}\label{m88}
(i) holds if $f_i \in \widetilde{\mathrm{Test}}F(X)$ holds for every $i$ (recall (\ref{3wwc}) for the definition of $\widetilde{\mathrm{Test}}F(X)$).
\end{claim}
The proof is as follows.

By an argument similar to the proof of Corollary \ref{ppm}, there exist sequences $g_{i, j} \in C^{\infty}(X_j)$ with \[\mathrm{sup}_{i, j}(||\nabla g_{i, j}||_{L^{\infty}}+||g_{i, j}||_{L^{\infty}})<\infty\]
such that $g_{i, j}, \nabla g_{i, j}$ $L^2$-converge strongly to $\Delta f_i, \nabla \Delta f_i$ on $X$, respectively and that
\[\int_{X_j}g_{i, j}dH^n=0.\]
Let $f_{i, j}:=\Delta^{-1}g_{i, j} \in C^{\infty}(X_j)$.

By the Lipschitz regularity and the continuity of solutions of Poisson's equations (2.2) and (2.8), we see that 
\begin{align}\label{byt}
\sup_{i, j}(||f_{i, j}||_{L^{\infty}}+||\nabla f_{i, j}||_{L^{\infty}})<\infty
\end{align}
and
that $f_{i, j}, \nabla f_{i, j}$ $L^2$-converge strongly to $f_i, \nabla f_i$ on $X$, respectively.

Let 
\[\omega_j:=f_{0, j}df_{1, j} \wedge \cdots \wedge df_{k, j} \in C^{\infty}(\bigwedge^kT^*X_j).\]
Then (\ref{byt}), Theorem \ref{L2Hess} and Corollary \ref{nn} yield that
\[\sup_{j}(||\omega_j||_{L^{\infty}}+||d\omega_j||_{L^{\infty}})<\infty.\]
and that $\mathrm{Hess}_{f_{i, j}}, [\nabla f_{i, j}, \nabla f_{l, j}]$ $L^2$-converge strongly to $\mathrm{Hess}_{f_i},  [\nabla f_i, \nabla f_l]$ on $X$, respectively.
Thus by \cite[Proposition 3.5.12]{gigli} we see that $\omega_j, d\omega_j, \nabla d\omega_j, \delta \omega_j, \nabla \omega_j$ $L^2$-converge strongly to $\omega, d\omega, \nabla d\omega, \delta \omega, \nabla \omega$ on $X$, respectively.
This completes the proof of Claim \ref{m88}.

We are now in a position to finish the proof of (i).
From the existence of an $H^{2, 2}$-approximation (3.1), there exist sequences $f_i^j \in \widetilde{\mathrm{Test}}F(X)$ with 
\[\sup_{i, j}||\nabla f_i^j||_{L^{\infty}}<\infty\]
such that $f_i^j, \Delta f_i^j \to f_i, \Delta f_i$ in $H^{1, 2}(X)$, respectively.

Let
\[\omega^j:=f_0^j df_1^j \wedge \cdots \wedge df_k^j.\]
Then we see that 
\[\sup_j(||\omega^{j}||_{L^{\infty}}+||d\omega^{j}||_{L^{\infty}})<\infty\]
and that $\omega^j, d\omega^j, \nabla d\omega^j, \delta \omega^j, \nabla \omega^j$ $L^2$-converge strongly to $\omega, d\omega, \nabla d\omega, \delta \omega, \nabla \omega$ on $X$, respectively.
Thus applying Claim \ref{m88} for $\omega^j$ gives (i).
\end{proof}
\begin{remark}\label{tildeapp}
By the proof of Theorem \ref{14}, we also have the following approximation.
\begin{enumerate}
\item Let
\begin{align*}
\widetilde{\mathrm{TestForm}}_k(X)&:=\left\{\sum_{i=1}^Nf_{0, i}df_{1, i} \wedge \cdots \wedge df_{k, i}; N \in \mathbf{N}, f_{j, i} \in \widetilde{\mathrm{Test}}F(X)\right\} \\
&\subset \mathrm{TestForm}_k(X).
\end{align*}
In (i) of Theorem \ref{14}, if $\omega \in \widetilde{\mathrm{TestForm}}_k(X)$, then we can take an approximate sequence $\omega_{i}$ as in the conclusion with an additional property:
\[\sup_i\left( ||\nabla d\omega_{i}||_{L^{2q}}+||\nabla \omega_{i}||_{L^{2q}}\right)<\infty\]
for every $q \in (1, 2)$.
Moreover if $k=1$, then we can also assume
\[\sup_i\left(||\nabla \delta \omega_{i}||_{L^{\infty}}+||\Delta_{H, 1}\omega_{i}||_{L^{2q}}\right)<\infty.\]
\item Let
\begin{align*}
&\widetilde{\mathrm{Test}}T^r_sX\\
&:=\left\{\sum_{i=1}^Nf_{i, 0}\nabla f_{i, 1}\otimes \cdots \otimes \nabla f_{i, r}  \otimes df_{i, r+1} \otimes \cdots \otimes df_{i, r+s}; N \in \mathbf{N}, f_{i, j} \in \widetilde{\mathrm{Test}}F(X)\right\}\\
& \subset \mathrm{Test}T^r_sX
\end{align*}
(see (\ref{testtensor}) for the definition of $\mathrm{Test}T^r_sX$).
In (iii) of Theorem \ref{14}, if $T \in \widetilde{\mathrm{Test}}T^r_sX$, then we can take an approximate sequence $T_{i}$ as in the conclusion with an additional property:
\[\sup_i ||\nabla T_{i}||_{L^{2q}}<\infty \]
for every $q \in (1, 2)$.
\end{enumerate}
\end{remark}
\subsubsection{Bochner's formula for differential one-forms}
In this section we prove Bochner's formula for differential one-forms (Theorem \ref{pointwiseboch2}) by using Theorem \ref{14}.

For any $p \in (1, \infty)$ and $q \in [1, \infty]$, let $\mathcal{D}_{p, q}(\Delta, X)$ be the set of $f \in H^{1, p}(X)$ such that there exists $g \in L^q(X)$ satisfying
\begin{align*}
\int_X\langle \nabla f, \nabla h\rangle dH^n=\int_XghdH^n
\end{align*}
for every $h \in \mathrm{LIP}(X)$.
Since $g$ is unique if it exists, we denote it by $\Delta f$.
Note that by the definition we have $\mathcal{D}^2(\Delta, X)=\mathcal{D}_{2, 2}(\Delta, X)$.

\begin{proposition}\label{qplap}
Let $X_i \stackrel{GH}{\to} X$ in $\overline{\mathcal{M}}$, let $p, q \in (1, \infty)$, let $f_i$ be a sequence in $\mathcal{D}_{p, q}(\Delta, X_i)$ with
\[\sup_i \left( ||f_i||_{H^{1, p}}+||\Delta f_i||_{L^q}\right)<\infty,\]
and let $f$ be the $L^p$-weak limit on $X$.
Then $f \in \mathcal{D}_{p, q}(\Delta, X)$ and $\Delta f_i$ $L^q$-converges weakly to $\Delta f$ on $X$.
\end{proposition}
\begin{proof}
The proof is essentially same to that of \cite[Theorem $4.1$]{holp}.
For reader's convenience we give a proof.

The Rellich compactness theorem (2.3) yields that $f \in H^{1, p}(X)$ and that $\nabla f_i$ $L^p$-converges weakly to $\nabla f$ on $X$.

Let $h \in \mathrm{LIP}(X)$ and let $i(j)$ be a subsequence of $\mathbf{N}$.
By the compactness of $L^q$-weak convergence (2.9), there exist a subsequence $j(k)$ of $i(j)$ and  the $L^q$-weak limit $F$ of $\Delta f_{j(k)}$ on $X$.
By the existence of an approximate sequence of Lipschitz functions (2.6), there exists a sequence $h_{j(k)} \in \mathrm{LIP}(X_{j(k)})$ with
\[\sup_k||\nabla h_{j(k)}||_{L^{\infty}}<\infty \]
such that $h_{j(k)}, \nabla h_{j(k)}$ $L^r$-converge strongly to $h, \nabla h$ on $X$ for every $r \in (1, \infty)$, respectively.
Since
\[\int_{X_{j(k)}}\langle \nabla h_{j(k)}, \nabla f_{j(k)}\rangle dH^n=\int_{X_{j(k)}}h_{j(k)}\Delta f_{j(k)}dH^n,\]
letting $k \to \infty$ yields
\[\int_X\langle \nabla h, \nabla f\rangle dH^n=\int_XhFdH^n.\]
Since $h$ and $i(j)$ are arbitrary, this completes the proof.
\end{proof}

The following is Bochner's formula for test differential one-forms on $X$.
\begin{theorem}[Bochner's formula for test differential one-forms]\label{12}
Let $\omega \in \mathrm{TestForm}_1(X)$.
Then we have the following:
\begin{enumerate}
\item $|\omega|^2 \in \mathcal{D}_{2, 1}(\Delta, X)$.
\item $|\omega|^2 \in \mathcal{D}_{2, q}(\Delta, X)$ for every $q<2$ if $\omega \in \widetilde{\mathrm{TestForm}}_1(X)$.
\item We have
\begin{align}\label{100}
-\frac{1}{2}\Delta |\omega|^2=|\nabla \omega|^2-\langle \Delta_{H, 1}\omega, \omega \rangle + \langle \mathrm{Ric}_X, \omega \otimes \omega \rangle.
\end{align}
\end{enumerate}
\end{theorem} 
\begin{proof}
Let $X_i$ be a seuqnece in $\mathcal{M}$ such that $X_i \stackrel{GH}{\to} X$ in $\overline{\mathcal{M}}$.

We first check (ii).
Assume that $\omega \in \widetilde{\mathrm{TestForm}}_1(X)$.
Then (ii) of Theorem \ref{14} with (i) of Remark \ref{tildeapp} yields that there exists a sequence  $\omega_{i} \in C^{\infty}(T^*X_{i})$ with 
\[\sup_i\left( ||\omega_{i}||_{L^{\infty}}+||\nabla \omega_{i}||_{L^{2q}}+||\Delta_{H, 1}\omega_{i}||_{L^{2q}}\right)<\infty\]
for every $q \in (1, 2)$ such that $\omega_{i}, \nabla \omega_{i}, \Delta_{H, 1}\omega_{i}$ $L^2$-converge strongly to $\omega, \nabla \omega, \Delta_{H, 1}\omega$ on $X$, respectively.
In particular since 
\[|\nabla |\omega_{i}|^2|\le 2|\omega_{i}||\nabla \omega_{i}|,\]
the Rellich compactness theorem (2.3) yields $|\omega|^2 \in H^{1, 2}(X)$.

On the other hand, Bochner's formula on a smooth Riemannian manifold yields
\begin{align}\label{bhhg}
-\frac{1}{2}\Delta |\omega_{i}|^2=|\nabla \omega_{i}|^2-\langle \Delta_{H, 1}\omega_{i}, \omega_{i} \rangle + \langle \mathrm{Ric}_{X_{i}}, \omega_{i} \otimes \omega_{i} \rangle.
\end{align}
In particular
\[\sup_i||\Delta |\omega_{i}|^2||_{L^{q}}<\infty\]
for every $q \in (1, 2)$.
Thus letting $i \to \infty$ in (\ref{bhhg}) with Proposition \ref{qplap} gives (ii) and (iii) in the case when $\omega \in \widetilde{\mathrm{TestForm}}_1(X)$.

Next we prove (i) and (iii) in general case.

By an argument similar to that above with the existence of an $H^{2, 2}$-approximate sequence (3.1), we see that $|\omega|^2 \in H^{1, 2}(X)$ and that there exists a sequence $\omega^i \in \widetilde{\mathrm{TestForm}}_1(X)$ with
\[\sup_j||\omega^j||_{L^{\infty}}<\infty \]
such that $\omega^j, \nabla \omega^j, \Delta_{H, 1}\omega^j$ $L^2$-converge strongly to $\omega,  \nabla \omega, \Delta_{H, 1}\omega$ on $X$, respectively.

Let $f \in \mathrm{LIP}(X)$.
Then (iii) of Theorem \ref{12} in the case when $\omega \in \widetilde{\mathrm{TestForm}}_1(X)$ yields
\begin{align}\label{ddc}
&-\frac{1}{2}\int_X\langle df, d|\omega^j|^2\rangle dH^n \nonumber \\
&=\int_X\left( f|\nabla \omega^j|^2 -\langle \Delta_{H, 1}\omega^j, f\omega^j\rangle + f\langle \mathrm{Ric}_X, \omega^j \otimes \omega^j \rangle \right)dH^n.
\end{align}
Note that by the Rellich compactness theorem (2.3), $d|\omega^j|^2$ $L^2$-converges weakly to $d |\omega|^2$ on $X$.
Thus letting $j \to \infty$ in (\ref{ddc}) gives (i) and (iii) in general case.
\end{proof}
Note that on a noncollapsed Ricci limit space $(Y, H^n)$ (see Remark \ref{boundei} for the definition of Ricci limit spaces), it was proven in \cite[Theorem 7.12]{hoell} that $H^{1, 2}_H(T^*Y)$ coincides with $H^{1, 2}_C(T^*Y)$ as sets and that  the identity map
\[\mathrm{id}:H^{1, 2}_H(T^*Y) \to H^{1, 2}_C(T^*Y)\]
 gives a homeomorphism between them.
The following gives a refinement of it on our setting.
\begin{corollary}\label{aavb}
We see that $H^{1, 2}_H(T^*X)=H^{1, 2}_C(T^*X)$ and that for every $\omega \in H^{1, 2}_H(T^*X)$, 
\begin{align}\label{nnghh}
||\omega||_{H^{1, 2}_H}^2=||\omega||_{H^{1, 2}_C}^2 + \int_X\langle \mathrm{Ric}_X, \omega \otimes \omega \rangle dH^n.
\end{align}
\end{corollary}
\begin{proof}
Theorem \ref{12} yields that (\ref{nnghh}) holds if $\omega \in \mathrm{TestForm}_1(X)$, which implies Corollary \ref{aavb}.
\end{proof}
We are now in a position to give Bochner's formula for differential one-forms.
\begin{theorem}[Bochner's formula for differential one-forms]\label{pointwiseboch}
Let $\omega \in H^{1, 2}_H(T^*X)$. Then we have the following:
\begin{enumerate}
\item $|\omega|^2 \in H^{1, 2n/(2n-1)}(X)$.
\item We have
\begin{align}\label{impor}
&-\frac{1}{2}\int_X\langle df, d|\omega|^2\rangle dH^n \nonumber \\
&= \int_X\left( f|\nabla \omega|^2- \langle d\omega, d(f\omega )\rangle - (\delta \omega) (\delta (f\omega))+f\langle \mathrm{Ric}_X, \omega \otimes \omega \rangle \right)dH^n
\end{align}
for every $f \in \mathrm{LIP}(X)$.
\end{enumerate}
In particular if $\omega \in \mathcal{D}^2(\Delta_{H, 1}, X)$,
then $|\omega|^2 \in \mathcal{D}_{2n/(2n-1), 1}(\Delta, X)$ and (\ref{100}) hold (i.e. we have Theorem \ref{pointwiseboch2}).
\end{theorem}
\begin{proof}
We first prove the following:
\begin{claim}\label{101}
For every $\eta \in \mathrm{TestForm}_1(X)$,
\begin{align}\label{poincare1form}
&\left(\frac{1}{H^n(X)}\int_X\left||\eta|^2-\frac{1}{H^n(X)}\int_X|\eta |^2dH^n\right|^{n/(n-1)} dH^n\right)^{(n-1)/n} \nonumber \\
&\le \frac{C(n, K_1, d)}{H^n(X)}\int_X\left(|\eta|^2+ |\nabla \eta|^2\right)dH^n.
\end{align}
\end{claim}
The proof is as follows.
Recall that the $(n/(n-1), 1)$-Poincar\'e inequality:
\begin{align}\label{poinc}
\left(\frac{1}{H^n(M)}\int_M\left|f-\frac{1}{H^n(M)}\int_MfdH^n\right|^{n/(n-1)} dH^n\right)^{(n-1)/n} \le \frac{C(n, K_1, d)}{H^n(M)}\int_M |\nabla f|dH^n.
\end{align}
holds for any $n$-dimensional closed Riemannian manifold $M$ with $\mathrm{Ric}_M \ge K_1$, and $f \in \mathrm{LIP}(M)$ (see \cite[Th\'eor\`eme 1.1]{mas}). 
By the stability of Poincar\'e inequalities with respect to the Gromov-Hausdorff topology in \cite[Theorem $3.2$]{hoell}, (\ref{poinc}) holds for any $(n, K_1)$-Ricci limit spaces. In particular this also holds for $X$.

Then since 
\[|\nabla |\eta|^2| \le 2|\eta| |\nabla \eta| \le |\eta|^2+|\nabla \eta|^2,\]
applying (\ref{poinc}) for $f=|\eta|^2 \in H^{1, 2}(X)$ gives Claim \ref{101}.

Let $f \in \mathrm{LIP}(X)$ and let $\omega^j$ be a sequence in $\mathrm{TestForm}_1(X)$ such that $\omega^j \to \omega$ in $H^{1, 2}_H(T^*X)$.
Note that Corollary \ref{aavb} yields that $\nabla \omega^j \to \nabla \omega$ in $L^2(T^0_2X)$.

On the other hand, since Young's inequality yields
\[|d|\omega^j|^2|^{p_n} \le 2^{p_n}|\omega^j|^{p_n}|\nabla \omega^j|^{p_n} \le (2-p_n)2^{p_n-1}|\omega^j|^{2n/(n-1)} + p_n2^{p_n-1}|\nabla \omega^j|^2,\]
the Rellich compactness (2.3) with Claim \ref{101} gives that (i) holds  and that $d|\omega^j|^2$ $L^{p_n}$-converges weakly to $d|\omega|^2$ on $X$,
where 
\[\frac{2n}{n-1}=\frac{2p_n}{2-p_n},\]
i.e. 
\[p_n:=\frac{2n}{2n-1}.\]
Since Theorem \ref{12} gives
\begin{align*}
&-\frac{1}{2}\int_X\langle df, d|\omega^j|^2\rangle dH^n \nonumber \\
&= \int_X\left( f|\nabla \omega^j|^2- \langle d\omega^j, d(f\omega^j )\rangle - (\delta \omega^j) (\delta (f\omega^j))+f\langle \mathrm{Ric}_X, \omega^j \otimes \omega^j \rangle \right)dH^n,
\end{align*}
letting $j \to \infty$ yields (\ref{impor}).
This completes the proof.
\end{proof}
We give a generalization of (\ref{hodgeformula}) in order to prove Theorem \ref{boch}.
\begin{proposition}\label{bbghnv}
Let $f \in \mathcal{D}^2(\Delta, X)$ and let $g \in \mathrm{LIP}(X)$.
Then $gdf \in H^{1, 2}_H(T^*X)$ with $d(gdf)=dg \wedge df$ and $\delta (gdf)=-\langle df, dg\rangle +g\Delta f$.
Moreover if $\Delta f \in H^{1, 2}(X)$ and $g \in \mathcal{D}^2(\Delta, X)$ with $\Delta g df \in L^{2}(T^*X)$, then $gdf \in \mathcal{D}^2(\Delta_{H, 1}, X)$ with
\begin{align}\label{hhu}
\Delta_{H, 1}(gdf)=gd\Delta f + \Delta gdf -2\mathrm{Hess}_f(\nabla g, \cdot ).
\end{align}
\end{proposition}
\begin{proof}
Note that the equality $\delta (gdf)=-\langle df, dg\rangle +g\Delta f$ was known in this case (see for instance \cite[Remark 4.3]{holp}).

Let $h_i$ be a sequence in $\mathrm{LIP}(X)$ with 
\[\int_Xh_idH^n=0\]
and $h_i \to \Delta f$ in $L^2(X)$.
Let $f_i:=\Delta^{-1}h_i$.
By the Lipschitz regurality and the continuity of solutions of Poisson's equations (2.2) and (2.8), we see that $f_i \in \widetilde{\mathrm{Test}}F(X)$ and that $f_i, df_i$ $L^2$-converge strongly to $f, df$ on $X$, respectively.
In particular $df_i$ is a convergent sequence in $H^{1, 2}_H(T^*X)$.

Recall that $\widetilde{H}_{\epsilon}g \in \mathrm{Test}F(X)$ for every $\epsilon >0$ (c.f. \cite[(3.2.3)]{gigli}), where $\widetilde{H}_{\epsilon}$ is a mollified heat flow defined by (\ref{mollified}).
Let $g_i:=\widetilde{H}_{i^{-1}}g$.
Note that 
\[\sup_i||\nabla g_i||_{L^{\infty}}<\infty\]
and that $g_i, dg_i$ $L^2$-converge strongly to $g, dg$ on $X$, respectively.
In particular since $g_idf_i$ is a convergent sequence in $H^{1, 2}_H(T^*X)$ with $d(g_idf_i)=dg_i \wedge df_i$ (c.f. \cite[Theorem $3.5.2$]{gigli}), 
we have $gdf \in H^{1, 2}_H(T^*X)$ with $d(gdf)=dg \wedge df$.

Next we assume that $\Delta f \in H^{1, 2}(X)$ and that $g \in \mathcal{D}^2(\Delta, X)$ with $\Delta g df\in L^{2}(T^*X)$.
We also use the same notation as above.

Then we can take $f_i$ as above with an additional condition that $\Delta f_i \to \Delta f$ in $H^{1, 2}(X)$.
Note that $\Delta g_i \to \Delta g$ in $L^2(X)$
and that (\ref{hodgeformula}) yields 
\[\Delta_{H, 1}(g_idf_i)=g_id\Delta f_i+\Delta g_idf_i -2\mathrm{Hess}_{f_i}(\nabla g_i, \cdot).\]
Thus for every $\omega \in \mathrm{TestForm}_1(X)$ since
\begin{align*}
&\int_X\left( \langle d(g_idf_i), d\omega \rangle + (\delta (g_idf_i))(\delta \omega) \right)dH^n\\
&=\int_X \left( \langle g_id\Delta f_i +\Delta g_idf_i, \omega\rangle -2\langle \mathrm{Hess}_{f_i}, dg_i \otimes \omega \rangle \right)dH^n,
\end{align*}
letting $i \to \infty$ with Theorem \ref{L2Hess} yields 
\begin{align*}
&\int_X\left( \langle d(gdf), d\omega \rangle + (\delta (gdf))(\delta \omega) \right)dH^n\\
&=\int_X \left( \langle gd\Delta f +\Delta gdf, \omega\rangle -2\langle \mathrm{Hess}_{f}, dg \otimes \omega \rangle \right)dH^n.
\end{align*}
This gives $gdf \in \mathcal{D}^2(\Delta_{H, 1}, X)$ with (\ref{hhu}).
\end{proof}

We end this subsection by giving a proof of Theorem \ref{boch}.

\textit{Proof of Theorem \ref{boch}.} 

By Proposition \ref{bbghnv}, applying Theorem \ref{pointwiseboch} for $\omega =dh$ yields Theorem \ref{boch}.
$\,\,\,\,\,\,\,\,\,\Box$
\subsection{Scalar curvature and $L^p$-strong convergence of curvature}
We first prove Theorem \ref{contsca}.

\textit{Proof of Theorem \ref{contsca}.}

This follows directly from an argument similar to the proof of Theorem \ref{ellp}, i.e. we recall that under a noncollapsed setting, for every $L^p$-weak convergent sequence of tensor fields of type $(0, 2)$, the traces is also an $L^p$-weak convergent sequence (\cite[Proposition 3.72]{holp}).
By applying this result for our Ricci curvature with Theorem \ref{ellp} we have Theorem \ref{contsca}.
$\,\,\,\,\,\,\,\,\,\Box$

The following Gauss-Bonnet and Gauss-Bonnet-Chern formulae on a singular setting are a direct consequence of Theorem \ref{contsca}.
Recall that the scalar curvature is twice the Gauss curvature on a smooth surface.
\begin{proposition}\label{gauss}
Let $X \in \overline{\mathcal{M}}$.
\begin{enumerate}
\item If $\mathrm{dim}\,X=2$ (i.e. $n=2$),
then we have
\[4\pi \chi(X)=\int_Xs_XdH^2,\]
where $\chi(X)$ is the Euler characteristic of $X$.
\item If $\mathrm{dim}\,X=4$ and there exists a sequence of $4$-dimensional closed manifolds $M_i$ with
\[\sup_i||\mathrm{Sec}_{M_i}||_{L^{\infty}}<\infty\]
such that $M_i \stackrel{GH}{\to} X$ and that $R_{M_i}$ $L^2$-converges strongly to $R_X$ on $X$, then we have
\[32\pi \chi(X)= \int_{X}\left( |R_{X}|^2 -4|\mathrm{Ric}_{X}|^2 +|s_{X}|^2\right)dH^n.\]
\end{enumerate}
\end{proposition}
\begin{proof}
We give a proof of (i) only because the proof of (ii) is similar.

Let $X_i$ be a sequence in $\mathcal{M}$ with $X_i \stackrel{GH}{\to} X$ in $\overline{\mathcal{M}}$.
It is well-known that $X_i$ is homeomorphic to $X$ for every sufficiently large $i$.
In particular 
\[\chi (X_i)=\chi (X)\]
for every sufficiently large $i$.

On the other hand, since
\[4\pi \chi (X_i)=\int_{X_i}s_{X_i}dH^2,\]
letting $i \to \infty$ with Theorem \ref{contsca} gives (i).
\end{proof}

We now give the following questions:
\begin{enumerate}
\item[\textbf{(Q1)}] When does the $L^2$-strong convergence of Ricci curvature imply that of Riemannian curvature?
\item[\textbf{(Q2)}] When does the $L^2$-strong convergence of scalar curvature imply that of Ricci curvature?
\end{enumerate}
Note that the converse are always satisfied (\cite[Proposition 3.74]{holp}).

Since we will give a positive answer to the question \textbf{(Q2)}  on a K\"ahler setting in Section 4 (Theorem \ref{equivalence}),
we here give a positive answer to the question \textbf{(Q1)} in the case when $n=4$ and the limit is smooth:
\begin{proposition}\label{22wn}
Let $X_i \stackrel{GH}{\to} X$ in $\mathcal{M}$ with $\mathrm{dim}\,X=4$.
Then if $\mathrm{Ric}_{X_i}$ $L^2$-converges strongly to $\mathrm{Ric}_X$ on $X$, then $R_{X_i}$ $L^2$-converges strongly to $R_X$ on $X$.
\end{proposition}
\begin{proof}
By Cheeger-Colding's stability theorem in \cite[Theorem A.1.12]{ch-co1}, $X_i$ is diffeomorphic to $X$ for every sufficiently large $i$.
In particular $\chi (X_i)=\chi (X)$ for every sufficiently large $i$.
Thus by the Chern-Gauss-Bonnet formula and Theorem \ref{riem2}, we see that $R_{X_i}$ $L^2$-coverges weakly to $R_X$ on $X$
(c.f. \cite[Theorem 1.13]{chna}).

Assume that  $\mathrm{Ric}_{X_i}$ $L^2$-converges strongly to $\mathrm{Ric}_X$ on $X$.
By the Chern-Gauss-Bonnet formula and the lower semicontinuity of norms with respect to the $L^p$-weak convergence (2.10), we have
\begin{align*}
32\pi \chi (X)&= \lim_{i \to \infty}32\pi \chi(X_i)\\
&=\lim_{i \to \infty}\int_{X_i}\left( |R_{X_i}|^2 -4|\mathrm{Ric}_{X_i}|^2 +|s_{X_i}|^2\right)dH^n \\
&\ge \int_{X}\left( |R_{X}|^2 -4|\mathrm{Ric}_{X}|^2 +|s_{X}|^2\right)dH^n=32\pi \chi (X).
\end{align*}
Thus  
\[\lim_{i \to \infty}\int_{X_i}|R_{X_i}|^2dH^n=\int_X|R_X|^2dH^n.\]
This completes the proof.
\end{proof}
The assumption of the smoothness of the limit as in Proposition \ref{22wn} is essential.
We give an example.
\begin{example}[Page \cite{pa}, Kobayashi-Todorov \cite{kt}]\label{pakt}
Let $\{-1, 1\}$ act on a complex $2$-torus $T:=\mathbf{C}^2/\mathbf{Z}^2$ by
\[(-1)\cdot (z_1, z_2):=(-z_1, -z_2).\]
Define an orbifold $X$ by 
\[X:=T/\{\pm 1\}\]
with the standard flat orbifold metric $g_{X}$.
It is not difficult to check that the second Betti number $b_2(X)$ is equal to $6$.

Let $\pi: Y \to X$ be the minimal resolusion and let $S$ be the singular set of $X$ (which consists of $16$-points).
It is well-known that $Y$ is a Kummer surface.  
In particular $\chi (Y)$ is equal to $24$ and the second Betti number $b_2(Y)$ is equal to $22$.

By the Calabi-Yau theorem, we see that there exists a sequence of Ricci flat K\"ahler-Einstein metrics $g_i$ on $Y$ such that the following hold.
\begin{enumerate}
\item The volume of $Y$ with respect to $g_i$ is equal to $1$.
\item For every $x \in S$, the volume of $\pi^{-1}(x)$ with respect to $g_i$ is equal to $i^{-1}$.
\item $(Y, g_i) \stackrel{GH}{\to} (X, g_{X})$.
\end{enumerate}
In particular we see that $\mathrm{Ric}_{(Y, g_i)}$ $L^2$-converges strongly to $\mathrm{Ric}_{(X, g_X)}$ on $X$ and that
the lower semicontinuity of second Betti numbers
\[\liminf_{i \to \infty}b_2(Y, g_i)>b_2(X, g_X)\]
holds.

On the other hand, the Chern-Gauss-Bonnet formula implies
\[\int_X|R_{(Y, g_i)}|^2dH^4=32\pi^2 \chi(Y)=768\pi^2.\]
Thus by Theorem \ref{riem2}, $R_{(Y, g_i)}$ $L^2$-converges weakly to $R_{(X, g_X)}$ on $X$.

However since
\[\int_{X}|R_{(X, g_{X})}|^2dH^4=0,\]
$R_{(Y, g_i)}$ does not $L^2$-converge strongly to $R_{(X, g_X)}$ on $X$.
\end{example}
We also give an example of a Gromov-Hausdorff convergent sequence with bounded curvature such that the scalar curvature does not $L^2$-strong converge.
\begin{example}\label{notsca}
Let $(M, g_M)$ be an $n$-dimensional closed Riemannian manifold.
Then it is not difficult to check that there exists a sequence $f_i \in C^{\infty}(M)$ with 
\[\sup_i ||\mathrm{Hess}_{f_i}||_{L^{\infty}}<\infty\]
such that $f_i$ $L^2$-converge strongly to a smooth function $f \in C^{\infty}(M)$ on $M$
and that $\Delta f_{i}$ does not $L^2$-converge strongly to $\Delta f$ on $M$. 
Note that by the Lipschitz regurality of solutions of Poisson's equations (2.2), 
\[\sup_i||\nabla f_i||_{L^{\infty}}<\infty.\]
In particular $f_i \to f$ in $C^0(M)$.

Then from \cite{sakai}, the following hold:
\begin{enumerate}
\item We have 
\[\sup_i||R_{(M, e^{f_i}g_M)}||_{L^{\infty}}<\infty.\]
\item $(M, e^{f_i}g_M) \stackrel{GH}{\to} (M, e^fg_M)$. 
\item $s_{(M, e^{f_i}g_M)}$ does not $L^2$-converge strongly to $s_{(M, e^fg_M)}$ on $M$.
\end{enumerate}
\end{example}

\subsection{Proof of compatibilities}
In this section we prove compatibilities as  we stated in Section 1.
\subsubsection{Smooth setting}
We give a proof of Proposition \ref{curvaturecompatibility}.

\textit{Proof of Proposition \ref{curvaturecompatibility}.}

Let $\hat{R}_U$ be the Riemannian curvature tensor of $(U, g_X|U)$ defined by the ordinary way of Riemannian geometry.
Then by (\ref{riem}) we have
\[\int_Uf_0\langle R_X, df_1 \otimes df_2, \otimes df_3 \otimes df_4\rangle dH^n=\int_Uf_0\langle \hat{R}_U, df_1 \otimes df_2, \otimes df_3 \otimes df_4\rangle dH^n\]
for any $f_i \in C^{\infty}_c(U)$, where $C^{\infty}_c(U)$ denotes the set of smooth functions on $U$ with compact supports.
Since the space
\[\left\{ \sum_{i=1}^Nf_0df_1 \otimes df_2, \otimes df_3 \otimes df_4; N \in \mathbf{N}, f_{i} \in C^{\infty}_c(U)\right\}\]
is dense in $L^q(T^4_0U)$ for any $q \in (1, \infty)$, we have $R_X=\hat{R}_U$ on $U$ a.e. sense.
$\,\,\,\,\,\,\Box$

\subsubsection{Lott's Ricci mesure}
In this section we give a proof of Theorem \ref{comlott}.
See page 7 in \cite{lott2} for the precise definition of Lott's Ricci measure on a $C^{1, 1}$-manifold.
Note that 
by Theorem \ref{boch} and the $L^{\infty}$-bound on our Ricci curvature, an essential assumption in order to define the Ricci measure, \cite[Assumption 3.19]{lott2}, holds for the regular set $\mathcal{R}$ of every $X \in \overline{\mathcal{M}}$.
In particular Lott's Ricci measure is well-defined on $\mathcal{R}$ (note that the $C^{1, \alpha}$-regularity of the Riemannian metric yields that the Levi-Civita connection is tame in the sense of \cite{lott2}).

\textit{Proof of Theorem \ref{comlott}.}

Let $f_{i}$ be locally Lipschitz functions on $\mathcal{R}$, let $g_i$ be $C^2$-functions on $\mathcal{R}$ with compact supports, and let $V_i:=f_i \nabla g_i$, where $i=1, 2$.
It suffices to check that 
\begin{align}\label{swg}
\int_{X}\mathrm{Ric}_X(V_1, V_2)dH^n=\int_X\left( \sum_{i, j} \left( (\nabla_iV^i_1)(\nabla_jV_2^i)-(\nabla_i V^j_1)(\nabla_jV_2^i)\right) \right)dH^n.
\end{align}
See \cite[(1.3)]{lott}.

Note that $g_i \in \mathcal{D}^2(\Delta, X)$ and that without loss of generality we can assume that $f_i \in \mathrm{LIP}(X)$.
Then  Corollary \ref{aavb} and  Proposition \ref{bbghnv} yield $V_i \in H^{1, 2}_C(TX)$ and
\begin{align}\label{lottproof}
\int_{X}\mathrm{Ric}_X(V_1, V_2)dH^n=\int_X\left(\langle dV_1^*, dV_2^*\rangle + (\delta V_1^*)(\delta V_2^*) -\langle \nabla V_1, \nabla V_2\rangle \right)dH^n.
\end{align}
On the other hand, since $V_i^* \in H^{1, 2}_H(T^*X)=H^{1, 2}_C(T^*X)$, $dV_i^*, \delta V_i^*, \nabla V_i$ coincide with that defined by the ordinary way of Riemannian geometry via the $C^1$-Riemannian metric $g_X|_{\mathcal{R}}$, respectively (c.f. \cite[Theorems 1.10 and 7.3]{hoell}).
Therefore (\ref{lottproof}) yields (\ref{swg}).
$\,\,\,\,\,\,\,\,\Box$

\subsubsection{Gigli's Ricci curvature} 
We first recall the definition of Gigli's Ricci curvature on our setting.

Let $X \in \overline{\mathcal{M}}$.
Then Gigli's Ricci curvature $\mathbf{Ric}_X$ is defined by the bilinear continuous map from  $H^{1, 2}_H(TX) \times H^{1, 2}_H(TX)$ to $\mathrm{Meas}(X)$ such that 
\begin{align}\label{bbvbb}
\mathbf{Ric}_X(V, W):=-\frac{\mathbf{\Delta} \langle V, W\rangle }{2} + \left(\frac{\langle \Delta_{H, 1}V^*, W^*\rangle}{2} + \frac{\langle \Delta_{H, 1}W^*, V^*\rangle}{2}-\langle \nabla V, \nabla W\rangle\right)dH^n \in \mathrm{Meas}(X)
\end{align}
holds for any $V, W \in \mathrm{Test}TX$,
where $H^{1, 2}_H(TX)$ is the set of $V \in L^2(TX)$ such that $V^* \in H^{1, 2}_H(T^*X)$ and $\mathbf{\Delta} \langle V, W\rangle  \in \mathrm{Meas}(X)$ is the measure valued Laplacian of $\langle V, W\rangle$ defined by satisfying 
\[\int_Xgd(\mathbf{\Delta} \langle V, W\rangle )=\int_X\langle df, dg\rangle dH^n\]
for every $g \in \mathrm{LIP}(X)$ (recall that $\mathrm{Meas}(X)$ is the Banach space of finite signed Radon measures on $X$ equipped with the total variation norm).
See 98 page of \cite{gigli} and \cite[Proposition 3.1.3 and Theorem 3.6.7]{gigli} for the details (note that the sign of our Laplacian is different from Gigli's one).

We now prove Theorem \ref{gigliricci}.

\textit{Proof of Theorem \ref{gigliricci}.}

For any $V, W \in \mathrm{Test}T^1_0X$, (iii) of Theorem \ref{12} yields 
\[\mathrm{Ric}_X(V, W)=-\frac{\Delta \langle V, W\rangle}{2} + \frac{\langle \Delta_{H, 1}V^*, W^*\rangle}{2} + \frac{\langle \Delta_{H, 1}W^*, V^*\rangle}{2}+\langle \nabla V, \nabla W\rangle \in L^1(X).\]

On the other hand, by the definition of $\mathbf{\Delta}$ with (i) of Theorem \ref{12}, we have
\[\Delta \langle V, W\rangle =\mathbf{\Delta}\langle V, W\rangle \]
in $\mathrm{Meas}(X)$ via the canonical inclusion $L^1(X) \hookrightarrow \mathrm{Meas}(X)$.
Therefore this completes the proof of Theorem \ref{gigliricci}.
$\,\,\,\,\,\,\,\,\Box$

\begin{remark}\label{gigliquestion}
Let $X \in \overline{\mathcal{M}}$.
We give positive answers to three questions on Gigli's Ricci curvature $\mathbf{Ric}_X$ given in the final section of \cite{gigli} on our setting.
His questions are closely related to the realization of $\mathbf{Ric}_X$ as a tensor field.
On our setting, Theorem \ref{gigliricci} gives the realization. 
In particular we can give positive answers to his questions, automatically.
 
In order to introduce the precise statement, we first discuss the following two questions of them.
\begin{enumerate}
\item[\textbf{(Q3)}] Can we extend $\mathbf{Ric}_X$ to a bilinear continuous map from $[L^2(TX)]^2$ to $\mathrm{Meas}(X)$?
\item[\textbf{(Q4)}] Let
\[\mathcal{E}_{\mathrm{diff}}(V):=\frac{1}{2}\inf_{\{V_i\}_i} \left(\liminf_{i \to \infty}\left( \int_X \left(|dV_i^*|^2+|\delta V_i^*|^2 - |\nabla V_i|^2\right)dH^n\right) \right)\]
for every $V \in L^2(TX)$,
where the infimum runs over all sequence $V_n \in H^{1, 2}_H(TX)$ converging to $V$ in $L^2(TX)$.
Then is it true that 
\[\mathcal{E}_{\mathrm{diff}}(V)=\frac{1}{2}\int_X \left(|dV^*|^2+|\delta V^*|^2 - |\nabla V|^2\right)dH^n\]
for every $V \in H^{1, 2}_H(TX)$?
\end{enumerate}
See 156 page in \cite{gigli}.

Theorem \ref{pointwiseboch} with the $L^{\infty}$-bound on our Ricci curvature yields that $\mathbf{Ric}_X$ can be extended to the bilinear continuous map from $[L^2(TX)]^2$ to $L^1(X)$ defined by
\[(V, W) \mapsto \int_X\mathrm{Ric}_X(V, W)dH^n.\]
This gives the positive answer to the question \textbf{(Q3)}.

On the other hand, for the question \textbf{(Q4)}, Theorem \ref{pointwiseboch} yields that
\[\mathcal{E}_{\mathrm{diff}}(V)=\frac{1}{2}\int_X \mathrm{Ric}_X(V, V)dH^n=\frac{1}{2}\int_X \left(|dV^*|^2+|\delta V^*|^2 - |\nabla V|^2\right)dH^n \]
for every $V \in H^{1, 2}_H(TX)$.
This gives the positive answer to the question \textbf{(Q4)}.

Finally, Gigli gave the following question:
\begin{enumerate}
\item[\textbf{(Q5)}] Does the following linearlity
\begin{align}\label{hhgj}
\sum_if_i\mathbf{Ric}_X(V_i, V)=\mathbf{Ric}_X\left(\sum_if_iV_i, V\right)
\end{align}
holds for any $f_i \in C^0(X)$ and $V_i, V \in H^{1, 2}_H(TX)$?
\end{enumerate}
See page 157 in \cite{gigli}.

It is also trivial from Theorem \ref{gigliricci} that the question \textbf{(Q5)} has the positive answer on our setting. 
Note that by the answer to the question \textbf{(Q3)}, (\ref{hhgj}) holds even if $V_i,  V \in L^2(TX)$.
\end{remark}
\subsubsection{$RCD$-setting}
We give a proof of Theorem \ref{equ}.

\textit{Proof of Theorem \ref{equ}.}

First we assume that (i) holds.
Let $h \in \mathcal{D}^2(\Delta, X)$.
Then by (\ref{pqty}) we have
\begin{align*}
-\frac{1}{2}\int_{X}\langle \nabla f, \nabla|\nabla h|^2 \rangle dH^n\ge\int_{X}\left( f \frac{(\Delta h)^2}{n} -f (\Delta h)^2 + \Delta h \langle \nabla f, \nabla h\rangle +f\mathrm{Ric}_{X}(\nabla h, \nabla h)\right)dH^n
\end{align*} 
for every $f \in \mathrm{LIP}(X)$ with $f \ge 0$.
In particular if $\Delta h \in H^{1, 2}(X)$, then
\begin{align*}
-\frac{1}{2}\int_{X} \Delta f |\nabla h|^2 dH^n \ge\int_{X}\left( f \frac{(\Delta h)^2}{n} -f\langle \nabla \Delta h, \nabla h\rangle +fK|\nabla h|^2\right)dH^n
\end{align*} 
for every $f \in \mathcal{D}^2(\Delta, X)$ with $f \ge 0$ and $\Delta f \in L^{\infty}(X)$.
Thus the equivalence between the Bochner inequalty and the reduced curvature-dimension condition in \cite[Theorem $7$]{eks} (see also \cite{ams}) implies that (ii) holds.

Since it is known that if (ii) holds, then (iii) holds,
finally, we assume that (iii) holds.
Then by the equivalence between a lower bound of Gigli's Ricci curvature and an $RCD$-condition \cite[Theorem $3.6.7$]{gigli}, Theorem \ref{gigliricci} yields that
\begin{align}\label{gitt}
\int_A\mathrm{Ric}_X(V, V)dH^n \ge K\int_A|V|^2dH^n
\end{align}
for any $V \in \mathrm{Test}T^1_0(X)$ and Borel subset $A$ of $X$.
By the positive answer to the question \textbf{(Q3)} in Remark \ref{gigliquestion}, (\ref{gitt}) holds for every $V \in L^2(X)$.
Then it is easy to check that this implies that (i) holds. 
$\,\,\,\,\,\,\,\Box$
\subsection{Proof of stabilities}
We prove Theorem \ref{stabilityricci}.

\textit{Proof of Theorem \ref{stabilityricci}.}

We give a proof in the case of lower bounds in (i) of Theorem \ref{stabilityricci} only because the proofs of other statements are similar.

It suffices to check that 
\begin{align}\label{nnuj}
\langle R_{X}, V \otimes W \otimes W \otimes V \rangle \ge K\left (|V|^2|W|^2-\langle V, W\rangle ^2\right)
\end{align}
on $B_r(x)$ a.e. sense for any $V, W \in L^{\infty}(TX)$. 

From the existence of an $L^p$-approximate sequence in \cite[Proposition 3.55]{holp}, there exist sequences  $V_i, W_i \in L^{\infty}(TX_i)$ with 
\[\sup_i (||V_i||_{L^{\infty}}+||W_i||_{L^{\infty}})<\infty\]
such that $V_i, W_i$ $L^p$-converge strongly to $V, W$ on $X$ for every $p \in (1, \infty)$, respectively.

Let $y \in B_r(x)$, let $t>0$ be a sufficiently small positive number, and let $y_i$ be a sequence in $B_r(x_i)$ with $y_i \stackrel{GH}{\to} y$.
Then by the assumption since  
\begin{align*}
&\frac{1}{H^n(B_t(y_i))}\int_{B_t(y_i)}\langle R_{X_{i}}, V_{i} \otimes W_{i} \otimes W_{i} \otimes V_{i} \rangle dH^n \\
&\ge \frac{K}{H^n(B_t(y_i))}\int_{B_t(y_i)}\left (|V_{i}|^2|W_{i}|^2-\langle V_i, W_i\rangle ^2\right)dH^n,
\end{align*}
letting $i \to \infty$ and $t \to 0$ with Theorem \ref{riem2} and the Lebesgue differentiation theorem yield (\ref{nnuj}).
$\,\,\,\,\,\,\,\Box$
\section{Applications}
\subsection{Spectral convergence of the Hodge Laplacian}
In order to establish the spectral convergence of $\Delta_{H, 1}$, we recall the space $W_{2p}$ for differential one-forms on our setting.
See \cite[Definition $6.2$]{hoell} in the case of tensor fields.
\begin{definition}[The space $W_{2p}$]
Let $X \in \overline{\mathcal{M}}$ and let $p >1$.
We denote by $W_{2p}(T^*X)$ the set of $\omega \in L^{2p}(T^*X)$ with the following three conditions:
\begin{enumerate}
\item $\omega$ is differentiable on $X$ a.e. in the sense of \cite{ho}.
\item $\nabla \omega \in L^2(T^0_2X)$, where $\nabla$ is the covariant derivative in the sense of \cite{ho}.
\item For every $x \in X$, every $r>0$, every harmonic function $h$ on $B_r(x)$, and every $s<r$, there exists $q>1$ such that
\[\langle \omega, dh\rangle|_{B_s(x)} \in H^{1, q}(B_s(x)).\]
\end{enumerate}
Define the complete norm $||\cdot ||_{W_{2p}}$ on $W_{2p}(T^*X)$ by
\[||\omega ||_{W_{2p}}:=||\omega ||_{L^{2p}}+||\nabla \omega ||_{L^2}.\] 
\end{definition}
We need the following previous results in \cite{hoell}:
\begin{proposition}\cite[Theorems $1.10$, $1.11$, $6.7$, $6.9$, $6.14$, and $7.8$]{hoell}\label{8877s}
Let $X \in \overline{\mathcal{M}}$. Then we have the following.
\begin{enumerate}
\item{(Sobolev space is a subspace of a $W_{2p}$-space)} We have $H^{1, 2}_H(T^*X) \subset W_{2n/(2n-1)}(T^*X)$. More precisely, for every $\omega \in H^{1,2}_H(T^*X)$ with 
\[\frac{1}{H^n(X)}\int_X\left(|\omega|^2+|d\omega |^2+|\delta \omega |^2\right)dH^n \le L,\]
we have
\[||\omega ||_{W_{2n/(2n-1)}}\le C(n, K_1, d, L).\]
\item{(Compatibility)} For every $\omega \in H^{1, 2}_H(T^*X)$, the derivatives $d\omega, \delta \omega, \nabla \omega$ in the sense of \cite{ho} coincide with Gigli's them in the sense of \cite{gigli}.
\item{(Rellich compactness theorem)} Let $X_i \stackrel{GH}{\to} X$ in $\overline{\mathcal{M}}$, let $p>1$, and let $\omega_i$ be a sequence in $W_{2p}(T^*X_i)$ with
\[\sup_i||\omega_i||_{W_{2p}}<\infty.\]
Then there exist a subsequence $i(j)$ and $\omega \in W_{2p}(T^*X)$ such that $\omega_{i(j)}$ $L^2$-converges strongly to $\omega$ on $X$ and that $\nabla \omega_{i(j)}$ $L^2$-converges weakly to $\nabla \omega$ on $X$.
\end{enumerate}
\end{proposition}
We prove a part of Theorem \ref{spectr} for $\Delta_{H, 1}$ and give the density of the space of eigen-one-forms in the Sobolev space:
\begin{theorem}\label{finitedimensional}
Let $X \in \overline{\mathcal{M}}$.
Then we have the following.
\begin{enumerate}
\item{(Finite dimensionalities of eigenspaces)} For every $\alpha \ge 0$,
the space 
\[E_{\alpha}^{H, 1}:=\{\omega \in \mathcal{D}^2(\Delta_{H, 1}, X); \Delta_{H, 1}\omega =\alpha \omega\}\]
is finite dimensional.
\item{(Discreteness and unboundedness of the spectrum)}
The set of the spectrum of $\Delta_{H, 1}$
\[S^{H, 1}:=\{\alpha \in \mathbf{R}_{\ge 0}; E_{\alpha}^{H, 1} \neq \{0\}\}\]
is discrete and unbounded.
\item{(Density of eigenspaces)}
The space
\[E^{H, 1}:=\bigoplus_{\alpha \ge 0}E_{\alpha}^{H, 1}\]
is dense in $H^{1, 2}_H(T^*X)$.
\end{enumerate}
\end{theorem}
\begin{proof}
We first prove (i).
Note that the proof is standard by using Proposition \ref{8877s}.

Let $\omega_i$ be a sequence in $E_{\alpha}^{H, 1}$ with 
$||\omega_i||_{L^2}=1$. 
Note that $||\omega_i||_{H^{1, 2}_H}^2=1+\alpha$.
Since every bounded sequence in $H^{1, 2}_H(T^*X)$ has a weak convergent subsequence in $H^{1, 2}_H(T^*X)$, Proposition \ref{8877s} with a weak continuity of the Hodge Laplacian in \cite[Theorem $7.16$]{hoell} yields that there exist a subsequence $i(j)$ and $\omega \in W_{2n/(2n-1)}(T^*X) \cap \mathcal{D}^2(\Delta_{H, 1}, X)$ such that $\omega_{i(j)} \to \omega$ in $L^2(T^*X)$ and that $\Delta_{H, 1}\omega_{i(j)}$ $L^2$-converges weakly to $\Delta_{H, 1}\omega$ on $X$, respectively.
In particular $\omega \in E_{\alpha}^{H, 1}$.

This gives that the set
\[\{ \omega \in E_{\alpha}^{H, 1}; ||\omega||_{L^2} = 1\}\]
is compact with respect to the norm $||\cdot ||_{L^{2}}$.
Recall that in general, a normed space $(V, ||\cdot ||)$ is a finite dimensional if and only if the set 
\[S:=\{v \in V; ||v||=1\}\]
is compact.
Therefore $E_{\alpha}^{H, 1}$ is finite dimensional.

Next we prove (ii).
Assume that $S^{H, 1}$ is not discrete. Then there exist a sequence $\alpha_i \in S^{H, 1}$ and $\alpha \in \mathbf{R}_{\ge 0}$ such that $\alpha_i \neq \alpha$ and that $\alpha_i \to \alpha$.
Let $\omega_i \in E_{\alpha_i}^{H, 1}$ with $||\omega_i||_{L^2}=1$.
Then by an argument similar to the proof of (i), without loss of generality we can assume that there exists the $L^2$-strong limit $\omega$ of $\omega_i$ on $X$ with $\omega \in E_{\alpha}^{H, 1}$.

Let $\{\eta_i\}_{i=1}^{\mathrm{dim}\,E_{\alpha}^{H, 1}}$ be a basis of $E^{H, 1}_{\alpha}$.
Since 
\[\int_X\langle \omega_i, \eta_l\rangle dH^n=0,\]
by letting $i \to \infty$, we have 
\[\int_X\langle \omega, \eta_l\rangle dH^n=0.\]
In particular 
\[\{\eta_i\}_{i=1}^{\mathrm{dim}\,E_{\alpha}^{H, 1}} \cup \{\omega\}\]
are linearly independent in $E_{\alpha}^{H, 1}$.
This is a contradiction.
Thus $S^{H, 1}$ is discrete.

The following is well-known on a smooth setting:
\begin{claim}[Min-max principle]\label{00oollpp}
We have 
\[\lambda^{H, 1}_k(X)=\inf_{E_k}\left( \sup_{\omega \in E_k}R^{H, 1}(\omega )\right),\]
where $E_k$ runs over all $k$-dimensional subspace of $H^{1, 2}_H(T^*X)$ and
\[R^{H, 1}(\omega ):=\frac{\int_X(|d\omega |^2+ |\delta \omega |^2)dH^n}{\int_X|\omega |^2dH^n}.\]
\end{claim}
We skip the proof of Claim \ref{00oollpp} because it is standard  by using Proposition \ref{8877s}.
See for instance \cite{sakai}.

It is easy to check that $H^{1, 2}_H(T^*X)$ is infinite dimensional. Thus Claim \ref{00oollpp} yields that $S^{H, 1}$ is an infinite set.
Since $S^{H, 1}$ is an infinite discrete set, $S^{H, 1}$ is unbounded.
This completes the proof of (ii).

On the other hand, the proof of (iii) is also standard  by using Proposition \ref{8877s} (see also \cite{sakai}).
Thus we have Theorem \ref{finitedimensional}.
\end{proof}
\begin{proposition}\label{1011}
Let $X \in \overline{\mathcal{M}}$ and let $T \in L^{\infty}(T^r_sX)$ (or $\omega \in L^{\infty}(\bigwedge^sT^*X)$, respectively).
Assume that the following three conditions hold.
\begin{enumerate}
\item $T$ (or $\omega$, respectively) is differentiable on $X$ a.e. in the sense of \cite{ho}.
\item $\nabla T \in L^2(T^r_{s+1}X)$ (or $\nabla \omega \in L^2(T^0_{s+1}X)$, respectively).
\item We have
\begin{align}\label{hin}
\left\langle T, \nabla f_1 \otimes \cdots \otimes \nabla f_r \otimes df_{r+1} \otimes \cdots \otimes df_{r+s} \right\rangle \in H^{1, 2}(X)
\end{align}
(or 
\begin{align}\label{hin2}
\left\langle \omega, df_1 \wedge \cdots \wedge df_s\right\rangle \in H^{1, 2}(X),
\end{align}
respectively) for any $f_i \in \mathcal{D}^2(\Delta, X)$ with $\Delta f_i \in L^{\infty}(X)$.
\end{enumerate}
Then we have $T \in H^{1, 2}_C(T^r_sX)$ (or $\omega \in H^{1, 2}_C(\bigwedge^sT^*X)$, respectively). 
\end{proposition}
\begin{proof}
We give a proof in the case when $r=0, s=1$, i.e. $\omega:=T$ is a differential one-form on $X$ only because the proof in the other case is similar.

Before starting the proof, we first recall several properties on the regular set $\mathcal{R}$ of $X$. 
The regular set $\mathcal{R}$ is defined by the set of points $x \in X$ such that every tangent cone at $x$ is isometric to the $n$-dimensional Euclidean space.
Then it is known that $\mathcal{R}$ is an open subset of $X$ and that $\mathcal{R}$ is a smooth manifold with a $C^{1, \alpha}$-Riemannian metric $g_{\mathcal{R}}$ on $\mathcal{R}$ for any $\alpha \in (0,1)$ such that the induced distance $d_{\mathcal{R}}$ on $\mathcal{R}$ coincides with the restriction of the original distance $d$ on $X$ to $\mathcal{R}$.
See \cite[Theorem 7.2]{ch-co1} and \cite[Corollary 3.10]{ch-co2} for the detail.

We now prove the following:
\begin{claim}\label{1012}
If $\omega$ has compact support in the regular set $\mathcal{R}$ of $X$, then $\omega \in H^{1, 2}_C(T^*X)$. 
\end{claim}
The proof is as follows.

By taking a partition of unity of $\mathcal{R}$ by smooth (or $C^2$) functions, without loss of generality we can assume that there exists a smooth coordinate patch $\Phi : U \to \mathbf{R}^n$ of $\mathcal{R}$ such that $\mathrm{supp}\,\omega \subset U$.
Let 
\[\omega = \sum_{i=1}^ng_idf_i,\]
where $\Phi:=(f_1, \ldots, f_n)$ and $g_i \in L^{\infty}(U)$ with the compact supports in $U$.
Moreover without loss of generality we can assume that $f_i \in C^{\infty}_c(\mathcal{R})$, where $C^{\infty}_c(\mathcal{R})$ denotes the set of smooth functions on $\mathcal{R}$ with compact supports.
In particular $f_i \in \mathcal{D}^2(\Delta, X)$ with $\Delta f_i \in L^{\infty}(X)$.

Therefore we have $\langle \omega, df_i \rangle \in H^{1, 2}(X)$.
Note that the matrix valued function
\[\left(G_{ij}\right)_{ij}(x):=\left( \langle df_i, df_j\rangle(x) \right)_{ij}\]
is invertible for every $x \in U$, that $\langle df_i, df_j\rangle  \in C^1(U) \cap L^{\infty}(U)$, and that
\[g_i= \sum_{j=1}^nG^{ij}\langle \omega, df_j \rangle,\]
where $(G^{ij})_{ij}$ is the inverse matrix of $(G_{ij})_{ij}$.
Since $g_i$ has compact support in $U$, we have $g_i \in H^{1, 2}(X) \cap L^{\infty}(X)$.

Then there exist sequences $g_{i, j} \in \mathrm{LIP}(X)$ with 
\[\sup_{i, j}||g_{i, j}||_{L^{\infty}}<\infty \]
such that $g_{i, j} \to g_i$ in $H^{1, 2}(X)$.
Let
\[\omega_j:=\sum_{i=1}^ng_{i, j}df_{i}.\]
By applying the existence of an $H^{1, p}$-approximation (3.2) to $g_{i, j}$ (or Corollary \ref{aavb} and Proposition \ref{bbghnv}), it is easy to check that $\omega_j \in H^{1, 2}_C(T^*X)$.
Thus from Leibniz'z rule, we see that $\omega_j$ is a Cauchy sequence in $H^{1, 2}_C(T^*X)$ and that $\omega_j \to \omega$ in $L^2(T^*X)$.
Thus $\omega \in H^{1, 2}_C(T^*X)$.
Thus we have Claim \ref{1012}.

From now on we finish the proof of Proposition \ref{1011}.

By \cite[Theorem 1.4]{chna} since
\[\mathrm{dim}_H\left(X \setminus \mathcal{R}\right) \le n-4,\]
it is not difficult to check that the canonical embedding 
\begin{align}\label{isom emb}
H^{1, 2}_c(\mathcal{R}) \hookrightarrow H^{1, 2}(X)
\end{align}
is isomorphic, where $H^{1, p}_c(\mathcal{R})$ is the completion of the set of Lipschitz functions on $\mathcal{R}$ with compact supports, denoted by $\mathrm{LIP}_c(\mathcal{R})$, with respect to the norm (\ref{wbt}).

For reader's convenience we give a proof of (\ref{isom emb}) as follows;
let us recall the definition of the Sobolev $2$-capacity $c_2(E)$ of a subset $E$ of $X$;
\[c_2(E):=\inf_{u \in \mathcal{A}(E)}\|u\|_{H^{1, 2}(X)}^2,\]
where 
\[\mathcal{A}(E):=\{u \in H^{1, 2}(X); u \ge 1\,\mathrm{on\, a\, neighbourhood\, of\,} E\}.\]
See for instance \cite[(1.4)]{KM} for the definition of the Sobolev $p$-capacity and see also \cite[Theorem 1.0.6]{KZ} for the equivalence between several definitions of Sobolev spaces.
Since $H^n$ is Ahlfors $n$-regular on $X$, we have for some $C>0$
\[c_2(E) \le CH^{n-2}(E)\]
for any $E \subset X$. See \cite[4.13 Theorem]{KM} for the proof.
In particular we have $c_2(X \setminus \mathcal{R})=0$.
Thus there exists a sequence $\psi_i \in H^{1, 2}(X)$ such that $\psi_i \in \mathcal{A}(X \setminus \mathcal{R})$ and that $\psi_i \to 0$ in $H^{1, 2}(X)$. Without loss of generality we can assume that $0 \le \psi_i \le 1$ and that $\psi_i \equiv 1$ on a neighbourhood of $X \setminus \mathcal{R}$ because the truncated functions $\psi_i \wedge 1 \vee 0$ also converge to $0$ in $H^{1, 2}(X)$.
Let $\phi_i:=1-\psi_i$.

In order to finish the proof of (\ref{isom emb}) we need to check that for any $f \in \mathrm{LIP}(X)$ we have $f \in H^{1, 2}_c(\mathcal{R})$.
For any $f \in \mathrm{LIP}(X)$ let us consider the sequence $\phi_if \in H^{1, 2}(X)$. 
Since $X \setminus \mathcal{R}$ is a compact subset of $X$ and $\phi_i \equiv 0$ holds on a neighbourhood of $X \setminus \mathcal{R}$, 
we have $\phi_if \in H^{1, 2}_c(\mathcal{R})$. 
Moreover since 
\[|\nabla (\phi_if)| \le |\nabla \phi_i| |f| +|\phi_i| |\nabla f|\]
with $f \in L^{\infty}(X)$ and $0 \le \phi_i \le 1$, we see that $\sup_i\|\phi_if\|_{H^{1, 2}(X)}<\infty$ and that $\phi_if \to f$ in $L^2(X)$.
Then by Mazur's theorem we have $f \in H^{1, 2}_c(\mathcal{R})$, which completes the proof of (\ref{isom emb}).

Let $\hat{\omega}_i:=\phi_i\omega$.
Since 
\[\langle \hat{\omega}_i,  df \rangle \in H^{1, 2}(X)\]
for every $f \in \mathcal{D}^2(\Delta, X)$ with $\Delta f \in L^{\infty}(X)$,
Claim \ref{1012} yields $\hat{\omega}_i \in H^{1, 2}_C(T^*X)$.
By the compatibility of the covariant derivatives between Gigli's one in \cite{gigli} and the author's one in \cite{ho} proven in \cite{hoell} (see also Proposition \ref{8877s}),
since
\[\nabla \omega_i =  \omega \otimes d\phi_i+ \phi_i\nabla \omega, \]
we have
\[\sup_i||\hat{\omega}_i||_{H^{1, 2}_C}<\infty.\]
Since $\hat{\omega}_i \to \omega$ in $L^2(T^*X)$, we see 
 that $\omega$ belongs to the closure of $\mathrm{TestForm}_1(X)$ in $W^{1, 2}_C(T^*X)$ with respect to the weak topology (recall that in the original definition of a `H-Sobolev space'  in \cite{gigli} by Gigli is defined by the closure of a test class in a `W-Sobolev spaces' with respect to the strong topology. In particular $H^{1, 2}_C(T^*X)$ is the closure of $\mathrm{TestForm}_1(X)$ in $W^{1, 2}_C(T^*X)$ with respect to the strong topology).
Since $\mathrm{TestForm}_1(X)$ is a linear subspace of $W^{1, 2}_C(T^*X)$, $\omega$ belongs to the closure of $\mathrm{TestForm}_1(X)$ in $W^{1, 2}_C(T^*X)$ with respect to the strong topology (c.f. Mazur's theorem), i.e. $\omega \in H^{1, 2}_C(T^*X)$.
This completes the proof.
\end{proof}
\begin{remark}
Under the same setting as in Proposition \ref{1011}, by the existence of an $H^{2, 2}$-approximate sequence (3.1), it is easy to check that the following three conditions are equivalent.
\begin{enumerate}
\item The assumption (\ref{hin}) (or (\ref{hin2}), respectively) holds.
\item For every $S \in \mathrm{Test}T^r_sX$ (or $S \in \mathrm{TestForm}_s(X)$, respectively),  
\begin{align}\label{hin3}
\left\langle T, S \right\rangle \in H^{1, 2}(X)
\end{align}
holds.
\item (\ref{hin3}) holds for every  $S \in \widetilde{\mathrm{Test}}T^r_sX$ (or $S \in \widetilde{\mathrm{TestForm}}_s(X)$, respectively).
\end{enumerate}
Moreover by an argument similar to the proof of \cite[Corollary 6.6]{hoell} it is not difficult to check that if one of conditions above holds, then (i) of Proposition \ref{1011} holds.
\end{remark}
\begin{proposition}\label{iiuy}
Let $X_i$ be a sequence in $\mathcal{M}$, let $X \in \overline{\mathcal{M}}$ be the Gromov-Hausdorff limit, let $\lambda_i$ be a bounded sequence in $\mathbf{R}_{\ge 0}$,  let $\omega_i$ be a sequence in $C^{\infty}(T^*X_i)$ with
\[\frac{1}{H^n(X_i)}\int_{X_i}|\omega_i|^2dH^n=1\]
and $\Delta_{H, 1}\omega_i=\lambda_i \omega_i$, and let $\omega$ be the $L^2$-weak limit on $X$.
Then we have the following.
\begin{enumerate}
\item The limit $\lambda := \lim_{i \to \infty}\lambda_i$ exists.
\item We have $\omega \in \mathcal{D}^2(\Delta_{H, 1}, X)  \cap L^{\infty}(T^*X)$ with 
\begin{align}\label{rwg}
||\omega ||_{L^{\infty}} \le C(n, K_1, d, \mu)
\end{align}
and
\begin{align}\label{rwg2}
\Delta_{H, 1}\omega=\lambda \omega,
\end{align}
where $\mu$ is an upper bound of $\lambda$.
\item $\omega_i, d\omega_i, \delta \omega_i, \nabla \omega_i$ $L^2$-converge strongly to $\omega, d\omega, \delta \omega, \nabla \omega$ on $X$, respectively.
\end{enumerate}
\end{proposition}
\begin{proof}
Note that (i) was already proven in \cite[Theorem 1.12]{hoell}.
Thus we first prove (ii).

For that,  by \cite[Theorems 1.12,  7.8 and Proposition 7.1]{hoell}, it suffices to check that $\omega \in H^{1, 2}_H(T^*X)$ holds.

Let $f \in \mathcal{D}^2(\Delta, X)$ with $\Delta f \in L^{\infty}(X)$, and let $g:=\Delta f$.
By the existence of an $L^2$-approximate sequence \cite[Proposition 3.56]{holp},
There exists a sequence $g_i \in L^{2}(X_i)$ such that $g_i$ $L^2$-converges strongly to $g$ on $X$ and that 
\[\int_{X_i}g_idH^n=0.\] 
Since the sequence $\hat{g}_i \in L^{\infty}(X_i)$ defined by
\begin{align*}
\hat{g}_i(x):=
\begin{cases}  ||g||_{L^{\infty}} \,\,\,\,\,\mathrm{if}\,g_i(x) \ge  ||g||_{L^{\infty}}, \\
g_i(x) \,\,\,\,\,\mathrm{if}\, |g_i(x)| < ||g||_{L^{\infty}}, \\
 -||g||_{L^{\infty}} \,\,\,\,\,\mathrm{if}\, \phi_i(x) \le -||g||_{L^{\infty}},
\end{cases}
\end{align*}
$L^2$-converges strongly to $g$ on $X$, without loss of generality we can assume that  
\[\sup_i||g_i||_{L^{\infty}}<\infty\]
(c.f. \cite[Proposition 3.24]{holp}).
Moreover by the smoothing via the heat flow, without loss of generality we can assume that $g_i \in C^{\infty}(X_i)$.

Let $f_i:=\Delta^{-1}g_i \in C^{\infty}(X_i)$.
From the Lipschitz regularity of solutions of Poisson's equations (2.2) and Theorem \ref{L2Hess}, we see that 
\[\sup_i||\nabla f_i||_{L^{\infty}}<\infty\]
and that $\mathrm{Hess}_{f_i}$ $L^2$-converges strongly to $\mathrm{Hess}_f$ on $X$.

On the other hand,  $L^{\infty}$-estimates of eigen-one-forms in \cite[Proposition $7.17$]{hoell} yields
\[\sup_i ||\omega_i||_{L^{\infty}}<\infty.\]
Therefore we have 
\[ \sup_i||\nabla \langle \omega_i, df_i \rangle||_{L^2}<\infty.\]
Since $\langle \omega_i, df_i \rangle$ $L^2$-converges strongly to $\langle \omega, df\rangle$ on $X$, the Rellich compactness theorem (2.3) yields 
$\langle \omega, df \rangle \in H^{1, 2}(X)$.
Therefore Propositions \ref{8877s} and \ref{1011} yield $\omega \in H^{1, 2}_C(T^*X)$.
Thus we have (ii).

By (ii),  \cite[Theorems 1.11, 7.8 and 7.22]{hoell} yield that $\omega_i, d\omega_i, \delta \omega_i$ $L^2$-converge strongly to $\omega, d\omega, \delta \omega$ on $X$, respectively and that $\nabla \omega_i$ $L^2$-converges weakly to $\nabla \omega$ on $X$.
By Theorem \ref{ellp} and Corollary \ref{aavb}, since
\begin{align*}
\int_{X_i}|\nabla \omega_i|^2dH^n&= \int_{X_i}\left( |d\omega_i |^2+|\delta \omega_i|^2 -\langle \mathrm{Ric}_{X_i}, \omega_i \otimes \omega_i \rangle\right)dH^n\\
&\to \int_{X_i}\left( |d\omega |^2+|\delta \omega|^2 -\langle \mathrm{Ric}_{X}, \omega \otimes \omega \rangle\right)dH^n\\
&=\int_X|\nabla \omega|^2dH^n
\end{align*}
as $i \to \infty$, we see that $\nabla \omega_i$ $L^2$-converges strongly to $\nabla \omega$ on $X$.
Thus we have (iii).
\end{proof}
The following completes the proof of Theorem \ref{spectr} for $\Delta_{H, 1}$.
\begin{theorem}[Spectral convergence of the Hodge Laplacian]\label{contihodge}
Let $X_i$ be a sequence in $\mathcal{M}$ and let $X \in \overline{\mathcal{M}}$ be the Gromov-Hausdorff limit.
Then
\[\lim_{i \to \infty}\lambda_{k}^{H, 1}(X_i)=\lambda_k^{H, 1}(X)\]
for every $k \ge 1$.
Moreover if $\omega \in L^2(T^*X)$ is the $L^2$-weak limit of a sequence of $\lambda^{H, 1}_k(X_i)$-eigen-one-forms $\omega_i \in C^{\infty}(T^*X_i)$ of $\Delta_{H, 1}$ with 
\[\frac{1}{H^n(X)}\int_X|\omega_i|^2dH^n=1,\]
then we see that $\omega$ is a $\lambda^{H, 1}_k(X)$-eigen-one-form of $\Delta_{H, 1}$, that $\omega_i, d\omega_i, \delta \omega_i, \nabla \omega_i$ $L^2$-converge strongly to $\omega, d\omega, \delta \omega, \nabla \omega$ on $X$, respectively, and that  
\begin{align}\label{linft}
||\omega ||_{L^{\infty}}\le C(n, K_1, d, \alpha),
\end{align}
where $\alpha$ is an upper bound of $\lambda^{H, 1}_k(X)$.
\end{theorem}
\begin{proof}
We first prove the upper semicontinuity:
\begin{align}\label{upphod}
\limsup_{i \to \infty}\lambda^{H, 1}_k(X_i)\le \lambda_k^{H, 1}(X).
\end{align}
Let $\epsilon>0$, let $E_k$ be a $k$-dimensional subspace of $H^{1, 2}_H(T^*X)$ with
\[\sup_{\eta \in E_k}R^{H, 1}(\eta )\le \lambda^{H, 1}_k(X)+\epsilon,\]
and let $\{\omega_i\}_{i=1}^{\mathrm{dim}\,E_k}$ be a basis of $E_k$.
By (ii) of Theorem \ref{14}, there exist sequences $\omega_{i, j} \in C^{\infty}(T^*X_j)$ such that $\omega_{i, j}, d\omega_{i, j}, \delta \omega_{i, j}$ $L^2$-converge strongly to $\omega_i, d\omega_i, \delta \omega_i$ on $X$, respectively.

Let $E_{k, j}:=\mathrm{span}\{\omega_{i, j}\}_{i=1}^{\mathrm{dim}\,E_k}$.
Note that $\mathrm{dim}\,E_{k, j}=k$ for every sufficiently large $j$.
Since it is easy to check that 
\begin{align}\label{rlim}
\lim_{i \to \infty}\sup_{\eta \in E_{k, i}}R^{H, 1}(\eta)=\sup_{\eta \in E_k}R^{H, 1}(\eta),
\end{align}
we have 
\begin{align*}
\limsup_{i \to \infty}\lambda^{H, 1}_k(X_i) \le \limsup_{i \to \infty}\left( \sup_{\eta \in E_{k, i}}R^{H, 1}(\eta )\right) = \sup_{\eta \in E_k}R^{H, 1}(\eta ) \le \lambda_k^{H, 1}(X) +\epsilon.
\end{align*}
Since $\epsilon$ is arbitrary, we have the upper semicontinuity (\ref{upphod}).

Next we prove the lower semicontinuity:
\begin{align}\label{lowerhod}
\liminf_{i \to \infty}\lambda_k^{H, 1}(X_i) \ge \lambda_k^{H, 1}(X).
\end{align}
Let $k \ge 1$.
For every $l \le k$, let $\omega_{l, i}$ be a sequence of $\lambda_l^{H, 1}(X_i)$-eigen-one-forms of $\Delta_{H, 1}$ in $C^{\infty}(T^*X_i)$ with $||\omega_{l, i}||_{L^2}=1$.
By Proposition \ref{iiuy} with the upper semicontinuity (\ref{upphod}), without loss of generality we can assume that there exist $\omega_l \in \mathcal{D}^2(\Delta_{H, 1}, X)$ such that $\omega_{l, i}, d\omega_{l, i}, \delta \omega_{l, i}$ $L^2$-converge strongly to $\omega_l, d\omega_l, \delta \omega_l$ on $X$, respecctively. 
In particular the space $E_k:=\mathrm{span}\{\omega_l\}_{l=1}^k$ is $k$-dimensional.

Let $E_{k, i}:=\mathrm{span}\{\omega_{l, i}\}_{l=1}^k$.
Since it is easy to check that (\ref{rlim}) also holds in this setting, we have
\[\liminf_{i \to \infty}\lambda_k^{H, 1}(X_i)=\liminf_{i \to \infty}\left( \sup_{\eta \in E_{k, i}}R^{H, 1}(\eta) \right)= \sup_{\eta \in E_{k}}R^{H, 1}(\eta) \ge \lambda_k^{H, 1}(X).\]
Thus we have the lower smicontinuity (\ref{lowerhod}).
Moreover the argument above also completes the proof of the rest of statements in Theorem \ref{contihodge}.
\end{proof}
\begin{remark}\label{boundei}
Note that most results given in \cite{hoell} are on general Ricci limit spaces. 
We say that a compact metric measure space $(X, \upsilon)$ is a \textit{an $(n, K)$-Ricci limit space} if there exist a sequence of real numbers $K_i$ with $K_i \to K$, and a sequence of $n$-dimensional closed Riemannian manifolds $X_i$ with $\mathrm{Ric}_{X_i} \ge K_i(n-1)$ such that $(X_i, H^n/H^n(X_i)) \stackrel{GH}{\to} (X, \upsilon)$. 

By an argument similar to the proof Theorem \ref{contihodge} with \cite{hoell}, we have the following:
\begin{enumerate}
\item  For a Ricci limit space $(X, \upsilon)$ with $\mathrm{diam}\,X >0$ the same conclusions as in Theorem \ref{finitedimensional} hold.
\item Let $(X_i, \upsilon_i)$ be a sequence of $(n, K)$-Ricci limit spaces and let $(X, \upsilon)$ be the Gromov-Hausdorff limit with $\mathrm{diam}\,X>0$.
Then we have the upper semicontinuity of the spectrum:
\begin{align}\label{uppsehod}
\limsup_{i \to \infty}\lambda^{H, 1}_k(X_i)\le \lambda_k^{H, 1}(X).
\end{align}
\end{enumerate}

Note that  by using (ii) above with Gromov's compactness theorem, we can prove the following estimate:
\begin{align}\label{unibo}
\lambda_k^{H, 1}(X) \le C(n, K, \tau, d, k)
\end{align}
for every $(n, K)$-Ricci limit space $(X, \upsilon)$ with $0<\tau\le \mathrm{diam}\,X \le d$, and every $k \ge 1$.
In particular we can choose a constant in the right hand side of (\ref{linft}) depending only on $n, K_1, d, k, v$.

We now give a proof of (\ref{unibo}) as follows.

Assume that the assertion (\ref{unibo}) is false.
Then there exists a sequence $(X_i, \upsilon_i)$ of $(n, K)$-Ricci limit spaces with $\tau \le \mathrm{diam}\,X_i \le d$ such that 
\begin{align}\label{x}
\lim_{i \to \infty}\lambda_k^{H, 1}(X_i)=\infty.
\end{align}
Gromov's compactness theorem states that there exist a subsequence $i(j)$ and the Gromov-Hausdorff limit $(X, \upsilon)$ of $(X_{i(j)}, \upsilon_{i(j)})$ with $\tau \le \mathrm{diam}\,X \le d$.
The upper semicontinuity (\ref{uppsehod}) with (\ref{x}) yields
\[\lambda_k^{H, 1}(X)=\infty.\]
This is a contradiction.
\end{remark}

\begin{remark}
This remark is based on referee's comments.
I would like to thank the referee for suggestions.

By the proof of Theorem \ref{contihodge} (more precisely, by the proof of Proposition \ref{1011}), even if we consider the following setting, we have the same conclusion as in Theorem \ref{contihodge} except for the $L^2$-strong convergence of $\nabla \omega_i$, i.e. 
the spectral convergence of $\Delta_{H, 1}$ also holds.
\begin{enumerate}
\item Let $X_i$ be a sequence of $n$-dimensional closed Riemannian manifolds with $\mathrm{Ric}_{X_i} \ge K$, and let $X$ be the noncollapsed compact Gromov-Hausdorff limit.
\item There exists an open subset $O$ of $X$ such that the following two conditions hold;
\begin{enumerate}
\item there exist $\{r(x)\}_{x \in O} \subset \mathbf{R}_{>0}$ and $\{\phi_i^x\}_{x \in O, i=1, 2, \ldots, n} \subset \mathcal{D}(\Delta, X)$ such that $B_{r(x)}(x) \subset O$, that $\langle \nabla \phi_i^x, \nabla \phi_j^x\rangle$ is continuous on $B_{r(x)}(x)$ and that $\{(B_{r(x)}(x), \phi^x)\}_{x \in O}$ is a $C^1$-atlas of $O$, where $\phi^x:=(\phi_1^x|_{B_{r(x)}(x)}, \ldots, \phi_n^x|_{B_{r(x)}(x)})$;
\item the inclusion
\[H^{1, 2}_c(O) \hookrightarrow H^{1, 2}(X)\]
gives an isomorphism between them.
\end{enumerate}
\end{enumerate}
Note that the second condition (b) is satisfied if $H^{n-2}(X \setminus O)=0$ holds.
See  \cite[Theorem 4.6]{KKM}, \cite[Theorem 4.13]{KM}, and \cite[Theorem 4.8]{Shanm}.
\end{remark}

We end this subsection by giving a proof of Theorem \ref{bettibetti}.

\textit{Proof of Theorem \ref{bettibetti}.}

Since (\ref{uppersemibetti}) is a direct consequence of Theorem \ref{contihodge}, it suffices to check the equivalence between (\ref{contibetti}) and (\ref{spectralgap}).

Assume that (\ref{contibetti}) holds.
Then Theorem \ref{contihodge} implies that 
\[\mu_{H, 1}(X_i) \to \mu_{H, 1}(X).\]
In particular this with Claim \ref{00oollpp} yields that (\ref{spectralgap}) holds.

Next we assume that (\ref{spectralgap}) holds.
Then by Claim \ref{00oollpp} $\mu_{H, 1}(X_i)$ does not converge to $0$.
In particular applying Theorem \ref{contihodge} for all zero eigenvalues with (\ref{uppersemibetti}) yields that (\ref{contibetti}) holds.
This completes the proof. $\,\,\,\,\,\,\,\,\Box.$
\subsection{Spectral convergence of the connection Laplacian}
In this section we finish the proofs of Theorems \ref{spectr} and \ref{wll}.
\begin{definition}[Connection Laplacian]
Let $X \in \overline{\mathcal{M}}$.
We denote by $\mathcal{D}^2(\Delta_{C, (r, s)}, X)$ (or $\mathcal{D}^2(\Delta_{C, s}, X)$, respectively) the set of $T \in H^{1, 2}_C(T^r_sX)$ (or $\omega \in H^{1, 2}_C(\bigwedge^sT^*X)$, respectively) such that there exists $\hat{T} \in L^2(T^r_sX)$ (or $\hat{\omega} \in L^2(\bigwedge^sT^*X)$, respectively) satisfying 
\[\int_X \langle \nabla T, \nabla S \rangle dH^n=\int_X\langle \hat{T}, S \rangle dH^n\]
(or 
\[\int_X \langle \nabla \omega, \nabla \sigma \rangle dH^n=\int_X\langle \hat{\omega}, \sigma \rangle dH^n,\]
respectively)
for every $S \in H^{1, 2}_C(T^r_sX)$ (or $\sigma \in H^{1, 2}_C(\bigwedge^sT^*X)$, respectively).
Since $\hat{T}$ (or $\hat{\omega}$, respectively) is unique if it exists, we denote it by $\Delta_{C, (r, s)}T$ (or $\Delta_{C, s}\omega$, respectively).
\end{definition}
\begin{proposition}\label{roughequal}
We have $\mathcal{D}^2(\Delta_{C, 1}, X)=\mathcal{D}^2(\Delta_{H, 1}, X)$.
Moreover for every $\omega \in \mathcal{D}^2(\Delta_{C, 1}, X)$, we have
\[ \Delta_{H, 1}\omega=\Delta_{C, 1}\omega + \mathrm{Ric}_X(\omega^*, \cdot ).\]
\end{proposition}
\begin{proof}
Let $\omega, \eta \in H^{1, 2}_C(T^*X)$ (recall that  by Corollary \ref{aavb}, $H^{1, 2}_C(T^*X)=H^{1, 2}_H(T^*X)$).
Theorem \ref{pointwiseboch} yields
\[\int_X \left(  \langle d\omega, d\eta \rangle + (\delta \omega)(\delta \eta)\right) dH^n=\int_X\left( \langle \nabla \omega, \nabla \eta \rangle + \langle \mathrm{Ric}_X, \omega \otimes \eta \rangle \right)dH^n.\]
This gives Proposition \ref{roughequal}.
\end{proof}
\begin{proposition}\label{mjoo}
Let $M$ be an $n$-dimensional closed Riemannian manifold with $\mathrm{Ric}_M \ge K$ and $\mathrm{diam}\,M\le d$, let $\alpha \le \beta$, and let $T \in C^{\infty}(T^r_sM)$ (or $T \in C^{\infty}(\bigwedge^sT^*M)$, respectively) be an $\alpha$-eigen-tensor field of $\Delta_{C, (r, s)}$, i.e. $\Delta_{C, (r, s)}T =\alpha T$, (or an $\alpha$-eigen-$s$-form of $\Delta_{C, s}$, respectively) with 
\[\frac{1}{H^n(M)}\int_M|T |^2dH^n=1.\]
Then 
\[||T ||_{L^{\infty}}\le C(n, K, d, \beta ).\] 
\end{proposition}
\begin{proof}
We give a proof in the case of $\Delta_{C, (r, s)}$ only because the proof in the other case is similar.
Since
\begin{align*}
-\frac{1}{2}\Delta |T |^2&=|\nabla T|^2-\langle \Delta_{C, (r, s)}T, T \rangle \\
&\ge - \alpha |T|^2 \\
&\ge -\beta |T|^2,
\end{align*}
Li-Tam's mean value inequality in \cite[Theorem $1.1$]{LT} (or \cite[Corollary $3.6$]{L}) gives
\[|T|^2 \le C(n, K, d, \beta)\frac{1}{H^n(M)}\int_M|T |^2dH^n \le C(n, K, d, \beta ).\]
\end{proof}
\begin{remark}
By an argument similar to the proof of Theorem \ref{pointwiseboch} we have the following:

Let $X \in \overline{\mathcal{M}}$ and let $T \in \mathcal{D}^2(\Delta_{C, (r, s)}, X)$ (or $T \in \mathcal{D}^2(\Delta_{C, s}, X)$, respectively).
Then we see that $|T|^2 \in \mathcal{D}_{2n/(2n-1), 1}(\Delta, X)$ and that
\[-\frac{1}{2}\Delta |T|^2 = |\nabla T|^2 - \langle \Delta_{C, (r, s)}T, T \rangle\]
(or
\[-\frac{1}{2}\Delta |T|^2 = |\nabla T|^2 - \langle \Delta_{C, s}T, T \rangle,\]
respectively) on $X$ a.e. sense.
\end{remark}
\begin{proposition}\label{conco}
Let $X_i$ be a sequence in $\mathcal{M}$, let $X \in \overline{\mathcal{M}}$ be the Gromov-Hausdorff limit, let $\lambda_i$ be a bounded sequence in $\mathbf{R}_{\ge 0}$,  let $T_i$ be a sequence in $C^{\infty}(T^r_sX_i)$ (or in $C^{\infty}(\bigwedge^sT^*X_i)$, respectively) with 
\[\frac{1}{H^n(X_i)}\int_{X_i}|T_i|^2dH^n=1\]
and $\Delta_{C, (r, s)}T_i=\lambda_i T_i$ (or $\Delta_{C, s}T_i=\lambda_iT_i$, respectively), and let $T$ be the $L^2$-weak limit of $T_i$ on $X$.
Then we have the following:
\begin{enumerate}
\item The limit $\lambda := \lim_{i \to \infty}\lambda_i$ exists.
\item We have $T \in \mathcal{D}^2(\Delta_{C, (r, s)}, X)  \cap L^{\infty}(T^r_sX)$ (or $T \in \mathcal{D}^2(\Delta_{C, s}, X)  \cap L^{\infty}(\bigwedge^sT^*X)$, respectively) with 
\begin{align}\label{rwg4}
||T ||_{L^{\infty}} \le C(n, K_1, d, \mu)
\end{align}
and
\begin{align}\label{rwg5}
\Delta_{C, (r, s)}T=\lambda T
\end{align}
(or $\Delta_{C, s}T=\lambda T$, respectively), where $\mu$ is an upper bound of $\lambda$.
\item $T_i, \nabla T_i$ $L^2$-converge strongly to $T, \nabla T$ on $X$, respectively.
\end{enumerate}
\end{proposition}
\begin{proof}
We give a proof in the case of $\Delta_{C, (r, s)}$ only because the proof in the other case is similar.

By the tensor fields version of Proposition \ref{8877s} (see \cite{hoell}) with an argument similar to the proof of (ii) of  Proposition \ref{iiuy}, 
since it is not difficult to check that the assumption of Proposition \ref{1011}  for $T$ holds,
we see that $T \in H^{1, 2}_C(T^r_sX)$, that $T_i$ $L^2$-converges strongly to $T$ and that $\nabla T_i$ $L^2$-converges weakly to $\nabla T$ on $X$.

Let $i(j)$ be a subsequence of $\mathbf{N}$.
Then there exists a subsequence $j(l)$ of $i(j)$ such that the limit $\lim_{l \to \infty}\lambda_{j(l)}$ exists.
Let $S \in \mathrm{Test}T^r_sX$.
By Theorem \ref{14} there exists a sequence $S_{j(l)} \in C^{\infty}(T^r_sX_{j(l)})$ such that $S_{j(l)}, \nabla S_{j(l)}$ $L^2$-converge strongly to $S, \nabla S$ on $X$, respectively.

Since
\[\int_{X_{j(l)}}\langle \nabla T_{j(l)}, \nabla S_{j(l)}\rangle dH^n =\lambda_{j(l)}\int_{X_{j(l)}}\langle S_{j(l)}, T_{j(l)}\rangle dH^n,\]
letting $l \to \infty$ yields
\[\int_X\langle \nabla T, \nabla S \rangle dH^n =\left(\lim_{l \to \infty}\lambda_{j(l)}\right) \int_X\langle S, T \rangle dH^n.\]
Thus $T \in \mathcal{D}^2(\Delta_{C, (r, s)}, X)$ with $\Delta_{C, (r, s)} T=\left( \lim_{l \to \infty}\lambda_{j(l)}\right) T$.
In particular $\lim_{l \to \infty}\lambda_{j(l)}=||\Delta_{C, (r, s)}T||_{L^2}$.
Since $i(j)$ is arbitraly, we have (i).

On the other hand, (ii) and (iii) follow from (i), Proposition \ref{mjoo} and the equalities
\[\frac{1}{H^n(X_i)}\int_{X_i}|\nabla T_i|^2dH^n = \frac{1}{H^n(X_i)}\int_{X_i}\langle T_i, \Delta_{C, (r, s)}T_i\rangle dH^n =\lambda_i.\]
\end{proof}
By Propositions \ref{mjoo}, \ref{conco} and an argument similar to the proof of Theorem \ref{contihodge}, we have the following:
\begin{theorem}\label{contrough}
Let $X \in \overline{\mathcal{M}}$.
Then we have the following.
\begin{enumerate}
\item{(Finite dimensionalities of eigenspaces)} For every $\alpha \ge 0$, the spaces
\[E^{C, (r, s)}_{\alpha}:=\{T \in \mathcal{D}^2(\Delta_{C, (r, s)}, X); \Delta_{C, (r, s)}T =\alpha T\}\]
and
\[E^{C, s}_{\alpha}:=\{\omega \in \mathcal{D}^2(\Delta_{C, s}, X); \Delta_{C, s}\omega =\alpha \omega\}\]
are finite dimensional.
\item{(Discreteness and unboundedness of the spectrums)} The sets of spectrums of $\Delta_{C, (r, s)}$ and $\Delta_{C, s}$,
\[S^{C, (r, s)}:=\{\alpha \in \mathbf{R}_{\ge 0}; E^{C, (r, s)}_{\alpha} \neq \{0\}\},\]
\[S^{C, s}:=\{\alpha \in \mathbf{R}_{\ge 0}; E^{C, s}_{\alpha} \neq \{0\}\},\]
are discrete and unbounded.
\item{(Densities of eigenspaces)} The sets 
\[E^{C, (r, s)}:=\bigoplus_{\alpha \ge 0}E^{C, (r, s)}_{\alpha}\]
and 
\[E^{C, s}:=\bigoplus_{\alpha \ge 0}E^{C, s}_{\alpha}\]
are dense in $H^{1, 2}_C(T^r_sX)$ and in $H^{1, 2}_C(\bigwedge^sT^*X)$, respectively.
\item{(Spectral convergence of the connection Laplacian)} 
Let us denote the spectrums of $\Delta_{C, (r, s)}$ and $\Delta_{C, s}$ by
\[0\le \lambda_1^{C, (r, s)} \le \lambda_2^{C, (r, s)} \le \cdots \to \infty\]
and
\[0\le \lambda_1^{C, s}\le \lambda_2^{C, s} \le \cdots \to \infty\]
counted with multiplicities, respectively.
If $X_i \stackrel{GH}{\to} X$ in $\overline{\mathcal{M}}$, then 
\[\lambda_k^{C, (r, s)}(X_i) \to \lambda_k^{C, (r, s)}(X)\]
and
\[\lambda_k^{C, s}(X_i) \to \lambda_k^{C, s}(X)\]
hold for every $k \ge 1$.
Moreover if $T \in L^2(T^r_sX)$ (or $T \in L^2(\bigwedge^sT^*X)$, respectively) is the $L^2$-weak limit of a sequence of $\lambda^{C, (r, s)}_k(X_i)$-eigen-tensor fields $T_i \in C^{\infty}(T^r_sX_i)$ of $\Delta_{C, (r, s)}$ (or $\lambda^{C, s}_k(X_i)$-eigen-differential $s$-forms $T_i \in C^{\infty}(\bigwedge^sT^*X_i)$ of $\Delta_{C, s}$, respectively) with 
\[\frac{1}{H^n(X_i)}\int_{X_i}|T_i|^2dH^n=1,\]
then we see that $T$ is a $\lambda^{C, (r, s)}_k(X)$-eigen-tensor field of $\Delta_{C, (r, s)}$ (or a $\lambda^{C, s}_k(X)$-eigen-differential $s$-form of $\Delta_{C, s}$, respectively), that $T_i, \nabla T_i$ $L^2$-converge strongly to $T, \nabla T$ on $X$, respectively, and that  
\[||T ||_{L^{\infty}}\le C(n, K_1,  d, \alpha),\]
where $\alpha$ is an upper bound of $\lambda_k^{C, (r, s)}(X)$ (or $\lambda_k^{C, s}(X)$, respectively).
\end{enumerate}
\end{theorem}

Note that Proposition \ref{mjoo}, Theorems \ref{contihodge} and \ref{contrough} imply Theorems \ref{spectr} and \ref{gradd}
\begin{remark}
By an argument similar to the proof of Theorem \ref{finitedimensional} with Remark \ref{boundei}, for a Ricci limit space $(X, \upsilon)$ with $\mathrm{diam}\,X>0$, similar finite dimensionalities of eigenspaces, discreteness and unboundedness of spectrums, densities of eigenspaces in Sobolev spaces hold for $\Delta^{C, (r, s)}$ and $\Delta^{C, s}$. 
\end{remark}
We end this subsection by giving a proof of Theorem \ref{estim}.

\textit{Proof of Theorem \ref{estim}.}

The proofs of the existence upper bounds $C_2$ are similar to an argument in Remark \ref{boundei}.
Thus we give a proof of the existence of lower bounds $C_1$ in the case of $\Delta_{H, 1}$ only because the proofs in other cases are also similar.

The proof is done by a contradiction.
Assume that the assertion is false.
Then there exist a sequence of positive integers $l_i$ with $l_i \to \infty$, and a sequence $X_i \in \overline{\mathcal{M}}$ such that 
\begin{align}\label{bouk}
\sup_i\lambda_{l_i}^{H, 1}(X_i)<\infty.
\end{align}
By the compactness of $\overline{\mathcal{M}}$, without loss of generality we can assume that the limit $X \in \overline{\mathcal{M}}$ of $X_i$ exists.
(v) of Theorem \ref{spectr} with (\ref{bouk}) yields
\[\sup_k\lambda_k^{H, 1}(X)<\infty.\]
This contradicts  (iv) of Theorem \ref{spectr}. $\,\,\,\,\,\,\,\,\,\,\Box$.
\subsection{Noncollapsed K\"ahler Ricci limit spaces with bounded Ricci curvature}
\subsubsection{Ricci potential}
In this section we consider the following setting:
\begin{enumerate}
\item[(4.1)] Let $X_i$ be a sequence in $\mathcal{M}$ and let $X \in \overline{\mathcal{M}}$ be the Gromov-Hausdorff limit.  
\item[(4.2)] Let $J_i$ be a sequence of complex structures on $X_i$ such that $(X_i, g_{X_i}, J_i)$ is a K\"ahler manifold.
\item[(4.3)] Let $J \in L^2(\mathrm{End}TX) (\simeq L^2(TX \otimes T^*X))$ be the $L^2$-weak limit of $J_i$ on $X$.  
\end{enumerate}
Then by \cite{ch-co1, cct}, we see that $J$ has the $C^{1, \alpha}$-regularity on $\mathcal{R}$ for every $\alpha \in (0, 1)$
(see also \cite{DS}).

In this section we use standard notation in K\"ahler geometry.
For example the gradient of a function $f$ is denoted by
\[\mathrm{grad}f\]
instead of $\nabla f$, 
the Hermitian metric $h_X$ of $X$ is defined by
\[h_X(v, w):=g_X(v, \overline{w}),\]
and so on.

On the other hand, we defined in \cite{FHS} the notion of $L^p$-convergence for $\mathbf{C}$-valued tensor fields with respect to the Gromov-Hausdorff topology, and gave applications.
We introduce several results of them.

Let 
\[(T^r_s)_{\mathbf{C}}X:=  T^r_sX \otimes_{\mathbf{R}} \mathbf{C}.\] 
We denote the canonical Hermitian metric on $(T^r_s)_{\mathbf{C}}X$ by $(h_X)^r_s$.
Let $L^p_{\mathbf{C}}((T^r_s)_{\mathbf{C}}A)$ be the set of $L^p$-sections to $(T^r_s)_{\mathbf{C}}X$ over a Borel subset $A$ of $X$. 
In particular if $r=s=0$, i.e. in the case of $\mathbf{C}$-valued functions, then $L^p_{\mathbf{C}}(A)$ denotes the set of $\mathbf{C}$-valued $L^p$-functions on $A$. 

The following were key notions in \cite{FHS}:
\begin{definition}\cite{FHS}
\begin{enumerate}
\item{($L^p$-convergence of $\mathbf{C}$-valued tensor fields)} Let $R>0$, let $p \in (1, \infty)$, let $x_i \stackrel{GH}{\to} x$, where $x_i \in X_i$ and $x \in X$, and let $T_i$ be a sequence in $L^p_{\mathbf{C}}((T^r_s)_{\mathbf{C}}B_R(x_i))$.
Then we say that \textit{$T_i$ $L^p$-converges weakly (or strongly, respectively) to a $\mathbf{C}$-valued tensor field $T \in L^p_{\mathbf{C}}((T^r_s)_{\mathbf{C}}B_R(x))$ on $B_R(x)$} if $T_i^{\mathrm{Re}}, T_i^{\mathrm{Im}}$ $L^p$-converge weakly (or strongly, respectively) to $T^{\mathrm{Re}}, T^{\mathrm{Im}}$ on $B_R(x)$, respectively, where  $T^{\mathrm{Re}}$ is the real part of $T$ and $T^{\mathrm{Im}}$ is the imaginary part of $T$. See \cite[Definition 2.5]{FHS}.
\item{(Sobolev spaces)} For any $p \in (1, \infty)$ and open subset $U$ of $X$, let $H^{1, p}_{\mathbf{C}}(U)$ be the set of $f \in L^2_{\mathbf{C}}(U)$ such that $f^{\mathrm{Re}}, f^{\mathrm{Im}} \in H^{1, p}(U)$. See subsetion 2.1 in \cite{FHS}.
\item{($\mathbf{C}$-valued Laplacian)}
Let $\mathcal{D}^2_{\mathbf{C}}(\Delta, U)$ be the set of $f \in H^{1, 2}_{\mathbf{C}}(U)$ such that there exists $g \in L^2_{\mathbf{C}}(U)$ such that 
\[\int_Uh_X^*(df, dh)dH^n=\int_Ug\overline{h}dH^n\]
for every $h \in \mathrm{LIP}_{c, \mathbf{C}}(U)$, where $\mathrm{LIP}_{c, \mathbf{C}}(U)$ is the set of $\mathbf{C}$-valued Lipschitz functions on $U$ with compact supports, and $h^*_X$ is the canonical Hermitian metric on $T^*X$, i.e. $h^*_X=(h_X)^0_1$.
Since $g$ is unique if it exists, we denote it by $\Delta f$ or, by $\Delta_{\mathbf{C}}f$. See \cite[Definition 3.2]{FHS} and (iv) of Proposition \ref{t74}.
\item{($\overline{\partial}$-Laplacian)}
Let $\mathcal{D}^2_{\mathbf{C}}(\Delta_{\overline{\partial}}, U)$ be the set of $f \in H^{1, 2}_{\mathbf{C}}(U)$ such that there exists $g \in L^2_{\mathbf{C}}(U)$ such that 
\[\int_Uh_X^*(\overline{\partial}f, \overline{\partial}h)dH^n=\int_Ug\overline{h}dH^n\]
for every $h \in \mathrm{LIP}_{c, \mathbf{C}}(U)$.
Since $g$ is unique if it exists, we denote it by $\Delta_{\overline{\partial}} f$. See \cite[Definition 3.5]{FHS}.
\end{enumerate}
\end{definition}
\begin{proposition}\cite{ch-co1, cct, holp, hoell, FHS}\label{t74}
We have the following.
\begin{enumerate}
\item $J$ is the $L^2$-strong limit of $J_i$ on $X$. 
\item $\nabla J=0$.
\item The sequence of K\"ahler forms $\omega_{X_i}$ of $X_i$ $L^2$-converges strongly to the K\"ahler form $\omega_X$ of $X$ defined by
\[\omega_X (v, w):=g_X(Jv, w).\]
\item  Let $f \in L^2(U)$. Then the following three conditions are equivalent:
\begin{enumerate}
\item $f \in \mathcal{D}^2_{\mathbf{C}}(\Delta_{\overline{\partial}}, U)$.
\item $f \in \mathcal{D}^2_{\mathbf{C}}(\Delta, U)$.
\item $f^{\mathrm{Re}}, f^{\mathrm{Im}} \in \mathcal{D}^2(\Delta, U)$.
\end{enumerate}
Moreover, if one of the conditions above holds, then
\begin{align*}
\Delta_{\mathbf{C}} f&=\Delta f^{\mathrm{Re}}+\sqrt{-1}\Delta f^{\mathrm{Im}}\\
&=\Delta_{\mathbf{C}} f^{\mathrm{Re}}+\sqrt{-1}\Delta_{\mathbf{C}} f^{\mathrm{Im}} \\
&=2\Delta_{\overline{\partial}}f^{\mathrm{Re}}+2\sqrt{-1}\Delta_{\overline{\partial}}f^{\mathrm{Im}}=2\Delta_{\overline{\partial}}f.
\end{align*}
\end{enumerate}
\end{proposition}
\begin{proof}
\cite[Theorem 6.19]{hoell} yields (i) and (ii) (note that (ii) also follows from \cite{ch-co1, cct}).
On the other hand \cite[Proposition3 3.3 and 3.6]{FHS} yields (iv).

Since $X$ is the noncollapsed limit of $X_i$, by \cite[Theorem 1.2]{holp}, $g_{X_i}$ $L^2$-converges strongly to $g_X$ on $X$.
Thus this with \cite[Proposition 3.70]{holp} and (i) yields (iii).
\end{proof}
As a corollary of (iv) of Proposition \ref{t74},
for every $f \in \mathcal{D}^2_{\mathbf{C}}(\Delta, X)$, we see that $f$ is weakly twice differentiable on $X$ in the sense of \cite{ho}.
In particular $\overline{\partial}f$ is differentiable on $U$ a.e. sense.
Thus
\[d\overline{\partial}f=\partial \overline{\partial} f \in L^2_{\mathbf{C}}((T^*X)'\wedge (T^*X)'')\]
is well-defined and can be written by using the Hessian of $f$, defined by $\mathrm{Hess}_f:=\mathrm{Hess}_{f^{\mathrm{Re}}}+\sqrt{-1}\mathrm{Hess}_{f^{\mathrm{Im}}}$, as follows:
\begin{align}\label{jjjj}
\sqrt{-1}\partial \overline{\partial}f(v, w)=\frac{1}{2}\left( \mathrm{Hess}_f(Jv, w)-\mathrm{Hess}_f(v, Jw)\right),
\end{align}
where $(T^*X)':=\{v \in T^*X; J^*v=\sqrt{-1}v\}$, $(T^*X)'':=\{v \in T^*X; J^*v=-\sqrt{-1}v\}$ and $J^*$ is the dual complex structure of $J$ of $T^*X$ by the Riemannian metric $g_X$.
\begin{proposition}\label{222334}
Let $f_i$ be a sequence in $\mathcal{D}^2_{\mathbf{C}}(\Delta, X_i)$ with
\[\sup_i\left( ||f_i||_{L^2_{\mathbf{C}}}+||\Delta f_i||_{L^2_{\mathbf{C}}}\right)<\infty,\]
and let $f \in \mathcal{D}^2_{\mathbf{C}}(\Delta, X)$ be the $L^2$-weak limit on $X$.
Then $\partial \overline{\partial} f_i$ $L^2$-converges weakly to $\partial \overline{\partial} f$ on $X$.
Moreover if $\Delta f_i$ $L^2$-converges strongly to $\Delta f$ on $X$, then $\partial \overline{\partial} f_i$ $L^2$-converges strongly to $\partial \overline{\partial} f$ on $X$.
\end{proposition}
\begin{proof}
By the $L^2$-weak convergence of Hessians in \cite[Theorem $4.13$]{hoell}, we see that $\mathrm{Hess}_{f_i}$ $L^2$-converges weakly to $\mathrm{Hess}_f$ on $X$. 
Thus (i) of Proposition \ref{t74}, (\ref{jjjj}) and \cite[Proposition 3.48]{holp} yield that $\partial \overline{\partial} f_i$ $L^2$-converges weakly to $\partial \overline{\partial} f$ on $X$.

Moreover, assume that $\Delta f_i$ $L^2$-converges strongly to $\Delta f$ on $X$.
Then Theorem \ref{L2Hess} yields that $\mathrm{Hess}_{f_i}$ $L^2$-converges strongly to $\mathrm{Hess}_f$ on $X$. 
Thus applying  \cite[Proposition 3.70]{holp} with (\ref{jjjj}) yields that $\partial \overline{\partial} f_i$ $L^2$-converges strongly to $\partial \overline{\partial} f$ on $X$.
\end{proof}
\begin{proposition}[Ricci form]\label{ricciform}
Define the Ricci form $\mathrm{Ric}_{\omega_X} \in L^{\infty}_{\mathbf{C}}(\bigwedge^2T^*_{\mathbf{C}}X)$ of $X$  by
\[\mathrm{Ric}_{\omega_X}(v, w):=\frac{1}{2}\left( \mathrm{Ric}_X(Jv, w)-\mathrm{Ric}_X(v, Jw)\right).\]
Then $\mathrm{Ric}_{\omega_{X_i}}$ $L^2$-converges weakly to $\mathrm{Ric}_{\omega_X}$ on $X$.
In particular
\begin{align}\label{kkkkiiii}
\mathrm{Ric}_{\omega_X}(v, w)=\mathrm{Ric}_{X}(Jv, w).
\end{align}
Moreover $\mathrm{Ric}_{\omega_{X_i}}$ $L^2$-converges strongly to $\mathrm{Ric}_{\omega_X}$ on $X$ if and only if $\mathrm{Ric}_{X_i}$ $L^2$-converges strongly to $\mathrm{Ric}_{X}$ on $X$.
\end{proposition}
\begin{proof}
By Theorem \ref{ellp}, (i) of Proposition \ref{t74} and \cite[Proposition 3.48]{holp}, we see that $\mathrm{Ric}_{\omega_{X_i}}$ $L^2$-converges weakly to $\mathrm{Ric}_{\omega_X}$ on $X$.

On the other hand, by \cite[Proposition 3.70]{holp},
 if $\mathrm{Ric}_{X_i}$ $L^2$-converges strongly to $\mathrm{Ric}_{X}$ on $X$, then  $\mathrm{Ric}_{\omega_{X_i}}$ $L^2$-converges strongly to $\mathrm{Ric}_{\omega_X}$ on $X$.
From now on we check that if $\mathrm{Ric}_{\omega_{X_i}}$ $L^2$-converges strongly to $\mathrm{Ric}_{\omega_X}$ on $X$, then $\mathrm{Ric}_{X_i}$ $L^2$-converges strongly to $\mathrm{Ric}_{X}$ on $X$.

For that we need the following.
\begin{claim}\label{er}
For a.e. $x \in X$, we see that $\mathrm{Ric}_{\omega_X}(u', u')=\mathrm{Ric}_{X}(u, u)$ for every $u \in T_{x}X$, where
$u=u'+u''$ is the decomposition of $u$ by
\[T_{\mathbf{C}}X = T'X \oplus T''X,\]
$T'X:=\{v \in TX; Jv =\sqrt{-1}v\}$ and $T''X:=\{ v \in TX; Jv=-\sqrt{-1}v\}$.
In particular
\[\mathrm{tr}_{\mathbf{C}}\,\mathrm{Ric}_{\omega_X}=\frac{1}{2}s_X,\]
where $\mathrm{tr}_{\mathbf{C}}$ is the (complex) trace, i.e.
\[\mathrm{tr}_{\mathbf{C}}\,\mathrm{Ric}_{\omega_X}(x)=\sum_i\mathrm{Ric}_{\omega_X}(e_i, \overline{e_i}),\]
and $\{e_i\}_i$ is an orthogonal basis of $T_x'X$.
\end{claim}
The proof is as follows.
Recall that this holds on a smooth K\"ahler manifold (c.f. page 6 in \cite{Tian99}).
Thus for any convergent sequences $x_{i, j} \stackrel{GH}{\to} x_i$, where $x_{i, j} \in X_j$ and $x_i \in X$,
since 
\begin{align*}
&\frac{1}{H^n(B_r(x_{1, j}))}\int_{B_r(x_{1, j})}\mathrm{Ric}_{\omega_{X_j}}(\mathrm{grad}'\,r_{x_{2, j}}, \mathrm{grad}'\,r_{x_{3, j}}) dH^n\\
&=
\frac{1}{H^n(B_r(x_{1, j}))}\int_{B_r(x_{1, j})}\mathrm{Ric}_{X_j}(\mathrm{grad}\,r_{x_{2, j}}, \mathrm{grad}\,r_{x_{3, j}})dH^n
\end{align*}
for every $r>0$ (recall that $r_x$ is the distance function from $x$), by letting $j \to \infty$ and $r \to 0$ with the Lebesgue differentiation theorem,
we have 
\begin{align}\label{66t5}
\mathrm{Ric}_{\omega_X}(\mathrm{grad}'\,r_{x_2}, \mathrm{grad}'\,r_{x_3})(x)=\mathrm{Ric}_{X}(\mathrm{grad}\,r_{x_2}, \mathrm{grad}\,r_{x_3})(x)
\end{align}
for a.e. $x \in X$,
where we used the fact that $\mathrm{grad}'\,r_{x_{i, j}}$ $L^2$-converges strongly to $\mathrm{grad}'\,r_{x_i}$ on $X$ (see \cite[Proposition 3.44]{holp} and \cite[Proposition 2.7]{FHS}).

Recall that for a.e. $x \in X$, there exist $y_1, \ldots, y_n \in X$ such that 
\[T_xX=\mathrm{span}_i\{\mathrm{grad}\, r_{y_i}(x)\}\]
(c.f. \cite[Theorem 3.1]{holip}).
Thus this with (\ref{66t5}) yields Claim \ref{er}.

Note that by an argument similar to that above we have (\ref{kkkkiiii}).

We are now in a position to finish the proof of Proposition \ref{ricciform}.
Assume that $\mathrm{Ric}_{\omega_{X_i}}$ $L^2$-converges strongly to $\mathrm{Ric}_{\omega_X}$ on $X$.

Then Claim \ref{er} yields that
\begin{align}\label{aaaaaat}
&\lim_{j \to \infty}\int_{B_r(x_{1, j})}|\mathrm{Ric}_{X_j}(\mathrm{grad}\,r_{x_{2, j}}, \mathrm{grad}\,r_{x_{3, j}})|^2dH^n \nonumber  \\
&=\lim_{j \to \infty}\int_{B_r(x_{1, j})}|\mathrm{Ric}_{\omega_{X_j}}(\mathrm{grad}'\,r_{x_{2, j}}, \mathrm{grad}'\,r_{x_{3, j}})|^2dH^n \nonumber \\
&=\int_{B_r(x_1)}|\mathrm{Ric}_{\omega_X}(\mathrm{grad}'r_{x_2}, \mathrm{grad}'r_{x_3})|^2dH^n \nonumber \\
&=\int_{B_r(x_1)}|\mathrm{Ric}_{X}(\mathrm{grad}\,r_{x_2}, \mathrm{grad}\,r_{x_3})|^2dH^n
\end{align} 
for any $x_{i, j} \stackrel{GH}{\to} x_i$ and $r>0$.

By an argument similar to the proof of (\ref{hesskey}) (or \cite[Theorem $6.9$]{hoell}) with (\ref{aaaaaat}), we have
\[\limsup_{i \to \infty}\int_{X_{i}}|\mathrm{Ric}_{X_{i}}|^2dH^n \le \int_{X}|\mathrm{Ric}_{X}|^2dH^n.\]
Thus $\mathrm{Ric}_{X_i}$ $L^2$-converges strongly to $\mathrm{Ric}_{X}$ on $X$.
\end{proof}
\begin{proposition}\label{exdol}
Let $U$ be an open subset of $X$ and let $f \in \mathcal{D}^2_{\mathbf{C}}(\Delta_{\overline{\partial}}, U)$.
Then
\begin{align}\label{trlap}
-\mathrm{tr}_{\mathbf{C}}\left(\,\sqrt{-1}\partial \overline{\partial} f\right)=\Delta_{\overline{\partial}} f.
\end{align}
\end{proposition}
\begin{proof}
Recall that the real version of this statement was already known in \cite[(1) of Theorem 1.9]{hoell}.
Since the proof of Proposition \ref{exdol} is essentially same to that, we give a sketch of the proof only. 
It consists of three steps as follows.
\begin{enumerate}
\item[(First step)] Let $p \in (1, \infty)$, let $X_i$ be a sequence in $\mathcal{M}$ with $X_i \stackrel{GH}{\to} X$, let $T_i$ be a sequence of $T_i \in L^p((T^*X_i)' \wedge (T^*X_i)'')$ and let $T \in L^p((T^*X)' \wedge (T^*X)'')$ be the $L^p$-weak limit on $X$.
Then prove that $\mathrm{tr}_{\mathbf{C}}T_i$ $L^p$-converges weakly to $\mathrm{tr}_{\mathbf{C}}T$ on $X$ (c.f. \cite[Proposition 3.72]{holp}).
\item[(Second step)] Assume that $U=X$.  By using the continuity of solutions of Poisson's equations (2.8) and the elliptic regularity theorem,  prove that there exists a sequence $f_i \in C^{\infty}_{\mathbf{C}}(X_i)$ such that $f_i, df_i, \Delta f_i$ $L^2$-converge strongly to $f, df, \Delta f$ on $X$, respectively.
Then since (\ref{trlap}) holds on a smooth K\"ahler manifold, Proposition \ref{222334} with the first step yields (\ref{trlap}) in the case when $U=X$.
\item[(Final step)]For general $U$, by using a good cut-off function constructed in \cite[Theorem 6.33]{ch-co}  by Cheeger-Colding and the second step, complete the proof of Proposition \ref{exdol}.
\end{enumerate}
\end{proof}
\begin{remark}
By the proof of Proposition \ref{exdol}, the same conclusion as in Proposition \ref{exdol} holds even if $(X, g_X, J)$ is a noncollapsed K\"ahler Ricci limit space in the sense of \cite{FHS}.
\end{remark}
\begin{theorem}[Ricci potential]\label{riccipotential}
Assume that for every $i$, 
\[\lambda_{i} \omega_{X_i} \in 2\pi c_1(X_i)\]
for some $\lambda_i \in \mathbf{C}$, where $c_1(X_i)$ is the first Chern class of $X_i$. 
Then we have the following.
\begin{enumerate}
\item{(Existence)} There exist $\lambda \in \mathbf{C}$ and $F \in \mathcal{D}^2_{\mathbf{C}}(\Delta, X)$ such that 
\begin{align}\label{potential}
\mathrm{Ric}_{\omega_X}-\lambda \omega_X = \sqrt{-1}\partial \overline{\partial} F
\end{align}
holds. Moreover $\lambda$ is unique. In fact
\begin{align}\label{443}
\lambda=\frac{1}{2nH^n(X)}\int_{X}s_{X}dH^n.
\end{align}
We call $F$ is a Ricci potential of $(X, g_X, J)$.
\item{(Uniqueness)} If $F_1$ and $F_2$ are Ricci potentials of $(X, g_X, J)$, 
then there exists $c \in \mathbf{C}$ such that 
\[F_1 \equiv F_2 +c.\]
\item{(Lipschitz regularity)} There exists $\mathbf{R}$-valued Ricci potential $F \in \mathcal{D}^2(\Delta, X)$ such that $F \in \mathrm{LIP}(X)$ holds.
Thus it is well-defined that a Ricci potential is said to be \textit{normalized} if
\[\int_Xe^FdH^n=H^n(X).\]
Since the normalized Ricci potential is unique, we denote it by $F_X$ for short.
\end{enumerate}
\end{theorem}
\begin{proof}
We first prove the uniqueness of $\lambda$.
By taking the (complex) trace of (\ref{potential}), Proposition \ref{exdol} yields
\[\frac{1}{2}s_{X}-n\lambda =-\Delta_{\overline{\partial}} F.\]
Integrating this on $X$ with the equality $2\Delta_{\overline{\partial}}=\Delta$ ((iv) of Proposition \ref{t74}) yields the uniqueness (\ref{443}).

Next we prove (iii).
It is well-known as $\partial \overline{\partial}$-lemma that for every $i$, there exists $F_i \in C^{\infty}(X_i)$ with
\[\int_{X_i}F_idH^n=0\]
such that 
\begin{align}\label{sssss}
\mathrm{Ric}_{\omega_{X_i}}-\left(\frac{1}{2nH^n(X_i)}\int_{X_i}s_{X_i}dH^n\right)\omega_{X_i}=\sqrt{-1}\partial \overline{\partial} F_i.
\end{align}
Thus by taking the trace of (\ref{sssss}) we have 
\[||\Delta F_i||_{L^{\infty}} \le C(n, K_1, K_2).\]
Thus by the Lipschitz regularity of solutions of Poisson's equations (2.2) we have
\[||\nabla F_i||_{L^{\infty}}\le C(n, K_1, K_2, d).\]
On the other hand, by the Rellich compactness theorem (2.3) and the closedness of the Dirichlet Laplacian (2.4), without loss of generality we can assume that there exists the $L^2$-strong limit $F \in \mathcal{D}^2(\Delta, X)$ of $F_i$ on $X$.
Thus letting $i \to \infty$ in (\ref{sssss}) with Theorem \ref{contsca} and Proposition \ref{222334} yields (iii). 
In particular we have (i).

Finally we prove (ii).
Since $\partial \overline{\partial}F_1=\partial \overline{\partial}F_2$, by taking the complex trace of this, Proposition \ref{exdol} yields
\[\Delta_{\overline{\partial}} F_1=\Delta_{\overline{\partial}}F_2.\]
The equality $2\Delta_{\overline{\partial}}=\Delta$ yields that $F_1-F_2$ is a harmonic function on $X$.
Thus we have (ii).
\end{proof}
\begin{proposition}[Almost constant scalar curvature implies almost Einsten on a singular space]\label{asae}
Under the same assumption as in Theorem \ref{riccipotential},  if
\[\frac{1}{H^n(X)}\int_X|s_X-c|^2dH^n\le \epsilon \]
for some $c \in \mathbf{R}$, then 
\[\frac{1}{H^n(X)}\int_X|\mathrm{Ric}_{\omega_X}-\lambda \omega_X|^2dH^n \le \epsilon +C(n, K_1, d)\epsilon^{1/2}.\]
In particular if $s_X \equiv c$, then $\mathrm{Ric}_{\omega_X}\equiv \lambda \omega_X$.
\end{proposition}
\begin{proof}
Note that
\begin{align}\label{tt6}
\frac{1}{H^n(X)}\int_X\left| s_X-\frac{1}{H^n(X)}\int_Xs_XdH^n\right|^2dH^n=\inf_{\hat{c} \in \mathbf{R}}\left(\frac{1}{H^n(X)}\int_X\left| s_X-\hat{c}\right|^2dH^n\right)\le \epsilon.
\end{align}
Let $F$ be the ($\mathbf{R}$-valued) Ricci potential of $(X, g_X, J)$ with
\[\frac{1}{H^n(X)}\int_XFdH^n=0.\]
Then from (i) of Theorem \ref{riccipotential} with (\ref{tt6}) we have 
\[\frac{1}{H^n(X)}\int_X|\Delta F|^2dH^n = \frac{1}{H^n(X)}\int_X|s_X-2n\lambda |^2dH^n \le \epsilon.\]
Thus the Poincar\'e inequality yields
\begin{align*}
\frac{1}{H^n(X)}\int_X|F|^2dH^n &\le \frac{C(n, K_1, d)}{H^n(X)}\int_X|\nabla F|^2dH^n\\
&=\frac{C(n, K_1, d)}{H^n(X)}\int_XF\Delta FdH^n \\
&\le C(n, K_1, d)\left( \frac{1}{H^n(X)}\int_X|F|^2dH^n\right)^{1/2}\left( \frac{1}{H^n(X)}\int_X|\Delta F|^2dH^n\right)^{1/2}\\
&\le C(n, K_1, d)\epsilon^{1/2}\left( \frac{1}{H^n(X)}\int_X|F|^2dH^n\right)^{1/2}.
\end{align*}
Thus
\[\frac{1}{H^n(X)}\int_X|\nabla F|^2dH^n \le C(n, K_1, d)\epsilon^{1/2}.\]
Theorem \ref{boch} yields 
\begin{align*}
\frac{1}{H^n(X)}\int_X|\mathrm{Hess}_F|^2dH^n &= \frac{1}{H^n(X)}\int_X\left(|\Delta F|^2 - \mathrm{Ric}_X(\nabla F, \nabla F)\right)dH^n\\
&\le \epsilon -K_1\frac{1}{H^n(X)}\int_X|\nabla F|^2dH^n \\
&\le \epsilon+C(n, K_1, d)\epsilon^{1/2}.
\end{align*}
Thus we have
\[\frac{1}{H^n(X)}\int_X|\mathrm{Ric}_{\omega_X}-\lambda \omega_X|^2dH^n = \frac{1}{H^n(X)}\int_X|\partial \overline{\partial}F|^2dH^n \le  \epsilon+C(n, K_1, d)\epsilon^{1/2}.\]
\end{proof}
\begin{corollary}
Under the same assumption as in Theorem \ref{riccipotential}, if $s_X \in H^{1, 2}(X)$ and
\[\frac{1}{H^n(X)}\int_X|ds_X|^2dH^n \le \epsilon,\]
then
\[\frac{1}{H^n(X)}\int_X|\mathrm{Ric}_{\omega_X}-\lambda \omega_X|^2dH^n \le C_1(n, K_1, d)\epsilon +C_2(n, K_1, d)\epsilon^{1/2}.\]
\end{corollary}
\begin{proof}
This is a direct consequence of Proposition \ref{asae} and the Poincar\'e inequality.
\end{proof}
We are now in a position to give the main result in this section, which gives an answer to the question \textbf{(Q2)} in subsection 3.3.
\begin{theorem}\label{equivalence}
Under the same assumption as in Theorem \ref{riccipotential}, 
 we see that $F_{X_i},  dF_{X_i}$ $L^2$-converge strongly to $F_X, dF_X$ on $X$, respectively and that $\Delta F_{X_i}, \mathrm{Hess}_{F_{X_i}}$ $L^2$-converge weakly to $\Delta F_X, \mathrm{Hess}_{F_X}$ on $X$, respectively.
Moreover the following four conditions are equivalent:
\begin{enumerate}
\item $\mathrm{Ric}_{X_i}$ $L^2$-converges strongly to $\mathrm{Ric}_{X}$ on $X$.
\item $s_{X_i}$ $L^2$-converges strongly to $s_{X}$ on $X$.
\item $\Delta F_{X_i}$ $L^2$-converges strongly to $\Delta F_X$ on $X$.
\item $\mathrm{Hess}_{F_{X_i}}$ $L^2$-converges strongly to $\mathrm{Hess}_{F_X}$ on $X$.
\end{enumerate}
\end{theorem}
\begin{proof}
By the proof of (iii) of Theorem \ref{riccipotential}, it suffices to check the equivalence between (i), (ii), (iii), and (iv).

Theorem \ref{L2Hess} yields that if (iii) holds, then (iv) holds.

On the other hand, as we claimed in subsection 3.3, if (i) holds, then (ii) holds. 

Similarly by (1.3) if (iv) holds, then (iii) holds (c.f. \cite[Proposition 3.74]{holp}). 

Next assume that (ii) holds.
Then since 
\[\frac{1}{2}s_{X_i}-\frac{1}{2nH^n(X_i)}\int_{X_i}s_{X_i}dH^n=-\Delta_{\overline{\partial}}F_{X_i},\]
the equality $\Delta = 2\Delta_{\overline{\partial}}$ yields that (iii) holds.

Finally assume that (iii) holds. 
Proposition \ref{222334} yields that $\partial \overline{\partial}F_{X_i}$ $L^2$-converges strongly to $\partial \overline{\partial}F_{X}$ on $X$.
Since 
\[\mathrm{Ric}_{\omega_{X_i}}= \omega_{X_i}+\sqrt{-1}\partial \overline{\partial}F_{X_i},\]
by (iii) of Proposition \ref{t74}, we see that $\mathrm{Ric}_{\omega_{X_i}}$ $L^2$-converges strongly to $\mathrm{Ric}_{\omega_X}$ on $X$.
Therefore, Proposition \ref{ricciform} yields that (i) holds.
This completes the proof.  
\end{proof}
\begin{remark}
We can define the \textit{bisectional curvature} $R_X^b$ of $X$ on our setting (4.1)-(4.3) by the standard way of K\"ahler geometry:
\[R_X^b(u, \overline{u}, v, \overline{v}):=R_X(\alpha, \beta, \beta, \alpha)+R_X(\alpha, J\alpha, J\beta, \alpha),\]
where $\alpha, \beta \in T_xX$ with $\alpha \perp \beta$, $|\alpha|=|\beta|=1$,  
\[u=\frac{1}{\sqrt{2}}\left(\alpha - \sqrt{-1}J\alpha \right)\]
and
\[v=\frac{1}{\sqrt{2}}\left(\beta - \sqrt{-1}J\beta \right).\]
Then by arguments similar to the proofs of Theorems \ref{curvaturetensor} and \ref{riem2},
we can prove that $R_X^b \in L^q$ for any $q<2$
and that they behave continuously with respect to the $L^q$-weak topology. 
\end{remark}
\subsubsection{Fano-Ricci limit spaces}
In this subsection, besides (4.1)-(4.3), we add the following assumption:
\begin{enumerate}
\item[(4.4)] $X_i$ is a Fano manifold with $\omega_{X_i} \in 2\pi c_1(X_i)$ for every $i$.
\end{enumerate}
Note that the equation of the Ricci potential is  
\begin{align}\label{55467}
\mathrm{Ric}_X-\omega_X=\sqrt{-1}\partial \overline{\partial}F_X
\end{align}
and that the following Weitzenb\"ock formula on the smooth setting is known:
\begin{align}\label{pprrb}
\int_{X_i}|\Delta^{F_{X_i}}_{\overline{\partial}}f_i|^2dH^n_{F_{X_i}} = \int_{X_i}|\nabla''\mathrm{grad}'f_i|^2dH^n_{F_{X_i}}+\int_{X_i}|\overline{\partial}f_i|^2dH^n_{F_{X_i}}
\end{align}
for every $f_i \in C^{\infty}_{\mathbf{C}}(X_i)$, where $dH^n_{F_{X_i}}:=e^{F_{X_i}}dH^n$ (see page 41 in \cite{Futaki}).

In \cite[Theorem 4.1]{FHS}  we established  the following Weitzenb\"ock inequality:
\begin{align}\label{wet}
\int_X|\Delta^{F_X}_{\overline{\partial}}f|^2dH^n_{F_X}\ge \int_X|\nabla''\mathrm{grad}'f|^2dH^n_{F_X}+\int_X|\overline{\partial}f|^2dH^n_{F_X}
\end{align}
holds for every $f \in \mathcal{D}^2_{\mathbf{C}}(\Delta^{F_X}_{\overline{\partial}}, X)$ without the assumption of a uniform upper bound on Ricci curvature, and gave applications to the study of holomorphic vector fields, where $\mathcal{D}^2_{\mathbf{C}}(\Delta^{F_X}_{\overline{\partial}}, X)$ is the domain of the weighted Laplacian $\Delta^F_{\overline{\partial}}$, and
 $\nabla''\mathrm{grad}'f \in L^2_{\mathbf{C}}(T_{\mathbf{C}}'X \otimes (T^*_{\mathbf{C}}X)'')$ is defined by satisfing 
\begin{align*}
&\int_Xf_0\,g_X(\nabla''\mathrm{grad}'f, \mathrm{grad} f_1 \otimes df_2)\,dH^n  \\
&=\int_X\left( - \mathrm{div}(f_0\,\mathrm{grad}''f_2)\,g_X(V, \mathrm{grad} f_1) - f_0\,g_X(V, \nabla_{\mathrm{grad}''f_2}\mathrm{grad} f_1\right)\,dH^n 
\end{align*}
for any $f_i \in \mathrm{Test}_{\mathbf{C}}F(X):=\{f=f^{\mathrm{Re}}+\sqrt{-1}f^{\mathrm{Im}}; f^{\mathrm{Re}}, f^{\mathrm{Im}} \in \mathrm{Test}F(X)\}$.
We skip the definition of the divergence. 

Note that $\mathcal{D}^2_{\mathbf{C}}(\Delta^F_{\overline{\partial}}, X)=\mathcal{D}^2_{\mathbf{C}}(\Delta, X)$ with
\begin{align}\label{explicite}
\Delta_{\overline{\partial}}^{F_X}f=\Delta_{\overline{\partial}} f - h_X^*(\overline{\partial} f, \overline{\partial}F_X).
\end{align}
for every $f \in \mathcal{D}^2_{\mathbf{C}}(\Delta, X)$ on our setting,
and that $\nabla''\mathrm{grad}'\,f$ coincides with the ordinary one on a smooth setting, i.e. it is the $(1, 1)$-part of $\nabla \nabla f$.
See subsections 3.1 and 3.3 in \cite{FHS} for the detail.

The main result of this section is the following Weitzenb\"ock formula on our setting which was announced in \cite{FHS}:
\begin{theorem}[Weitzenb\"ock formula]\label{23}
For every $f \in \mathcal{D}^2_{\mathbf{C}}(\Delta, X)$,
we see that $\nabla''\mathrm{grad}'f$ coincides with the $(1, 1)$-part of $\nabla \nabla f$ and that
\[\int_X|\Delta^{F_X}_{\overline{\partial}}f|^2dH^n_{F_X}=\int_X|\nabla''\mathrm{grad}'f|^2dH^n_{F_X}+\int_X|\overline{\partial}f|^2dH^n_{F_X}\]
holds. 
\end{theorem}
\begin{proof}
From the proof of the  Weitzenb\"ock inequality (\ref{wet}), there exists a sequence of $f_i \in C^{\infty}_{\mathbf{C}}(X_i)$ such that $f_i, \overline{\partial}f_i, \Delta^{F_i}_{\overline{\partial}} f_i$ $L^2$-converge strongly to $f, \overline{\partial}f, \Delta^F_{\overline{\partial}}f$ on $X$, respectively and that $\nabla''\mathrm{grad}'f_i$ $L^2$-converges weakly to $\nabla''\mathrm{grad}'f$ on $X$.
(iv) of Proposition \ref{t74}, Theorem \ref{equivalence} and (\ref{explicite}) yield that $\Delta f_i$ $L^2$-converges strongly to $\Delta f$ on $X$.

Thus by Theorem \ref{L2Hess}, $\mathrm{Hess}_{f_i}$ $L^2$-converges strongly to $\mathrm{Hess}_f$ on $X$.
In particular $(1, 1)$-parts of $\nabla \nabla f_i$ $L^2$-converges strongly to that of $\nabla \nabla f$ on $X$ (c.f. \cite[Proposition 2.7]{FHS}).
Since $\nabla''\mathrm{grad}'f_i$ coincides with the $(1, 1)$-part of $\nabla \nabla f_i$, by letting $i \to \infty$ in (\ref{pprrb}), this completes the proof.
\end{proof}
\begin{corollary}
For every $f \in \mathcal{D}^2_{\mathbf{C}}(\Delta, X)$, the following two conditions are equivalent:
\begin{enumerate}
\item $\nabla''\mathrm{grad}'f=0$ with
\[\int_XfdH^n_{F_X}=0.\]
\item $\Delta^{F_X}_{\overline{\partial}}f=f$.
\end{enumerate}
\end{corollary}
\begin{proof}
It suffices to check that if (i) holds, then (ii) holds (c.f. \cite[Corollary $4.2$]{FHS}).

Assume that (i) holds.
Then by Theorem \ref{23} we have
\begin{align}\label{pko}
\int_X(\Delta^{F_X}_{\overline{\partial}}f)(\overline{\Delta^{F_X}_{\overline{\partial}}g})dH^n_{F_X}&=\int_X(h_X)^1_1(\nabla''\mathrm{grad}'f, \nabla''\mathrm{grad}'g)dH^n_{F_X}+\int_X h_X^*(\overline{\partial}f, \overline{\partial}g)dH^n_{F_X} \nonumber \\
&=\int_X h_X^*(\overline{\partial}f, \overline{\partial}g)dH^n_{F_X} \nonumber \\
&=\int_X f\overline{\Delta^{F_X}_{\overline{\partial}}g}dH^n_{F_X}
\end{align}
for every $g \in \mathcal{D}^2_{\mathbf{C}}(\Delta, X)$.

Let $h \in L^2_{\mathbf{C}}(X)$ with 
\[\int_XhdH^n_{F_X}=0,\]
and let $k:=(\Delta^{F_X}_{\overline{\partial}})^{-1}h$ (c.f. \cite[Proposition $3.16$]{FHS}).
Then applying (\ref{pko}) for $g=k$ yields
\[\int_X(\Delta^{F_X}_{\overline{\partial}}f-f)\overline{h}dH^n_{F_X}=0.\]
Since $h$ is arbitrary and 
\[\int_X(\Delta^{F_X}_{\overline{\partial}}f-f)dH^n_{F_X}=0,\]
we have (ii).
\end{proof}
In \cite{FHS} we gave the notion of an \textit{almost smooth Fano-Ricci limit space} which is closely related to $\mathbf{Q}$-Fano varieties (see Section 5 in \cite{FHS} for the precise definition).
We end this subsection by giving a sufficient condition for being an almost smooth Fano-Ricci limit space on our setting.
\begin{corollary}
Assume that $(\mathcal{R}, g_X|_{\mathcal{R}}, J|_{\mathcal{R}})$ is a smooth K\"ahler manifold and 
that every holomorphic function on $\mathcal{R}$ is constant. 
Then $(X, g_X, J, F_X)$ is an almost smooth Fano-Ricci limit space in the sense of \cite{FHS}.
\end{corollary}
\begin{proof}
By Proposition \ref{curvaturecompatibility},  the equation (\ref{55467}) and the elliptic regularity theorem, we see that $F_X$ is smooth on $\mathcal{R}$.

On the other hand, by an argument similar to the proof of \cite[Proposition $6.1$]{FHS}, if $u \in H^{1, 2}_{\mathbf{C}}(X)$ satisfies that $\mathrm{grad}'u|_{\mathcal{R}}$ is a holomorphic vector field, then $u \in \mathcal{D}^2_{\mathbf{C}}(\Delta, X)=\mathcal{D}^2_{\mathbf{C}}(\Delta^{F_X}_{\overline{\partial}}, X)$.  
This completes the proof. 
\end{proof}
\section{Questions}
In this section we give several open problems.

First, recall that Naber gave in \cite{na} a notion of \textit{bounded Ricci curvature on a metric measure space}:
\begin{align}\label{twoside}
-\kappa \le \mathrm{Ric} \le \kappa
\end{align}
for $\kappa \ge 0$, via the analysis of the path space.
Moreover he shows the compatibility between (\ref{twoside}) and an $RCD$-condition.
\begin{enumerate}
\item[\textbf{(Q5.1)}] Are there any relationships between our Ricci curvature bounds and Naber's one?
\end{enumerate}

Second, we discuss the first Betti number.
It is unknown that there exists a noncollapsed Gromov-Hausdorff convergent sequence with bounded Ricci curvature such that the Betti numbers does not converge that of the limit.
Moreover we do not know whether the first Betti number of every $X \in \overline{\mathcal{M}}$ in the ordinary sense is finite.
Note that as we mentioned in Section 1, Gigli proved that $h_1(X)$ coincide with his first Betti number, that is, the dimension of the first de Rham cohomology group defined in \cite{gigli} by him.
\begin{enumerate}
\item[\textbf{(Q5.2)}] For every $X \in \overline{\mathcal{M}}$, does $h_1(X)$ coincide with the first Betti number in the ordinary sense?
\item[\textbf{(Q5.3)}] In Theorem \ref{bettibetti}, does the continuity of $h_1$ with respect to the Gromov-Hausdorff topology always hold? Moreover is it true for a noncollapsed Gromov-Hausdorff convergent sequence under a uniform lower bound of Ricci curvature? 
\end{enumerate}

The following are on spectral convergence.
\begin{enumerate}
\item[\textbf{(Q5.4)}] Do the spectral convergence of $\Delta_{H, 1}$, $\Delta_{C, (r, s)}$ and $\Delta_{C, s}$ hold for a noncollapsed Gromov-Hausdorff convergent sequence under a uniform lower bound of Ricci curvature? 
\item[\textbf{(Q5.5)}] Does Theorem \ref{L2Hess} hold for a noncollapsed Gromov-Hausdorff convergent sequence under a lower Ricci curvature bound? 
\end{enumerate}
The question \textbf{(Q5.5)} is related to the question \textbf{(Q5.4)} because if the question \textbf{(Q5.5)} has a positive answer and the following ($\star$) holds, then \textbf{(Q5.4)} also has a positive answer by arguments similar to the proofs of Theorems \ref{spectr} and \ref{wll}.
\begin{enumerate}
\item[$(\star)$] Let $X_i$ be a sequence of $n$-dimensional closed Riemannian manifold with $\mathrm{Ric}_{X_i}\ge K$, and let $X$ be the noncollapsed Gromov-Hausdorff limit.
Let $\lambda_i$ be a bounded sequence in $\mathbf{R}$, let $T_i$ be a sequence of $\lambda_i$-eigen-one-forms (or $\lambda_i$-eigen-tensor fields, respectively) of $\Delta_{H, 1}$ (or $\Delta_{C, (r, s)}$, respectively) and let $T$ be the $L^2$-strong limit of them on $X$.
Then $T \in H^{1, 2}_C(T^*X)$ (or $T \in H^{1, 2}_C(T^r_sX)$, respectively).
\end{enumerate}

We end this section by giving a remark on the question \textbf{(Q5.5)}.
\textbf{(Q5.5)} has a positive answer if $(X_i, \upsilon_i)$ is isometric to $(X, \upsilon)$ for every $i$ because of Bochner's inequality (\cite[Corolalry 3.3.9]{gigli} or \cite[Theorem 1.4]{holp}).
It is worth pointing out that the noncollapsing assumption is essential.
We give an example.

Let $X_r$ be the spherical suspension of $\mathbf{S}^1(r)$.
Then it is easy to check that for every $0<r \le 1$, $X_r$ is the Gromov-Hausdorff limit of a sequence of $2$-dimensional closed Riemannian manifolds  $X_{r, i}$ with curvature bounded below by $1$.
In particular $(X_r, H^2/H^2(X_r))$ is a $(2, 1)$-Ricci limit space (see Remark \ref{boundei} for the definition of Ricci limit spaces).
By the compactness of measured Gromov-Hausdorff convergence in \cite{fu}, there exists a sequence of positive numbers $r_j \to 0$, a subsequence $i(j)$ and a Radon measure $\upsilon$ on $[0, \pi]$ such that
\[\left( Y_j, \frac{H^2}{H^2(Y_j)}\right) \stackrel{GH}{\to} ([0, \pi], \upsilon),\]
where $Y_j:=X_{r_j, i(j)}$.
In particular $ ([0, \pi], \upsilon)$ is also a $(2, 1)$-Ricci limit space.
Note that the limit measure $\upsilon$ does not coincide with the one-dimensional Hausdorff measure $H^1$ because it is easy to check that 
\[\lim_{t \to 0}\frac{\upsilon (B_t(0))}{t}=0.\]
Let $f_{j} \in C^{\infty}(Y_j)$ be a sequence of  $\lambda_1(Y_j)$-eigenfunctions with 
\[\frac{1}{H^2(Y_j)}\int_{Y_j}|f_{j}|^2dH^2=1.\]
Without loss of generality we can assume that the $L^2$-strong limit $f$ of $f_{j}$ on $X$.
Then we see that $f$ is a $2$-eigenfunction of the Dirichlet Laplacian $\Delta^{\upsilon}$ with respect to $\upsilon$, that $\Delta f_{j}$ ($=\lambda_1(Y_j)f_{j}$) $L^2$-converges strongly to $\Delta^{\upsilon} f$ ($=2f$) on $X$, and that $\mathrm{Hess}_{f_{j}}$ does not $L^2$-converge strongly to $\mathrm{Hess}_f$ on $X$.
See \cite[Remarks 4.32 and 4.33]{holp}.

\end{document}